\documentclass[12pt]{article}
\usepackage{amsfonts}
\usepackage{amsmath}
\usepackage{amssymb}
\usepackage{latexsym}
\usepackage{amscd}
\usepackage{amsthm}
\usepackage{bbm}
\usepackage{url}
\usepackage{hyperref}

\date{}

\newtheorem{theorem}{Theorem}[section]
\newtheorem*{theoremA*}{Theorem A}
\newtheorem*{theoremB*}{Theorem B}
\newtheorem*{theoremC*}{Theorem C}
\newtheorem*{theoremD*}{Theorem D}
\newtheorem{lemma}[theorem]{Lemma}
\newtheorem{cor}[theorem]{Corollary}

\newtheorem{prop}[theorem]{Proposition}

\newtheorem{claim}[theorem]{Claim}
\theoremstyle{definition}
\newtheorem{Remark}[theorem]{Remark}

\theoremstyle{plain}

\newcommand{\degreeset}{{\rm \bf Deg}}
\newcommand{\N}{\mathbb{N}}
\newcommand{\E}{\mathbb{E}}

\newcommand{\R}{\mathbb{R}}

\newcommand{\supp}{{\rm supp\,}}

\newcommand{\Event}{\mathcal{E}}

\def\nrangr{G}
\def\rangr{{\bf{}G}}
\def\ranmtx{{\bf{}M}}
\def\vofones{{\bf{}1}}
\def\inco{{\rm {\bf C}_{\nrangr}^{in}}}
\def\incopr{{\rm {\bf C}_{\nrangr'}^{in}}}

\def\innbr{{\mathcal N}_{\nrangr}^{in}}
\def\innbrpr{{\mathcal N}_{\nrangr'}^{in}}
\def\outnbr{{\mathcal N}_{\nrangr}^{out}}
\def\edg{{\bf E}_{\nrangr}}
\def\edgpr{{\bf E}_{\nrangr'}}
\def\epsnet{\mathcal{N}_{\varepsilon}}
\def\Prob{{\mathbb P}}
\def\InDeg{{\bf d}^{in}}
\def\OutDeg{{\bf d}^{out}}
\def\InI{{\cal{I}}^{in}}
\def\OutI{{\cal{I}}^{out}}
\def\DGraphSet{{\mathcal D}_n}
\def\UGraphSet{{\mathcal G}_{n}}
\def\MatrixSet{{\mathcal M}_n}
\def\SymMatrixSet{{\mathcal S}_n}
\def\Col{{\rm col}}
\def\Row{{\rm row}}
\def\pvector{{\mathcal P}}
\def\gfunction{{g}}
\def\relation{{\mathcal R}}

\title{The spectral gap of dense random regular graphs}

\author{Konstantin Tikhomirov and Pierre Youssef}

\begin{document}

\maketitle

\begin{abstract}
For any $\alpha\in (0,1)$ and any $n^{\alpha}\leq d\leq n/2$, we show that $\lambda(\rangr)\leq C_\alpha \sqrt{d}$ 
with probability at least $1-\frac{1}{n}$, where $\rangr$ is the uniform random $d$-regular graph on $n$ vertices,  
$\lambda(\rangr)$ denotes its second largest eigenvalue (in absolute value) and $C_\alpha$ is a constant depending only on $\alpha$. 
Combined with earlier results in this direction covering the case of sparse random graphs, 
this completely settles the problem of estimating the magnitude of $\lambda(\rangr)$, up to a multiplicative constant, for all values of $n$ and $d$, 
confirming a conjecture of Vu. 
The result is obtained as a consequence of an estimate for the second largest singular value of adjacency
matrices of random {\it directed} graphs with predefined degree sequences.
As the main technical tool, we prove a concentration inequality for arbitrary linear forms
on the space of matrices, where the probability measure is induced by the adjacency matrix of a random
directed graph with prescribed degree sequences. The proof
is a non-trivial application of the Freedman inequality for martingales,
combined with boots-trapping and tensorization arguments.
Our method bears considerable differences compared
to the approach used by
Broder, Frieze, Suen and Upfal (1999) who established the upper bound for $\lambda(\rangr)$
for $d=o(\sqrt{n})$, and to the argument of
Cook, Goldstein and Johnson (2015)
who derived a concentration inequality for linear forms
and estimated $\lambda(\rangr)$ in the range $d= O(n^{2/3})$
using size-biased couplings.
\end{abstract}



\section{Introduction}

Let $n$ be a natural number and let $d\leq n$. An {\it undirected $d$-regular} graph
$\nrangr$ with the vertex set $\{1,2,\ldots,n\}$ is a graph in which every vertex has exactly $d$ neighbors.
Spectral properties of random undirected $d$-regular graphs
have attracted considerable attention of researchers.
Regarding the empirical spectral distribution,
we refer, among others, to a classical result of McKay \cite{McKay}, as well as more recent papers \cite{DP,TVW,BKY}.
A new line of research deals with invertibility of adjacency matrices \cite{Cook PTRF,LLTTY}.
The seminal works of Alon and Milman \cite{AM} and Alon \cite{Alon}
established a connection between the magnitude of the second largest eigenvalue
of a regular graph with its expansion properties. The conjecture of Alon \cite{Alon} on the limit of
the spectral gap when the degree is fixed and the number of vertices tends to infinity, was resolved by Friedman \cite{Friedman}
(see \cite{BS,FKS,F91} for earlier results).
Friedman proved, in particular, that
$\lambda(\rangr)=2\sqrt{d-1}+o(1)$ with probability tending to one with $n\to\infty$, where
$\lambda(\rangr)$ is the second largest (in absolute value) eigenvalue of the undirected $d$-regular random graph
$\rangr$ on $n$ vertices, uniformly distributed on the set of all simple $d$-regular graphs
(see \cite{B} for an alternative proof of Friedman's theorem; see also \cite{puder} for a different approach producing a weaker bound).
A natural extension of Alon's question to the setting when $d$ grows with $n$ to infinity,
was considered in \cite{BFSU,DJPP,CGJ}. Namely, in \cite{BFSU} the authors showed that for
$d=o(\sqrt{n})$, one has $\lambda(\rangr)\leq C\sqrt{d}$ with probability tending to one with $n$
for some universal constant $C>0$. This result was extended to the range $d=O(n^{2/3})$ in \cite{CGJ}.
In \cite{DJPP}, the bound $\lambda(\rangr)\leq C\sqrt{d}$ w.h.p.\ was obtained for $\rangr$
distributed according to {\it the permutation model}, which we do not consider here.

In \cite{Vu}, Vu conjectured that $\lambda(\rangr)=(2+o(1))\sqrt{d-d^2/n}$ w.h.p.\ in the uniform model, when
$d\leq n/2$ and $d$ tends to infinity with $n$
(see also \cite[Conjectures~7.3, 7.4]{Vu ICM}). The ``isomorphic'' version of this question
was one of the motivations for our work. Apart from the previously mentioned connection with structural properties
of random graphs, this line of research
seems quite important in another aspect as well.
Random $d$-regular graphs supply a natural model of randomness for square matrices,
in which the matrix cannot be partitioned into independent disjoint blocks (say, rows or columns)
but the correlation between {\it very small} disjoint blocks is weak.
Techniques developed to deal with the adjacency matrices of these graphs may prove useful in other
problems within the random matrix theory.
In this respect, our intention was to develop, or rely on,
arguments which are flexible and admit various generalizations.

Given an $n\times n$ symmetric matrix $A$, we let
$\lambda_1(A)\geq\lambda_2(A)\geq\ldots\geq \lambda_n(A)$ be its
eigenvalues arranged in non-increasing order (counting multiplicities).
For an undirected graph $\nrangr$ on $n$ vertices, we define $\lambda_1(\nrangr),\ldots,\lambda_n(\nrangr)$
as the eigenvalues of its adjacency matrix.
\begin{theoremA*}
For every $\alpha\in(0,1)$ and $m\in\N$ there are $L=L(\alpha,m)>0$ and $n_0=n_0(\alpha,m)$ with the following property:
Let $n\geq n_0$, $n^\alpha\leq d\leq n/2$, and let $\rangr$ be a random graph uniformly distributed on
the set $\UGraphSet(d)$ of simple undirected $d$-regular graphs on $n$ vertices. Then
$$\max\bigl(|\lambda_2(\rangr)|,|\lambda_n(\rangr)|\bigr)\leq L\sqrt{d}$$
with probability at least $1-n^{-m}$.
\end{theoremA*}

Note that, combined with \cite{BFSU,CGJ}, our theorem gives
$\max\bigl(|\lambda_2(\rangr)|,|\lambda_n(\rangr)|\bigr)=O( \sqrt{d})$
w.h.p.\ for all $d\leq n/2$. Denote by $\ranmtx$ the adjacency matrix of $\rangr$. It is easy to see that (deterministically) 
$d$ is the largest eigenvalue of $\ranmtx$ with ${\bf 1}$ (vector of ones) as the corresponding eigenvector. Hence,  
from the Courant--Fischer formula, we obtain $\lambda(\rangr)=\Vert \ranmtx-  \frac{d}{n}{\bf 1}\,{\bf 1}^t\Vert=
\Vert \ranmtx-  \E\ranmtx\Vert$. Theorem~A thus implies 
that the spectral measure of $\frac{1}{\sqrt{d}} (\ranmtx-\E\ranmtx)$ is supported on an interval of constant
length with probability going to one with $n\to\infty$. 
We refer to \cite{BKY,BHKY} (and references therein) for recent advances concerning the limiting behavior of the spectral measure 
for random $d$-regular graphs in the uniform model. 
The proof of Theorem~A is obtained by a rather general
yet simple procedure which reduces the question to the non-symmetric (i.e.\ directed) setting,
which we are about to consider. 

\bigskip

A {\it directed $d$-regular} graph $\nrangr$ on $n$ vertices is a directed (labeled) graph in which every vertex has $d$
in-neighbors and $d$ out-neighbors. We allow directed graphs to have loops, but do not allow multiple edges
(edges connecting the same pair of vertices in opposite directions are distinct).
The corresponding set of graphs will be denoted by $\DGraphSet(d)$.
Note that the set of adjacency matrices for graphs in $\DGraphSet(d)$
is the set of all $0\text{-}1$-matrices with the sum of elements in each row and column equal to $d$.
Note also that there is
a natural bijection from $\DGraphSet(d)$ onto the set of {\it bipartite}
$d$-regular simple undirected graphs on $2n$ vertices. 
Given an $n\times n$ matrix $A$, we let
$s_1(A)\geq s_2(A)\geq\ldots\geq s_n(A)$ be its
singular values arranged in non-increasing order (counting multiplicities).
For a directed graph $\nrangr$ on $n$ vertices, we define $s_1(\nrangr),\ldots,s_n(\nrangr)$
as the singular values of its adjacency matrix.

\begin{theoremB*}
For every $\alpha\in(0,1)$ and $m\in\N$ there are
$L=L(\alpha,m)>0$ and $n_0=n_0(\alpha,m)\in\N$ with the following property:
Let $n\geq n_0$,
and let $n^\alpha\leq d\leq n/2$.
Further, 
let $\rangr$ be a random directed $d$-regular graph uniformly distributed on
$\DGraphSet(d)$. Then
$$s_2(\rangr)\leq L\sqrt{d}$$
with probability at least $1-n^{-m}$. 
Consequently, if $\widetilde \rangr$ is a random undirected graph uniformly distributed on the set of all bipartite simple
$d$-regular graphs on $2n$ vertices then
$$\lambda_2(\widetilde\rangr)\leq L\sqrt{d}$$
with probability at least $1-n^{-m}$.
\end{theoremB*}

\bigskip

Theorem~B above is stated for reader's convenience.
In fact, we prove a more general statement which deals with random graphs
with {\it predefined degree sequences}.
With every directed graph $\nrangr$ on $\{1,2,\ldots,n\}$, we associate two degree sequences:
{\it the in-degree sequence $\InDeg(\nrangr)=(\InDeg_1,\InDeg_2,\ldots,\InDeg_n)$}, with $\InDeg_i$
equal to the number of in-neighbors of vertex $i$, and
{\it the out-degree sequence $\OutDeg(\nrangr)=(\OutDeg_1,\OutDeg_2,\ldots,\OutDeg_n)$},
where $\OutDeg_i$ is the number of out-neighbors of $i$ ($i\leq n$).
Conversely, given two integer vectors $\InDeg, \OutDeg\in\R^n$, we will denote by $\DGraphSet(\InDeg, \OutDeg)$
the set of all directed graphs on $n$ vertices with the in- and out-degree sequence $\InDeg$ and $\OutDeg$,
respectively. Again, we allow the graphs to have loops but do not allow multiple edges.

Let us introduce the following two Orlicz norms in $\R^n$:
\begin{align}\label{eq: psi definition}
\|x\|_{\psi,n}&:=\inf\Big\{\lambda>0: \frac{1}{en}\sum_{i=1}^n e^{|x_i|/\lambda}\leq 1\Big\},\quad &x=(x_1,\ldots,x_n)\in\R^n;\\
\label{eq: log norm definition}
\|x\|_{\log,n}&:=\inf\Big\{\lambda>0:
\frac{1}{n}\sum_{i=1}^n \frac{|x_i|}{\lambda}\ln_+\Big(\frac{|x_i|}{\lambda}\Big)\leq 1\Big\},\quad &x=(x_1,\ldots,x_n)\in\R^n.
\end{align}
Here, $\ln_+(t):=\max(0,\ln t)$ ($t\geq 0$).
One can verify that the space $(\R^n,\|\cdot\|_{\log,n})$ is isomorphic (with an absolute constant) to the {\it dual} space
for $(\R^n,\|\cdot\|_{\psi,n})$. More properties of these norms will be considered later.
Now, let us state the spectral gap theorem for directed graphs in full generality:
\begin{theoremC*}
For every $\alpha\in(0,1)$, $m\in\N$ and $K>0$
there are $L=L(\alpha,m,K)>0$ and $n_0=n_0(\alpha,m,K)\in\N$ with the following property:
Let $n\geq n_0$,
and let $\InDeg, \OutDeg$ be two degree sequences
such that for some integer $n^\alpha\leq d\leq 0.501n$
we have
$$\max\big(\big\|\big(\InDeg_i-d\big)_{i=1}^n\big\|_{\psi,n},\,\big\|\big(\OutDeg_i-d\big)_{i=1}^n\big\|_{\psi,n}\big)\leq K\sqrt{d}.$$
Assume that $\DGraphSet(\InDeg,\OutDeg)$ is non-empty,
let $\rangr$ be a random directed graph uniformly distributed on
$\DGraphSet(\InDeg, \OutDeg)$. Then
$$s_2(\rangr)\leq L\sqrt{d}$$
with probability at least $1-n^{-m}$.
\end{theoremC*}
The condition on the degree sequences in the theorem can be viewed as a concentration inequality for $\InDeg_i-d$ and
$\OutDeg_i-d$, with respect to the ``uniform'' choice of $i$ in $[n]$.
In particular, if $\|(\InDeg_i-d\big)_{i=1}^n\|_{\infty},\,\|(\OutDeg_i-d)_{i=1}^n\|_{\infty}\leq K\sqrt{d}$
then the degree sequences satisfy the assumptions of the theorem.

Theorem~C is the main theorem in this paper, and Theorem~A (and, of course, B) is obtained
as its consequence. In note \cite{TY_short}, 
we proved a rather general comparison theorem for {\it jointly exchangeable}
matrices which, in particular, 
allows us to estimate the spectral gap of random undirected 
$d$-regular graphs in terms of the second singular value 
of directed random graphs with predefined degree sequences. 
Let us briefly describe the idea of the reduction scheme. Assume that $\rangr$
is uniformly distributed on $\UGraphSet(d)$ and let $\ranmtx$ be its adjacency matrix.
Then the results of \cite{TY_short} assert that with high probability
$s_2(\ranmtx)=\max(|\lambda_2(\rangr)|,|\lambda_n(\rangr)|)$
can be bounded from above by a multiple of
the second largest singular value of the
$n/2\times n/2$ submatrix of $\ranmtx$ located in its top right corner.
In turn, it can be verified that the distribution of this submatrix
is directly related to the distribution of the adjacency matrix
of a random {\it directed} graph on $n/2$ vertices with in- and out-degree
sequences ``concentrated'' around $d/2$.
We will cover this procedure in more detail in Section~\ref{s: undirected-case} and show how Theorem~A follows from Theorem~C.

\bigskip

In the course of proving Theorem~C, we obtain certain relations for random graphs with predefined degree sequences
which may be of separate interest. The rest of the introduction is devoted to
discussing these developments and, in parallel, provides an outline of the proof of Theorem~C.
Given an $n\times n$ matrix $M$, we denote the Hilbert--Schmidt norm of $M$ by $\|M\|_{HS}$.
Additionally, we will write $\|M\|_\infty$ for the maximum norm (defined as the absolute value of the
largest matrix entry).
The set of adjacency matrices of graphs in $\DGraphSet(\InDeg, \OutDeg)$
will be denoted by $\MatrixSet(\InDeg, \OutDeg)$. Obviously, $\MatrixSet(\InDeg, \OutDeg)$
coincides with the set of all $0\text{-}1$-matrices $M$ with
$|\supp\Col_i(M)|=\InDeg_i$, $|\supp\Row_i(M)|=\OutDeg_i$ for all $i\leq n$.

The proof of Theorem~C is composed of two major blocks. In the first block,
we derive a concentration inequality for linear functionals of the form
$\sum_{i,j=1}^n \ranmtx_{ij}Q_{ij}$, where $\ranmtx$ is a random matrix uniformly
distributed on $\MatrixSet(\InDeg, \OutDeg)$, and $Q$ is any fixed $n\times n$ matrix.
In the second block,
we use the concentration inequality to establish certain discrepancy properties of the random graph
associated with $\ranmtx$. Then, we apply a well known argument of Kahn and Szemer\'edi \cite{FKS}
in which the discrepancy property, together with certain covering arguments,
yields a bound on the matrix norm.

{\bf The first block.} Our concentration inequality for linear forms involves conditioning on a special event, 
having a probability close to one, on the space of matrices $\MatrixSet(\InDeg, \OutDeg)$.
Let us momentarily postpone the definition of the event (which is rather technical) and state the inequality first.
Define a function $H(t)$ on the positive semi-axis as
\begin{equation}\label{eq: H definition}
H(t):=(1+t)\ln (1+t) -t.
\end{equation}

\begin{theoremD*}
For every $\alpha\in(0,1)$, $m\in\N$ and $K>0$
there are $\gamma=\gamma(\alpha,m,K), L=L(\alpha,m,K)>0$ and $n_0=n_0(\alpha,m,K)\in\N$ with the following property:
Let $n\geq n_0$,
and let $\InDeg, \OutDeg$ be two degree sequences
such that for some integer $n^\alpha\leq d\leq 0.501n$
we have
$$\max\big(\big\|\big(\InDeg_i-d\big)_{i=1}^n\big\|_{\psi,n},\,\big\|\big(\OutDeg_i-d\big)_{i=1}^n\big\|_{\psi,n}\big)\leq K\sqrt{d}.$$
Assume further that $\MatrixSet(\InDeg,\OutDeg)$ is non-empty,
and let $\ranmtx$ be uniformly distributed on $\MatrixSet(\InDeg,\OutDeg)$.
Then for any fixed $n\times n$ matrix $Q$ and any $t\geq CL\sqrt{d}\|Q\|_{HS}$ we have
\begin{align*}
\Prob\Big\{\Big|\sum_{i,j=1}^{n}
{\ranmtx}_{ij}Q_{ij}-\frac{d}{n}\sum_{i,j=1}^n Q_{ij}\Big|
>t\,\mid\,{\ranmtx}\in\Event_{\pvector}(L)
\Big\}
\leq 2\exp\left(-\frac{d\,\|Q\|_{HS}^2}{n\,\|Q\|_\infty^2}
\, H\left(\frac{\gamma tn\|Q\|_\infty}{d\|Q\|_{HS}^2}\right)\right).
\end{align*}
Here, $C>0$ is a universal constant and $\Event_{\pvector}(L)$ is a subset of $\MatrixSet(\InDeg,\OutDeg)$
which is determined by the value of $L$, and satisfies $\Prob(\Event_\pvector(L))\geq 1-n^{-m}$.
\end{theoremD*}

The function $H$ in the above deviation bound is quite natural in this context.
It implicitly appears in the classical inequality of Bennett for sums of independent variables
(see, \cite[Formula~8b]{Bennett}), and later in the well known paper of Freedman
\cite{Freedman} where he extends Bennett's inequality to martingales.
In fact, our proof of Theorem~D uses the Freedman inequality (more precisely, 
Freedman's bound for the moment generating function) as a fundamental element.
Note that we require $t$ to be greater (by the order of magnitude) than
$\sqrt{d}\,\|Q\|_{HS}$, which makes the above statement a large deviation inequality.
The restriction on $t$ takes its roots into the way we obtain Theorem~D from concentration
inequalities for {\it individual} matrix rows. The {\it tensorization} procedure
involves estimating the differences between conditional and unconditional expectations
of rows, and we apply a rather crude bound by summing up absolute values of the ``errors''
for individual rows. In fact, the lower bound $CL\sqrt{d}\,\|Q\|_{HS}$ for $t$
can be replaced with a smaller quantity $C'L\sqrt{d}\sum_{i=1}^n \|\Row_i(Q)\|_{\log,n}$,
provided that we choose a different ``point of concentration'' than $\frac{d}{n}\sum_{i,j=1}^n Q_{ij}$;
we prefer to avoid discussing these purely technical aspects in the introduction.

A concentration inequality very similar to the one from Theorem~D, was established in a recent
paper of Cook, Goldstein and Johnson \cite{CGJ} which strongly influenced our work.
The Bennett-type inequality from \cite{CGJ}, formulated for adjacency matrices of undirected $d$-regular graphs,
also involves a restriction on the parameter $t$, which, however, exhibits a completely different behavior
compared to the lower bound $\sqrt{d}\,\|Q\|_{HS}$ in our work. In particular,
the concentration inequality in \cite{CGJ} is not strong enough in the range $d\gg n^{2/3}$ to
yield the correct order of the second largest eigenvalue. For the {\it permutation model},
a Bernstein-type concentration inequality was obtained in \cite{DJPP} by constructing a single martingale
sequence for the whole matrix and applying Freedman's inequality. We will discuss in detail in Section~\ref{s: tensorization}
why a direct use of the same approach is problematic in our setting.

Theorem~D, the way it is stated, is already sufficient to complete the proof of Theorem~C,
without any knowledge of the structure of the event $\Event_\pvector(L)$.
However, defining this event explicitly should give more insight and enable us to draw a comprehensive
picture.
Let $G$ be a digraph on $n$ vertices
with degree sequences $\InDeg$, $\OutDeg$, and let $M=(M_{ij})$ be the adjacency matrix of $G$.
Further, let $I$ be a subset of $\{1,2,\dots,n\}$ (possibly, empty).
We define quantities $p_j^{col}(I,M)$ and $p_j^{row}(I,M)$ ($j\leq n$) as
\begin{align*}
p_j^{col}(I,M)&:=\InDeg_j-|\{q\in I:\,M_{qj}=1\}|=|\{q\in I^c:\,M_{qj}=1\}|;\\
p_j^{row}(I,M)&:=\OutDeg_j-|\{q\in I:\,M_{jq}=1\}|=|\{q\in I^c:\,M_{jq}=1\}|.
\end{align*}
Further, let us define $n$-dimensional vectors $\pvector^{col}(I,M)
=(\pvector^{col}_1(I,M), \ldots,\pvector^{col}_n(I,M))$ and $\pvector^{row}(I,M)
=(\pvector^{row}_1(I,M),\ldots,\pvector^{row}_n(I,M))$ as
\begin{equation*}
\begin{split}
\pvector^{col}_j(I,M)&:=\sum_{\ell=1}^n |p_j^{col}(I,M)-p_\ell^{col}(I,M)|\\
\pvector^{row}_j(I,M)&:=\sum_{\ell=1}^n |p_j^{row}(I,M)-p_\ell^{row}(I,M)|
\end{split}\quad\quad\quad j\leq n.
\end{equation*}

Conceptually, the vectors $\pvector^{col}(I,M),\pvector^{row}(I,M)$ can be thought of as
a measure of ``disproportion'' in the locations of $1$'s across the matrix $M$.
Given any non-empty subset $I\subset\{1,2,\dots,n\}$, let $M^I$ be the $I\times n$-submatrix of $M$.
Then for every $j\leq n$, $\pvector_j^{col}(I,M)$ is just the sum of differences of $\ell_1^n$-norms
of the $j$-th column and every other column of $M^I$:
$$\pvector_j^{col}(I,M)=\sum_{\ell=1}^n \big|\|\Col_j(M^I)\|_1-\|\Col_\ell(M^I)\|_{1}\big|.$$
The event $\Event_\pvector(L)$ employed in Theorem~D, controls the magnitude of those vectors:
for every $L>0$ we define the event as
\begin{equation}\label{eq: event pvector definition}
\begin{split}
\Event_\pvector(L):=\Big\{&M\in\MatrixSet(\InDeg,\OutDeg):\;
\|\pvector^{row}(I,M)\|_{\psi,n},\|\pvector^{col}(I,M)\|_{\psi,n}\leq  L  n\sqrt{d}\;\mbox{ for}\\
&\mbox{any interval subset $I\subset\{1,2,\dots,n\}$ of cardinality at most $0.001 n$}\Big\},\end{split}
\end{equation}
where $\|\cdot\|_{\psi,n}$ is given by \eqref{eq: psi definition}.
Note that the subsets $I$ in the definition are assumed to be {\it interval subsets},
which gives importance to the way we enumerate the vertices.
It is not difficult to see that if the definition involved {\it every} subset $I$ with $|I|\leq 0.001 n$ then the probability
of the event would be just zero as one can always find two vertices with largely non-overlapping
sets of in-neighbors. 

Loosely speaking, the condition secured by the event $\Event_\pvector(L)$ is a skeleton for our matrix:
it indicates that $1$'s are spread throughout the matrix more or less evenly.
Assuming this property (i.e.\ conditioning on the event), we can establish stronger ``rules'' for the distribution
of the non-zero elements and, in particular, obtain Theorem~D.
This can be viewed as a realization of the boots-trapping strategy.

From the technical perspective, the proof of Theorem~D requires many preparatory statements
and is quite long. Our exposition is largely self-contained; probably the only essential ``exterior''
statement which we employ in the first part of the paper is Freedman's inequality for martingales,
which is given (together with some corollaries) in Sub-section~\ref{s: freedman}.
It is followed by the ``graph'' Sub-section~\ref{sec-graph} where we state
and prove a rough bound on the number of common in-neighbors of two vertices of
a random graph using a standard argument involving simple switchings and multimaps (relations).
Section~\ref{s: row concentr} is the core of the paper. There, we apply the Freedman inequality
and derive deviation bounds for individual rows of our random adjacency matrix.
The first sub-section contains a series of lemmas dealing with a fixed
row coordinate (and conditioned on the upper rows and all previous
coordinates within this fixed row) and provides a foundation for our analysis.
Sub-section~\ref{sec-row-concentration} integrates the information for the individual matrix entries
and, after resolving some technical issues, culminates in Theorem~\ref{th-concentration-row}
which is the main statement of Section~\ref{s: row concentr}.
Finally, we apply a tensorization procedure in Section~\ref{s: tensorization}
and prove (a somewhat technical version of) Theorem~D.

\medskip

{\bf The second block.} 
Equipped with the concentration inequality given by Theorem~D,
we follow the Kahn--Szemer\'edi argument \cite{FKS} to prove Theorem~C. 
For simplicity, let us describe the procedure for the uniform model on $\DGraphSet(d)$
and disregard conditioning on the event $\Event_\pvector(L)$ in Theorem~D.
Denoting by $\ranmtx$ the adjacency matrix of a random $d$-regular graph uniformly
distributed on $\DGraphSet(d)$,
it is easy to see that its largest singular value is equal $d$ (deterministically), and the corresponding normalized
singular vector is $(1/\sqrt{n},1/\sqrt{n},\ldots,1/\sqrt{n})=\frac{1}{\sqrt{n}}\,\vofones$.
By the Courant--Fischer formula and the singular value decomposition, we have
$$
s_2(\ranmtx)
=\big\Vert \ranmtx-\frac{d}{n}\, \vofones\cdot\vofones^t\big\Vert_{2\to 2}
=\sup_{\substack{x\in  \vofones^\perp\cap S^{n-1},\\ y\in S^{n-1}} } \langle \ranmtx x,y\rangle.
$$
A natural approach to bounding the supremum on the right hand side would be to apply
the standard covering argument, which plays a key role in Asymptotic Geometric Analysis.
The argument
consists in showing first that
$\langle \ranmtx x,y\rangle$ is bounded by certain threshold value (in this case, $O(\sqrt{d})$) with high probability
for any pair of admissible $x,y$. 
Once this is done, a quite general approximation scheme allows to replace the supremum over 
$\vofones^\perp\cap S^{n-1}\times S^{n-1}$ by the supremum over a finite discrete subset (a net).
From the probabilistic viewpoint, we pay the price by taking the union bound over the net
(which can be chosen to have cardinality exponential in dimension)
to obtain an estimate for the entire set.
In order for such a procedure to work,
we need a concentration inequality for $\langle \ranmtx x,y\rangle$ (for fixed $x,y$) which would
``survive'' multiplication by the cardinality of the net.
By Theorem~D (applied to the matrix $Q=yx^t$ for any fixed $(x,y)\in \vofones^\perp\cap S^{n-1}\times S^{n-1}$), 
we have
$$
\Prob\big\{\vert \langle \ranmtx x,y\rangle\vert \gg\sqrt{d}\big\}\ll \exp\left(-\frac{d}{n\,\|x\|_\infty^2\,\|y\|_\infty^2 }
\, H\left(\frac{n\, \|x\|_\infty\, \|y\|_\infty}{\sqrt{d}}\right)\right).
$$
However, the expression on the right hand side is an increasing function of $\|x\|_\infty\, \|y\|_\infty$,
and becomes larger than $C^{-n}$ when $\|x\|_\infty\, \|y\|_\infty\gg \sqrt{d}/n$.
Hence, the union bound in the above description can work only for $x,y$ having small $\|\cdot\|_\infty$-norms.
A key idea in the argument by Kahn and Szemer\'edi, which distinguishes it from the standard covering procedure,
is to split the quadratic form associated with $\langle \ranmtx x,y\rangle$
into ``flat'' and ``spiky'' parts:
\begin{equation}\label{eq: intro-flat-spike}
\langle \ranmtx x,y\rangle
= \sum_{\substack{(i,j)\in [n]\times[n]:\\\vert x_jy_i\vert \leq \sqrt{d}/n}} y_i\ranmtx_{ij}x_j 
+  \sum_{\substack{(i,j)\in[n]\times[n]:\\\vert x_jy_i\vert > \sqrt{d}/n}} y_i\ranmtx_{ij}x_j. 
\end{equation}

Let us note that a somewhat similar decomposition of the sphere into ``flat''
and ``spiky'' vectors was used in \cite{nicole-sasha} 
and \cite{RV} to bound the smallest singular value of certain random matrices. 
The first term in \eqref{eq: intro-flat-spike} can be dealt with by directly using the concentration inequality from Theorem~D
(plus standard covering).
On the other hand, the second summand needs a more delicate handling.
Kahn and Szemer\'edi proposed a way to relate the quantity to discrepancy properties of the underlying graph,
more precisely, to deviations of the edge count between subsets of the vertices from its mean value.
To illustrate the connection, let $a,b$ be any positive numbers with $ab\gg \sqrt{d}/n$ and let
$S:=\{i\leq n:\, \vert y_i\vert \approx b\}$ and $T:=\{j\leq n:\, \vert x_j\vert \approx a\}$.
Then
$$\sum_{\substack{(i,j)\in[n]\times[n]:\\ |x_j|\approx a,|y_i|\approx b}} y_i\ranmtx_{ij}x_j=O\big(ab\, |\edg(S, T)|\big),$$
where $|\edg(S,T)|$ is the number of edges of graph $\rangr$ corresponding to $\ranmtx$,
starting in $S$ and ending in $T$.
In the actual proof, this simplified illustration should be replaced by 
a careful partitioning of vectors $x$ and $y$ into ``almost constant'' blocks.  
We refer to Section~\ref{s: kahn-szemeredi} for a rigorous exposition
of the argument allowing to complete the proof of Theorem~C. 
Once Theorem~C is proved, we apply it, together with the ``de-symmetrization'' result of \cite{TY_short},
to prove Theorem~A. This is accomplished in Section~\ref{s: undirected-case}.

\section{Notation and Preliminaries}\label{s: global notation}

Everywhere in the text, we assume that $n$ is a large enough natural number.
For a finite set $I$, by $|I|$ we denote its cardinality.
For any positive integer $m$, the set $\{1,2,\ldots,m\}$ will be denoted by $[m]$.
If $I\subset[n]$ then,
unless explicitly specified otherwise, the set $I^c$ is the complement of $I$ in $[n]$.
For a real number $a$, $\lceil a\rceil$ is the smallest integer greater or equal to $a$, and
$\lfloor a\rfloor$ is the largest integer not exceeding $a$.
A vector $y\in\R^n$ is called {\it $r$-sparse} for some $r\geq 0$ if the support $\supp y$
has cardinality at most $r$.
By $\langle \cdot,\cdot\rangle$
we denote the standard inner product in $\R^n$, by $\|\cdot\|$ --- the standard Euclidean norm in $\R^n$,
and by $\{e_1,e_2,\ldots,e_n\}$ --- the canonical basis vectors.
For every $1\leq p<\infty$, the $\|\cdot\|_p$-norm in $\R^n$ is defined by
$$\|(x_1,x_2,\dots,x_n)\|_p:=\bigg(\sum_{i=1}^\infty |x_i|^p\bigg)^{1/p},$$
and the canonical maximal norm is
$$\|(x_1,x_2,\dots,x_n)\|_\infty:=\max\limits_{i\leq n}|x_i|.$$
Universal constants are denoted by $C,c,c'$, etc.
In some situations we will add a numerical subscript to the name of a constant
to relate it to a particular numbered statement. For example, $C_{\ref{l: elementary psi estimate}}$
is a constant from Lemma~\ref{l: elementary psi estimate}.

Let $M$ be a fixed $n\times n$ matrix.
The $(i,j)$-th entry of $M$ is denoted by $M_{ij}$.
Further, we will denote rows and columns by $\Row_1(M),\ldots,\Row_n(M)$ and $\Col_1(M),\ldots,\Col_n(M)$.
We denote the Hilbert--Schmidt norm of $M$ by $\|M\|_{HS}$.
Additionally, we write $\|M\|_\infty$ for the maximum norm (defined as the absolute value of the
largest matrix entry) and $\|M\|_{2\to 2}$ for its spectral norm.

\medskip

Let $\InDeg,\OutDeg$ be two degree sequences.
Everywhere in this paper, we assume that for an integer $d$ we have
\begin{equation}\label{eq: degree condition}
(1-c_0) d\leq\InDeg_i,\OutDeg_i\leq d\quad\mbox{ for all }i\leq n,\quad\mbox{ where }
c_0:=0.001\mbox{ and }d\leq (1/2+c_0)n.
\end{equation}

Recall that, given two degree sequences $\InDeg, \OutDeg$,
the set of adjacency matrices of graphs in $\DGraphSet(\InDeg, \OutDeg)$
is denoted by $\MatrixSet(\InDeg, \OutDeg)$.
We will write $\SymMatrixSet(d)$
for the set of adjacency matrices of undirected simple $d$-regular graphs on $[n]$.
Each of the sets $\DGraphSet(\InDeg, \OutDeg)$, $\MatrixSet(\InDeg, \OutDeg)$, $\UGraphSet(d)$, 
$\SymMatrixSet(d)$ can be turned into a probability space by defining the normalized counting measure.
We will use the same notation $\Prob$ for the measure in each of the four cases. The actual probability space
will always be clear from the context.

\medskip

The expectation of a random variable $\xi$ is denoted by $\E\xi$. We will
use vertical bar notation for conditional expectation and conditional probability. For example, the expectation of $\xi$
conditioned on an event $\Event$, will be written as $\E[\xi\,|\,\Event]$, and
the conditional expectation given a $\sigma$-sub-algebra $\mathcal F$ --- as
$\E[\xi\,|\,\mathcal F]$.

\bigskip

Let $A$, $B$ be sets, and $R\subset A\times B$ be a relation.
Given $a\in A$ and $b\in B$, the image of $a$ and preimage of $b$ are defined by
$$
R(a): = \{ y \in B \, : \, (a,y)  \in  R\} \quad \mbox{ and } \quad
R^{-1}(b): = \{ x \in A \, : \, (x,b)  \in  R\}.
$$
We also set $R(A):=\cup _{a\in A} R(a)$.
Further in the text, we will define relations between sets in order to estimate their cardinality,
using the following elementary claim (see \cite{LLTTY} for a proof):
\begin{claim} \label{multi-al}
Let $s, t >0$.
Let $R$ be a relation between two finite sets $A$ and $B$ such that for
every $a\in A$ and every $b\in B$ one has $|R(a)|\geq s$ and $|R^{-1}(b)|\leq t$.
Then $s |A|\leq  t |B|$.
\end{claim}

\subsection{Orlicz norms}\label{s: Orlicz norms}

In the Introduction, we defined two Orlicz norms $\|\cdot\|_{\psi,n}$ and $\|\cdot\|_{\log,n}$ in $\R^n$.
Let us state some of their elementary properties (see \cite{RR}
for extensive information on Orlicz functions and Orlicz spaces).
First, it can be easily checked that
\begin{equation}\label{eq: psi infinity}
\|x\|_{\psi,n}\leq \|x\|_\infty\leq \ln(en)\,\|x\|_{\psi,n}\quad\quad\mbox{for all }x\in\R^n.
\end{equation}
Similarly, we have
\begin{equation}\label{eq: log 1}
\frac{n}{\ln n}\|x\|_{\log,n}\leq \|x\|_{1}\leq en\|x\|_{\log,n}\quad\quad\mbox{for all }x\in\R^n.
\end{equation}
\begin{lemma}\label{l: elementary psi estimate}
For any vector $y\in\R^n$ with $m:=|\supp y|\leq n$ we have
$$\|y\|\leq C_{\ref{l: elementary psi estimate}}\sqrt{m}\|y\|_{\psi,n}\ln\frac{2n}{m},$$
where $C_{\ref{l: elementary psi estimate}}>0$ is a universal constant.
\end{lemma}
\begin{proof}
Without loss of generality, $z:=\|y\|_{\psi,n}=\sqrt{n/m}$.
The convex conjugate of the exponential function is $t\ln(t)-t$ ($t>0$).
Hence, by Fenchel's inequality, for any $i\in\supp y$ we have
\begin{align*}
{y_i}^2&\leq e^{|y_i|/z}+z|y_i|\ln(z|y_i|)-z|y_i|\\
&\leq e^{|y_i|/z}+z|y_i|\ln(2z^2)+z|y_i|\ln\Big(\frac{|y_i|}{2z}\Big)\\
&\leq e^{|y_i|/z}+z|y_i|\ln(2z^2)+\frac{1}{2}{y_i}^2.
\end{align*}
Summing over all $i\in\supp y$, we get
$$\|y\|^2\leq 2en+2z\ln(2z^2)\sum_{i\in\supp y}|y_i|\leq 2en+2z\ln (2z^2)\sqrt{m}\|y\|.$$
Plugging in the definition of $z$ and solving the above inequality, we get
$$\|y\|\leq C\sqrt{n}\ln\frac{2n}{m}$$
for some universal constant $C>0$. The result follows.
\end{proof}

By a duality argument, we also have the following.
\begin{lemma}\label{l: elementary log estimate}
For any vector $y\in\R^n$ with $m:=|\supp y|\leq n$ we have
$$n\|y\|_{\log,n}\leq C_{\ref{l: elementary log estimate}}\|y\|\sqrt{m}\ln\frac{2n}{m}
\leq C_{\ref{l: elementary log estimate}}\|y\|_1\ln\frac{2n}{m},$$
where $C_{\ref{l: elementary log estimate}}>0$ is a universal constant.
\end{lemma}
Finally, given a vector $x$ with $\|x\|_{\psi,n}=1$, we can bound the number of coordinates $x$ of any given magnitude:
\begin{lemma}\label{l: psi lower bound}
Let $x\in\R^n$ with $\|x\|_{\psi,n}=1$. Then there is a natural number $k\leq 2\ln(en)$ such that
$$\big|\big\{i\leq n:\,|x_i|\geq k/2\big\}\big|\geq n (2e)^{-k}.$$
\end{lemma}
\begin{proof}
By the definition of the norm $\|\cdot\|_{\psi,n}$, we have
$$\sum_{i=1}^n e^{|x_i|}=en,$$
whence
$$\sum\limits_{\substack{i\leq n:\\|x_i|\geq 1/2}}e^{|x_i|}\geq n.$$
Thus,
$$\sum_{k=1}^{\infty} \big|\big\{i\leq n:\,|x_i|\geq k/2\big\}\big|e^{(k+1)/2}\geq \sum_{k=1}^\infty 2^{-k}n.$$
It remains to note that, in view of \eqref{eq: psi infinity}, we have $\big\{i\leq n:\,|x_i|\geq k/2\big\}=\emptyset$
for all $k>2\ln(en)$ and that $2^{-k}e^{-(k+1)/2}\leq (2e)^{-k}$ for all $k$.
\end{proof}

\subsection{Freedman's inequality}\label{s: freedman}

In this sub-section, we recall the classical concentration inequality for martingales
due to Freedman, and provide several auxiliary statements which we will apply later
in Section~\ref{s: tensorization}.
Define
\begin{equation}\label{eq: e definition}
\gfunction(t):=e^t-t-1,\quad t>0.
\end{equation}
In \cite{Freedman}, Freedman proved the following bound for the moment-generating function
which will serve as a fundamental block of this paper:
\begin{theorem}[Freedman's inequality]\label{th-freedman-moment}
Let $m\in\N$, let $(X_i)_{i\leq m}$ be a martingale with respect to a filtration $(\mathcal{F}_i)_{i\leq m}$, and 
let
$$d_i:=X_i-X_{i-1},\;\;i\leq m,$$
be the corresponding difference sequence. Assume that $|d_i|\leq M$ a.s.\ for some $M>0$ and 
$\sum_{i=1}^m\E({d_i}^2\,|\,\mathcal F_{i-1})\leq \sigma^2$ a.s.\ for some $\sigma>0$. Then for any $\lambda>0$, we have 
$$
\E e^{\lambda (X_m-X_0)}\leq \exp\left(\frac{\sigma^2}{M^2}\,\gfunction(\lambda M)\right).
$$
\end{theorem}
As a consequence of the above relation, Freedman derived the inequality
\begin{equation}\label{eq: Freedman Bennett}
\Prob\big\{X_m-X_0\geq t\big\}
\leq \exp\bigg(-\frac{\sigma^2}{M^2}H\Big(\frac{Mt}{\sigma^2}\Big)\bigg),\quad t>0,
\end{equation}
where $H$ is defined by \eqref{eq: H definition}.
It is easy to check that
\begin{equation}\label{eq-min-h}
H(t)\geq \frac{t^2}{2(1+t/3)}\quad\mbox{ for any }t\geq 0,
\end{equation}
whence, with the above notation,
\begin{equation}\label{eq: Freedman Bernstein}
\Prob\big\{X_m-X_0\geq t\big\}
\leq \exp\bigg(-\frac{t^2}{2\sigma^2+2Mt/3}\bigg),\quad t>0.
\end{equation}
In the special case when
the martingale consists of partial sums of a series of i.i.d.\ centered random variables,
i.e.\ $d_i$ ($i\leq m$) are i.i.d., \eqref{eq: Freedman Bennett}
was obtained by Bennett \cite{Bennett} and \eqref{eq: Freedman Bernstein} derived by Bernstein \cite{Bernstein}.
Returning to arbitrary martingale sequences, the estimate \eqref{eq: Freedman Bennett}
is often referred to as {\it the Freedman inequality}. However, in our setting it is crucial to have the stronger
relation provided by Theorem~\ref{th-freedman-moment},
as it will allow us to {\it tensorize} concentration inequalities obtained for individual rows of the matrix.

\begin{lemma}\label{lem-freedman-tensorization}
Let $m\in\N$ and let $\xi_1,\xi_2,\ldots,\xi_m$ be random variables.
Further, assume that $f_i(\lambda):\R_+\to\R_+$ are functions such that
$$\E [e^{\lambda \xi_i}\mid \xi_1,\ldots,\xi_{i-1}]\leq f_i(\lambda)$$
for any $\lambda>0$ and $i\leq m$. Then for any subset $T\subset[m]$ we have
$$\E e^{\lambda \sum_{i\in T}\xi_i}\leq \prod_{i\in T}f_i(\lambda).$$
\end{lemma}
\begin{proof}
Without loss of generality, take $T=[m]$.
Note that 
$$
\E\, e^{\lambda \sum_{i=1}^m\xi_i} = 
\E\left[ \E [e^{\lambda \sum_{i=1}^m\xi_i} \mid \xi_1,\ldots,\xi_{m-1}]\right]
=\E \left[ e^{\lambda\sum_{i=1}^{m-1} \xi_i}\,  \E [e^{\lambda \xi_m}\mid \xi_1,\ldots,\xi_{m-1}]\right].
$$
Hence, by the assumption on $f_m$, we get
$$
\E\, e^{\lambda \sum_{i=1}^m\xi_i}
\leq f_m(\lambda)\,  \E\, e^{\lambda\sum_{i=1}^{m-1} \xi_i}.
$$
Iterating this procedure, we obtain
$$
\E\, e^{\lambda \sum_{i=1}^m\xi_i}\leq \prod_{i=1}^m f_i(\lambda).
$$
\end{proof}

As a corollary, we obtain a tail estimate for the sum of random variables satisfying a ``Freedman type''
bound for their moment generating functions. 

\begin{cor}\label{cor-freedman-tensorization}
Let $m\in\N$; let $(M_i)_{i\leq m}$ and $(\sigma_i)_{i\leq m}$ be two sequences of positive numbers
and let random variables $\xi_1,\ldots,\xi_m$ satisfy
\begin{equation*}
\E [e^{\lambda \xi_i}\mid \xi_1,\ldots,\xi_{i-1}]\leq \exp\left(\frac{{\sigma_i}^2}{{M_i}^2}\,\gfunction(\lambda M_i)\right).
\end{equation*}
for any $i\leq m$ and $\lambda\geq 0$. Then for any $t\geq 0$, we have 
$$
\Prob\Big\{\sum_{i\leq m}\xi_i\geq t\Big\}\leq  \exp\left(-\frac{\sigma^2}{M^2}\, H\left(\frac{tM}{\sigma^2}\right)\right),
$$
where $M:=\max_{i\leq m} M_i$ and $\sigma^2:=\sum_{i=1}^m {\sigma_i}^2$.
\end{cor}
\begin{proof}
Fix any $t>0$ and set $\lambda:=\ln(1+tM/\sigma^2)/M$.
In view of the assumptions on $\xi_i$'s and Lemma~\ref{lem-freedman-tensorization},
we have
$$\E\, e^{\lambda \sum_{i=1}^m\xi_i}\leq \prod_{i=1}^m\exp\left( \frac{{\sigma_i}^2}{{M_i}^2}\,\gfunction(\lambda M_i)\right).$$
Since the function $\gfunction(\lambda t)/t^2$ is increasing on $(0,\infty)$, the last relation implies
$$
\E\, e^{\lambda \sum_{i=1}^m\xi_i}\leq \prod_{i=1}^m\exp\left( \frac{{\sigma_i}^2}{M^2}\,\gfunction(\lambda M)\right)
=\exp\left(\frac{\sigma^2}{M^2}\,\gfunction(\lambda M)\right).
$$
Hence, by Markov's inequality,
$$
\Prob\Big\{ \sum_{i\leq m}\xi_i \geq t\Big\}\leq
e^{-\lambda t}\, \E\, e^{\lambda \sum_{i=1}^m\xi_i}
\leq e^{-\lambda t}\, \exp\left( \frac{\sigma^2}{M^2}\,\gfunction(\lambda M)\right).
$$
The result follows after plugging in the expression for $\lambda$.
\end{proof}

\subsection{A crude bound on the number of common in-neighbors}\label{sec-graph}

We start this sub-section with some graph notations.
Let
$\nrangr=([n],E)$ be a directed graph on $[n]$ with the edge set $E$ and adjacency matrix $M$.
For any vertex $i\in[n]$, we define the set of its in-neighbors
\begin{align*}
\innbr(i):=\bigl\{v\leq n:\, \, \, (v,i) \in E\bigr\}=\supp \Col_i(M).
\end{align*}
Similarly, the set of out-neighbors is
\begin{align*}
\outnbr(i):=\bigl\{v\leq n:\,  (i,v) \in E\bigr\}=\supp \Row_i(M).
\end{align*}
Further, for every $I,J\subset [n]$
the set of all edges departing from $I$ and landing in $J$ is denoted by
$$
 \edg(I,J):=\bigl\{ e\in E:\, e=(i,j) \text{ for some } i\in I \text{ and } j\in J\bigr\}.
$$
The set of common in-neighbors of two vertices $u,v$ is
$$
\inco(u, v):=\{i\leq n \, :\, (i, u), (i, v)\in E\} = \supp \Col_u(M) \cap \supp \Col_v(M).
$$

In this sub-section, we estimate the probability that a pair of distinct vertices of a random graph
uniformly distributed on $\DGraphSet(\InDeg,\OutDeg)$,
has many common in-neighbors, conditioned on a special $\sigma$-algebra.
Let us note that (much stronger) results of this type for {\it $d$-regular} directed graphs,
as well as bipartite regular undirected graphs, were obtained in \cite{Cook RSA}.
Unlike in \cite{Cook RSA}, we are only interested in large deviations for $\inco(i,j)$.
On the other hand, the specifics of our setting is that our graphs are not regular (instead, have predefined in- and out-degree
sequences) and that the probability is conditional.
More precisely, given a subset $S\subset [n]$,
let $\mathcal F$ be the $\sigma$-algebra on $\DGraphSet(\InDeg,\OutDeg)$
with atoms of the form $\{\nrangr\in\DGraphSet(\InDeg,\OutDeg):\,\edg( S, [n])=F\}$
for all subsets $F\subset [n]\times[n]$. In other words, each atom of $\mathcal F$
is a set of graphs sharing the same collection of out-edges for vertices in $S$.
Then for any event $\Event\subset \DGraphSet(\InDeg,\OutDeg)$,
we let
$\Prob\big\{\nrangr\in\Event \,|\,\edg( S, [n])\big\}$
be the conditional probability of $\Event$ given $\mathcal F$.

Let us remark that the proof of the main statement of this sub-section
is a rather standard application of the method of {\it simple switchings}
introduced by Senior \cite{Senior}
and developed by McKay and Wormald (see \cite{McKay} as well as survey \cite{Wormald}).
We provide the proof for the reader's convenience.

\begin{prop}\label{th-codegree}
There exist universal constants $c_{\ref{th-codegree}},C_{\ref{th-codegree}}>0$ with the following property.
Asssume that $C_{\ref{th-codegree}} \ln n\leq d\leq (1/2+c_0)n$,
and let two degree sequences $\InDeg,\OutDeg\in\R^n$ satisfy \eqref{eq: degree condition}.
Let $I\subset [n]$ be such that $\vert I\vert \leq c_0 n$.
Then, denoting by $\Event_{\ref{th-codegree}}$ the event
$$
\Event_{\ref{th-codegree}}:=\Big\{\nrangr\in \DGraphSet(\InDeg,\OutDeg):\, \exists i\neq j,\, \vert\inco(i,j)\cap I^c\vert
\geq 0.9d\Big\},
$$
we have
$$\Prob\big\{\nrangr\in\Event_{\ref{th-codegree}}\,|\,\edg(I,[n])\big\} \leq \exp(-c_{\ref{th-codegree}}d).$$
\end{prop}

\bigskip

For the rest of the sub-section, we will assume that $d$ and $I$ satisfy the assumptions of Proposition~\ref{th-codegree},
and we restrict ourselves to an atom of the $\sigma$-algebra generated by $\edg(I,[n])$.
Namely, let $F\subset[n]\times[n]$ be such that the set of graphs from $\DGraphSet(\InDeg,\OutDeg)$
satisfying $\edg(I,[n])=F$, is non-empty. 
Given $1\leq i\neq j\leq n$, we let 
$$
\Event_{i,j}:=\Big\{\nrangr\in \DGraphSet(\InDeg,\OutDeg):\, \edg(I,[n])=F,\,\vert\inco(i,j)\cap I^c\vert \geq 0.9d\Big\}
$$
and for any natural $q\geq 0.8d$, let
$$
\Event_{i,j}^q:=\Big\{\nrangr\in \DGraphSet(\InDeg,\OutDeg):\, \edg(I,[n])=F,\,\vert\inco(i,j)\cap I^c\vert =q\Big\}.
$$

\begin{lemma}\label{lem-size-gd-edges}
Let $\nrangr\in\Event_{1,2}^q$ (for some $q\geq 0.8d$), $q'<q$,
and denote
$$J:=\{j\geq 3:\, \vert \inco(1,2)\cap I^c\setminus \innbr(j)\vert\geq q'\}.$$
Let $\Phi_{1,2}:= I^c\setminus \big(\innbr(1)\cup\innbr(2)\big)$. Then
$$
\vert \edg(\Phi_{1,2}, J)\vert \geq dq\biggl(2-c_0-\frac{6c_0d}{q}-\frac{d}{q-q'}\biggr). 
$$
\end{lemma}
\begin{proof}
First note that since $d\leq (1/2+c_0)n$ and the degree sequences satisfy \eqref{eq: degree condition},
we have
$$\vert \Phi_{1,2}\vert \geq \vert I^c\vert - \vert \innbr(1)\cup\innbr(2)\vert \geq
(1-c_0)n-2d+q\geq q-6c_0 d.$$
Therefore 
\begin{equation}\label{eq-size-edg-phi12}
\vert \edg(\Phi_{1,2}, [n])\vert \geq \vert \Phi_{1,2}\vert\, \max_{i\in \Phi_{1,2}} \OutDeg_i\geq  (1-c_0)d(q-6c_0 d).
\end{equation}
In view of the definition of $J$, 
for any $j\in J^c$ we have
$$\vert\inco(1,2)\cap I^c\cap \innbr(j)\vert\geq q-q'.$$
Hence,
$$
(q-q')\vert J^c\vert \leq  \vert \edg(\inco(1,2)\cap I^c, J^c)\vert \leq \vert \edg(\inco(1,2)\cap I^c, [n])\vert
\leq qd,
$$ 
which implies that $\vert J^c\vert \leq qd/(q-q')$.
On the other hand, for every $j\in J^c$ we have 
$$
\vert \Phi_{1,2}\cap \innbr(j)\vert\leq \vert\innbr(j)\vert - \vert \inco(1,2)\cap I^c\cap \innbr(j)\vert \leq d-q+q',
$$
whence
$$\vert \edg(\Phi_{1,2}, J^c)\vert\leq \vert J^c\vert (d-q+q')
\leq qd\biggl(\frac{d}{q-q'}-1\biggr).$$
Together with \eqref{eq-size-edg-phi12}, this gives the result.
\end{proof}

\begin{lemma}
For any integer $q\geq 0.8d+1$, we have
$\vert \Event_{1,2}^q\vert \leq 0.9 \, \vert \Event_{1,2}^{q-1}\vert .$
\end{lemma}
\begin{proof}
Let us define a relation $R$ on $\Event_{1,2}^{q}\times \Event_{1,2}^{q-1}$ as follows:
 
Pick any $\nrangr\in  \Event_{1,2}^{q}$, and choose an edge $(i,j)\in \edg(\Phi_{1,2}, J)$
and $k\in \inco(1,2)\cap I^c\setminus \innbr(j)$,
where $J$ and $\Phi_{1,2}$ are defined in Lemma~\ref{lem-size-gd-edges} with $q':=\lceil q/7\rceil$. 
Perform the simple switching on the graph $\nrangr$, replacing the edges $(i,j)$ and $(k,1)$ with $(i,1)$ and $(k,j)$ respectively.
Note that the conditions $i\not\in\innbr(1)$ and $k\not\in \innbr(j)$
guarantee that the simple switching does not create multiple edges.
Moreover, since $i\in \Phi_{1,2}$, we obtain a valid graph $\nrangr'\in\Event_{1,2}^{q-1}$.  
We define $R(\nrangr)$ as the set of all graphs $\nrangr'$ which can be obtained from $\nrangr$
via the above procedure.

Using Lemma~\ref{lem-size-gd-edges} and the definition of $J$, we get 
\begin{equation}\label{eq-size-image}
\vert R(\nrangr)\vert \geq \vert \edg(\Phi_{1,2}, J)\vert\cdot \min\limits_{j\in J}\vert \inco(1,2)\cap I^c\setminus \innbr(j)\vert
\geq
\frac{1}{7} dq^2\biggl(2-c_0-\frac{6c_0d}{q}-\frac{d}{q-\lceil q/7\rceil}\biggr).
\end{equation}

Now we estimate the cardinalities of preimages. Let $\nrangr'\in R(\Event_{1,2}^q)$.
In order to reconstruct a graph $\nrangr$ for which $(\nrangr,\nrangr')\in R$, 
we need to perform a simple switching which destroys an edge in $\edgpr(\innbrpr(1)\cap I^c\setminus \incopr(1,2), \{1\})$
and adds an edge connecting a vertex in $\innbrpr(2)\cap I^c\setminus \incopr(1,2)$ to vertex $1$. 
There are at most $(\InDeg_1-q+1)$ choices to destroy an edge in $\edgpr(\innbrpr(1)\cap I^c\setminus \incopr(1,2), \{1\})$,
and at most $(\InDeg_2-q+1)$ possibilities to add an edge connecting $\innbrpr(2)\cap I^c\setminus \incopr(1,2)$ 
to $1$. Finally, there are at most $d$ possibilities to complete the switching. Thus,
by the assumptions on the degree sequences,
$$
\vert R^{-1}(\nrangr')\vert \leq d(d-q+1)^2.
$$
This, together with \eqref{eq-size-image}, the choice of q and the constant $c_0$, 
finishes the proof after using Claim~\ref{multi-al}.
\end{proof}

\bigskip

\begin{proof}[Proof of Proposition~\ref{th-codegree}]
Iterating the last lemma, we deduce that for any $q\geq 0.8d+1$ we have
$$
\vert \Event_{1,2}^q\vert \leq 0.9^{q-\lceil 0.8d\rceil} \,
\vert \Event_{1,2}^{\lceil 0.8d\rceil}\vert
\leq 0.9^{q-\lceil 0.8d\rceil} \, \big\vert \big\{\nrangr\in\DGraphSet(\InDeg,\OutDeg):\,\edg(I,[n])=F\big\}\big\vert.
$$
Hence,
$$
\Prob( \Event_{1,2})=\sum_{q\geq 0.9d} \Prob(\Event_{1,2}^q)
\leq 0.9^{0.1d-1}
\, \Prob\big\{\nrangr\in\DGraphSet(\InDeg,\OutDeg):\,\edg(I,[n])=F\big\}.
$$
Similarly, we have 
$$
\Prob( \Event_{i,j})\leq 0.9^{0.1d-1}
\, \Prob\big\{\nrangr\in\DGraphSet(\InDeg,\OutDeg):\,\edg(I,[n])=F\big\}
$$
for any $i\neq j$. 
Applying the union bound and the definition of $\Event_{\ref{th-codegree}}$,
we deduce that
$$\Prob(\Event_{\ref{th-codegree}}\cap \{\nrangr:\,\edg(I,[n])=F\})\leq n^2\, 0.9^{0.1d-1}\,
\Prob\big\{\nrangr\in\DGraphSet(\InDeg,\OutDeg):\,\edg(I,[n])=F\big\}.$$
The result follows in view of the assumptions on $n$ and $d$.
\end{proof}

\begin{Remark}
Let us emphasize that much sharper bounds on the number of common
in-neighbors can be obtained by applying results
proved later in this paper. However, not being the central subject of this work,
no improvements to Proposition~\ref{th-codegree} will be pursued.
\end{Remark}

\section{A concentration inequality for a matrix row}\label{s: row concentr}

Take a large enough natural number $n$, two degree sequences $\InDeg,\OutDeg\in\R^n$,
satisfying \eqref{eq: degree condition}, 
and a non-negative integer number $m\leq c_0 n$.  
Further, let $Y^1,Y^2,\ldots,Y^m$ be $\{0,1\}$-vectors such that the set of matrices
$$\widetilde\MatrixSet:=\bigl\{M\in\MatrixSet(\InDeg,\OutDeg):\,\Row_i(M)=Y^i\mbox{ for all }i\leq m\bigr\}$$
is non-empty. The parameters $n$, $\InDeg$, $\OutDeg$, $m$ and $Y^1,Y^2,\ldots,Y^m$ are fixed throughout this section.
As we mentioned in Section~\ref{s: global notation}, we always assume \eqref{eq: degree condition}.
Our goal here is to show that, under certain conditions on the degree sequences and
vectors $Y^1,\ldots,Y^m$, the $(m+1)$-st row of the random matrix uniformly
distributed in the set $\widetilde\MatrixSet$ enjoys strong concentration properties.

For each $\ell\leq n$, define
\begin{equation}\label{eq: pl fixed rows}
p_\ell:=\InDeg_\ell-|\{i\leq m:\, Y^i_\ell=1\}|.
\end{equation}
Everywhere in this section, we assume that vectors $Y^1,\ldots,Y^m$ are such that
$p_\ell$'s satisfy
\begin{equation}\label{eq: pl strong}
\InDeg_\ell\geq p_\ell\geq (1-2c_0)\InDeg_\ell\quad \forall \ell\in[n].
\end{equation}
Note that the above condition implies $\vert p_\ell-p_{\ell'}\vert \leq |\InDeg_\ell-\InDeg_{\ell'}|+4c_0d$
for all $\ell,\ell'\in [n]$.
Let us remark that in the second part of the section we will employ much stronger assumptions on $p_\ell$.

Further, let $\Omega$ be the set of all $\{0,1\}$-vectors $v\in\R^n$ such that $|\supp v|=\OutDeg_{m+1}$.
Then we can define an induced probability measure $\Prob_\Omega$ on $\Omega$ by setting
$$\Prob_\Omega(A):=\Prob\bigl\{\Row_{m+1}(M)\in A\,|\,M\in\widetilde\MatrixSet\bigr\},\;\;\;A\subset\Omega.$$
Let $\mathcal{F}_0=\{\emptyset, \Omega\}$. Consider the filtration of $\sigma$-algebras $\{\mathcal{F}_i\}_{i=0}^{n}$
on $(\Omega,\Prob_\Omega)$ 
which reveals the coordinates of the $(m+1)$-st row one by one, i.e.\ 
$\mathcal{F}_i$ is generated by $\mathcal{F}_{i-1}$ and by the variable $v_i$ (where $v$ is distributed on $\Omega$
according to the measure $\Prob_\Omega$)
for any $i=1,2,\ldots,n$.

\subsection{Distribution of the $i$-th coordinate}\label{sec-i-coord}

Everywhere in this sub-section, we assume that the number $d$ satisfies conditions of
Proposition~\ref{th-codegree}, i.e.\
$$d\geq C_{\ref{th-codegree}}\ln n.$$

We fix a number $i\leq n$ and numbers $\varepsilon_1,\varepsilon_2\ldots,\varepsilon_{i-1}\in\{0,1\}$
such that $\Prob_{\Omega}\{v\in\Omega:\,v_j=\varepsilon_j\mbox{ for all }j<i\}>0$.
Let us denote
$$Q:=\{v\in\Omega:\,v_j=\varepsilon_j\mbox{ for all }j<i\}$$
and let $\Prob_Q:=\Prob_\Omega(\cdot\,|\,Q)$ be the induced probability measure on $Q$.
By $\mathcal{F}_{k}\cap Q$ we denote restrictions of the previously defined $\sigma$-algebras to $Q$.
Obviously, $\mathcal{F}_{k}\cap Q=\{\emptyset,Q\}$ for all $k\leq i-1$.

The goal of the sub-section is to develop machinery for dealing with arbitrary {\it functions} on $Q$.
Loosely speaking, given a function $h:Q\to \R$ satisfying certain conditions, we will study the ``impact''
of the $i$-th coordinate of its argument on its value. Then, in Sub-section~\ref{sec-row-concentration},
we will apply the relations established here, together with the Freedman inequality,
to obtain concentration inequalities for the $(m+1)$-st row
of a random matrix uniformly distributed on $\widetilde\MatrixSet$.
The central technical statement of this part of the paper is Lemma~\ref{lem-estimate-diff}.
On the way to stating and proving the lemma, we will go through several auxiliary statements
and introduce several useful notions.

\begin{lemma}\label{lem-prob-v-v'}
Let $Q$ be as above, and let $k\neq \ell\in\{i,\ldots,n\}$.
Further, let vectors $v,v'\in Q$ be such that $\{j\leq n:\,v_j\neq v_j'\}=\{k,\ell\}$ with $v_k=v_\ell'=1$ and $v_\ell=v_k'=0$.
Then
$$\Prob_Q(v)\leq \gamma_{k,\ell}\, \Prob_Q(v'),$$
where 
\begin{equation}\label{eq: gamma def}
\gamma_{k,\ell}:=\frac{1}{1-e^{-c_{\ref{th-codegree}}d}}\left[\frac{p_\ell}{p_k}\, \mathbf{1}_{p_\ell<p_k}+ 
\Big( 1+\frac{p_\ell -p_k}{p_k-\lfloor 0.9d\rfloor}\Big)\, \mathbf{1}_{p_\ell\geq p_k}\right],
\end{equation}
$p_\ell$'s are given by \eqref{eq: pl fixed rows}
and the constant $c_{\ref{th-codegree}}$ is defined in Proposition~\ref{th-codegree}.
\end{lemma}
\begin{proof}
First, let us define 
$$
\widetilde\MatrixSet(Q):=\big\{ M\in \widetilde\MatrixSet:\, \Row_{m+1}(M)\in Q\big\}
$$
and
\begin{align*}
\widetilde\MatrixSet(w)&:=\big\{ M\in \widetilde\MatrixSet(Q):\, \Row_{m+1}(M)=w\big\},\;\;w=v,v'.
\end{align*}
With these notations, we have 
\begin{align*}
\Prob_Q(w)= \frac{\vert\widetilde\MatrixSet(w)\vert }{\vert\widetilde\MatrixSet(Q)\vert},\;\;w=v,v'.
\end{align*}
Next, we denote
$$
\widetilde\MatrixSet^{*}(w):=
\bigl\{M\in \widetilde\MatrixSet(w):\, \vert \supp\Col_k(M)\cap\supp\Col_\ell(M)\cap [m]^c\vert\leq 0.9d\bigr\},\;\;w=v,v',
$$
and for any non-negative integer $r\leq 0.9d$ we set
\begin{align*}
\widetilde\MatrixSet^{*}(w,r)&:=
\bigl\{M\in \widetilde\MatrixSet(w):\, \vert \supp\Col_k(M)\cap\supp\Col_\ell(M)\cap [m]^c\vert=r\bigr\},\;\;w=v,v'.
\end{align*}
Clearly,
\begin{equation}\label{eq1-lem-prob-v-v'}
\widetilde\MatrixSet^{*}(w)=\bigsqcup_{r=1}^{\lfloor 0.9d\rfloor} \widetilde\MatrixSet^{*}(w,r).
\end{equation}
Applying the ``matrix'' version of Proposition~\ref{th-codegree} to the set $\widetilde\MatrixSet^{*}(v)$, we get
\begin{equation}\label{eq2-lem-prob-v-v'}
\vert \widetilde\MatrixSet^{*}(v)\vert \geq \big(1-\exp(-c_{\ref{th-codegree}}d)\big)\, \vert \widetilde\MatrixSet(v)\vert.
\end{equation}
Fix an integer $r\leq 0.9d$. We shall compare the cardinalities of 
$\widetilde\MatrixSet^{*}(v,r)$ and $\widetilde\MatrixSet^{*}(v',r)$.
Let us define a relation $\widetilde R\subset \widetilde\MatrixSet^{*}(v,r)
\times \widetilde\MatrixSet^{*}(v',r)$ as follows:

Pick any $M\in \widetilde\MatrixSet^{*}(v,r)$ and
$s\in [m]^c\cap\supp\Col_{\ell}(M) \setminus \supp\Col_k(M)$. Clearly, we have $M_{ik}=M_{s\ell}=1$ and 
$M_{i\ell}=M_{sk}=0$. Let $M^s$ be the matrix obtained from $M$ by a simple switching operation on the entries
$(i,k)$, $(i,\ell)$, $(s,k)$, $(s,\ell)$.
It is easy to see that $M^s$ belongs to $\widetilde\MatrixSet^{*}(v',t)$.
We set $\widetilde R(M):=\{M^s:\,s\in [m]^c\cap\supp\Col_{\ell}(M) \setminus \supp\Col_k(M)\}$.

Thus,
 $$
\vert  \widetilde R(M)\vert= \vert [m]^c\cap\supp\Col_\ell(M) \setminus \supp\Col_k(M)\vert= p_\ell-r.
 $$
Further, it is not difficult to check that
$\widetilde R(\widetilde\MatrixSet^{*}(v,r))=\widetilde\MatrixSet^{*}(v',r)$
and for any $M'\in \widetilde\MatrixSet^{*}(v',r)$, we have 
  $$
\vert  \widetilde R^{-1}(M')\vert= \vert [m]^c\cap\supp\Col_k(M') \setminus \supp\Col_\ell(M')\vert= p_k-r.
 $$
Hence, by Claim~\ref{multi-al},
$$
\vert \widetilde\MatrixSet^{*}(v,r)\vert 
= \frac{p_\ell-r}{p_k-r}\, \vert \widetilde\MatrixSet^{*}(v',r)\vert .
$$
Using this together with \eqref{eq1-lem-prob-v-v'} and \eqref{eq2-lem-prob-v-v'}, we can write
$$
\big(1-\exp(-c_{\ref{th-codegree}}d)\big)\, \vert \widetilde\MatrixSet(v)\vert\leq \vert \widetilde\MatrixSet^{*}(v)\vert 
=\sum_{r=1}^{\lfloor 0.9d\rfloor}\vert \widetilde\MatrixSet^{*}(v,r)\vert   
\leq \max_{1\leq r\leq 0.9d} \frac{p_\ell-r}{p_k-r}\, \vert \widetilde\MatrixSet^{*}(v')\vert. 
$$
Finally, we divide both sides by $\vert\widetilde\MatrixSet(Q)\vert$ and notice that 
$$
\max_{1\leq r\leq 0.9d} \frac{p_\ell-r}{p_k-r}= 
\frac{p_\ell}{p_k}\, \mathbf{1}_{p_\ell<p_k}+ 
\Big( 1+\frac{p_\ell -p_k}{p_k-\lfloor 0.9d\rfloor }\Big)\, \mathbf{1}_{p_\ell\geq p_k}.
$$
\end{proof}
\begin{Remark}\label{rem: gamma}
Note that under our assumptions on $p_\ell$ and $d$,
we have
$$1-8c_0\leq\gamma_{k,\ell}\leq 1+50c_0\;\;\mbox{ for all $k,\ell\in\{1,2,\ldots,n\}$}.$$
\end{Remark}

\bigskip

We define a relation $\relation\subset Q\times Q$
as
\begin{equation}\label{eq: relation definition}
(v,v')\in \relation\;\;\mbox{if and only if}\;\;|\{j\leq n:\,v_j\neq v_j'\}|=2,
\end{equation}
i.e.\ the pair $(v,v')$ belongs to $\relation$ if $v'$ can be obtained from $v$ by transposing two coordinates.
Further, let us define sets $T_+$ and $T_0$:
\begin{equation}\label{eq: T+ T0 definition}
T_+:=\{v\in Q:\, v_i=1\}\quad \text{ and }\quad T_0:=\{v\in Q:\, v_i=0\},
\end{equation}
so that $Q=T_+\sqcup T_0$. Denote by $\relation_+\subset T_+\times T_0$ the restriction of the relation
$\relation$ to $T_+\times T_0$. 
For a vector $v'\in T_0$, let $N$ be the number of coordinates of $v'$ equal to $1$,
starting from the $i$-th coordinate.
Note that this number does not depend on the choice of $v'\in T_0$,
and is entirely determined by the values of the signs $\varepsilon_1,\varepsilon_2,\ldots,\varepsilon_{i-1}$
which we fixed at the beginning of the sub-section. More precisely, 
\begin{equation}\label{eq: N definition}
N:= \OutDeg_i-\sum_{j<i} \varepsilon_j.
\end{equation}
Note that, provided that both $T_0,T_+$ are not empty,
for any $v'\in T_0$ we have $\vert \relation_+^{-1}(v')\vert=N$.
Moreover, for any $v\in T_+$, the cardinality of $\relation_+(v)$ is the number of coordinates equal to $0$ after 
the $i$-th coordinate in $v$. Therefore, for any $v\in T_+$, we have $\vert \relation_+(v)\vert= n-i-N+1$.

In what follows, we will make frequent use of the quantities
\begin{equation}\label{eq: delta def}
\delta_{k,\ell}:= \max\big( \vert 1-{\gamma_{k,\ell}}^{-1}\vert,
\vert 1-{\gamma_{\ell,k}}^{-1}\vert,\vert 1-\gamma_{k,\ell}\vert,\vert 1-\gamma_{\ell,k}\vert\big),\;\;k,\ell\in\{1,2,\ldots,n\},
\end{equation}
where $\gamma_{k,\ell}$ are defined by \eqref{eq: gamma def}.
From Remark~\ref{rem: gamma}, it immediately follows that $\delta_{k,\ell}\leq 1/4$ for all $k,\ell\in\{1,\ldots,n\}$.
Moreover,
a simple computation shows that
\begin{equation}\label{eq: delta easy bound}
\delta_{k,\ell}\leq \frac{40}{d}|p_k-p_\ell|+4\exp(-c_{\ref{th-codegree}}d),\;\;\;\;k,\ell\in\{1,2,\ldots,n\}.
\end{equation}

\begin{lemma}\label{lem-prob-T+}
Suppose that the sets $T_0,T_+$ are non-empty.
Then
\begin{align*}
\big\vert (n-i-N+1)\Prob_Q(T_+)-N\Prob_Q(T_0)\big\vert
&\leq \Prob_Q(T_+) \frac{2(n-i-N+1)}{n-i} \sum_{\ell=i+1}^n \delta_{i,\ell}\\
&\leq \frac{1}{2}(n-i-N+1)\Prob_Q(T_+),
\end{align*}
where $\delta_{k,\ell}$ are defined by \eqref{eq: delta def}.
\end{lemma}
\begin{proof}
First note that 
$$
\Prob_Q(T_+)=\sum_{v\in T_+} \Prob_Q(v)
=\sum_{v\in T_+}\sum_{v'\in \relation_+(v)} \frac{\Prob_Q(v)}{\vert \relation_+(v)\vert}
=\frac{1}{n-i-N+1}\sum_{(v,v')\in \relation_+} \Prob_Q(v). 
$$
Similarly, for $T_0$ we have
$$\Prob_Q(T_0)=\frac{1}{N}\sum_{(v,v')\in \relation_+} \Prob_Q(v').$$
Hence,
$$
\big\vert (n-i-N+1)\Prob_Q(T_+)-N\Prob_Q(T_0)\big\vert 
=\big\vert \sum_{(v,v')\in \relation_+} \Prob_Q(v)-\Prob_Q(v')\big\vert
\leq \sum_{(v,v')\in \relation_+} \Big\vert 1-\frac{\Prob_Q(v')}{\Prob_Q(v)}\Big\vert\, \Prob_Q(v).
$$
Since for any pair $(v,v')\in \relation_+$, $v$ and $v'$ differ just at one coordinate after $i$-th, we have
$$\relation_+=\bigsqcup_{\ell=i+1}^n \bigl\{(v,v')\in \relation_+:\,v_\ell\neq v_{\ell}'\bigr\},$$
whence
$$
\big\vert (n-i-N+1)\Prob_Q(T_+)-N\Prob_Q(T_0)\big\vert 
\leq \sum_{\ell=i+1}^n \sum_{\substack{v_\ell\neq v_{\ell}'\\(v,v')\in \relation_+}}
\Big\vert 1-\frac{\Prob_Q(v')}{\Prob_Q(v)}\Big\vert\, \Prob_Q(v).
$$
Applying Lemma~\ref{lem-prob-v-v'}, we obtain 
\begin{equation}\label{eq1-lem-prob-T+}
\big\vert (n-i-N+1)\Prob_Q(T_+)-N\Prob_Q(T_0)\big\vert 
\leq \sum_{\ell=i+1}^n \delta_{i,\ell}\sum_{\substack{v_\ell=0\\ v\in T_+}} \, \Prob_Q(v).
\end{equation}

Let us now compare the quantities $a_\ell:=\sum_{\substack{v_\ell=0\\v\in T_+}} \, \Prob_Q(v)$ for two different values of $\ell$. 
Fix $\ell, \ell'>i$ ($\ell\neq \ell'$)
and define a bijection $f: \{v\in T_+:\, v_\ell=0\}\to \{v\in T_+:\, v_{\ell'}=0\}$ as follows:
given $v\in T_+$, if $v_\ell=v_{\ell'}=0$ 
then we set $f(v):=v$; otherwise, if $v_{\ell}=0$ and $v_{\ell'}=1$ then
we let $f(v)$ be the vector obtained by swapping the $\ell$-th and $\ell'$-th coordinates of $v$. 
Note that whenever $v\neq f(v)$, we have $(v,f(v))\in \relation$. Hence, using Lemma~\ref{lem-prob-v-v'}, we get
$$
\frac{a_\ell}{a_{\ell'}}\leq \max_{\substack{v_\ell=0\\v\in T_+}} \frac{\Prob_Q(v)}{\Prob_Q(f(v))} 
\leq \max(1,\gamma_{\ell',\ell})\leq 2,$$
where the last inequality follows from Remark~\ref{rem: gamma}.
This implies 
$$
\max_{\ell >i} a_\ell\leq 2\min_{\ell >i} a_\ell \leq \frac{2}{n-i} \sum_{\ell=i+1}^n a_\ell.
$$
Plugging this estimate into \eqref{eq1-lem-prob-T+}, we deduce that 
$$
\big\vert (n-i-N+1)\Prob_Q(T_+)-N\Prob_Q(T_0)\big\vert 
\leq \frac{2}{n-i} \sum_{\ell=i+1}^n \delta_{i,\ell}\,\sum_{\ell'=i+1}^n a_{\ell'}.
$$
The proof is finished by noticing that 
$$
\sum_{\ell'=i+1}^n a_{\ell'}
= \sum_{(v,v')\in \relation_+}\Prob_Q(v)= (n-i-N+1)\, \Prob_Q(T_+).
$$
\end{proof}

\bigskip

Assume that $T_0,T_+$ are non-empty. Given a couple $(v,v')\in \relation_+$, define 
$$
\rho(v,v'):=\frac{\Prob_Q(v)}{(n-i-N+1)\, \Prob_Q(T_+)}\quad \text{and} \quad \rho'(v,v'):=\frac{\Prob_Q(v')}{N\, \Prob_Q(T_0)}.
$$
Note that $\rho$ and $\rho'$ are probability measures on $\relation_+$. 
In what follows, given a function $h:Q\to\R$,
by $\E_{T_+}$ we denote the expectation of the restriction of $h$ to $T_+$ with respect to $\Prob_Q$, i.e., 
$$
\E_{T_+} h= \frac{1}{\Prob_Q(T_+)}\sum_{v\in T_+} h(v)\,\Prob_Q(v)= \sum_{(v,v')\in \relation_+} \rho(v,v')\, h(v).
$$
Similarly,
$$
\E_{T_0} h=\frac{1}{\Prob_Q(T_0)}\sum_{v'\in T_0} h(v')\,\Prob_Q(v')= \sum_{(v,v')\in \relation_+} \rho'(v,v')\, h(v').
$$
We shall proceed by comparing the measures $\rho$ and $\rho'$:
\begin{lemma}\label{lem-estimate-rho}
Assume that the sets $T_+,T_0$ are non-empty.
Let $(v,v')\in \relation_+$ and let  $q>i$ be an integer such that $v_q\neq v_q'$. 
Then 
$$
\vert \rho(v,v')-\rho'(v,v')\vert \leq \rho'(v,v') \Big[\delta_{i,q}+\frac{4}{n-i} \sum_{\ell=i+1}^n \delta_{i,\ell}\Big].
$$
\end{lemma}

\begin{proof}
Using Lemma~\ref{lem-prob-v-v'} and the definition \eqref{eq: delta def}, we get 
$$
1-\delta_{i,q}\leq\frac{\Prob_Q(v)}{\Prob_Q(v')}\leq 1+\delta_{i,q}.
$$
Now, from Lemma~\ref{lem-prob-T+}, we have
$$
1-\frac{2}{n-i} \sum_{\ell=i+1}^n \delta_{i,\ell}
\leq \frac{N\Prob_Q(T_0)}{(n-i-N+1)\Prob_Q(T_+)}\leq 1+\frac{2}{n-i} \sum_{\ell=i+1}^n \delta_{i,\ell}.
$$
Recall that the assumptions on $d$ and $p_\ell$'s imply that $\delta_{i,\ell}\leq 1$.
Hence, putting together the last two estimates, we obtain
$$
(1-\delta_{i,q})\Big[1-\frac{2}{n-i} \sum_{\ell=i+1}^n \delta_{i,\ell}\Big]\leq \frac{\rho(v,v')}{\rho'(v,v')}\leq 
(1+\delta_{i,q})\Big[1+\frac{2}{n-i} \sum_{\ell=i+1}^n \delta_{i,\ell}\Big].
$$
The proof is finished by multiplying the inequalities by $\rho'(v,v')$ and employing the bound $\delta_{i,q}\leq 1$.
\end{proof}

\begin{lemma}\label{lem-h-T+}
Let, as before, $T_0,T_+$ be given by \eqref{eq: T+ T0 definition}, and
assume that both $T_0,T_+$ are non-empty. Let
$h$ be any function on $Q$. Then for any $\lambda\in\R$, we have
\begin{align*}
\vert \E_{T_+} h-\E_{T_0} h\vert
&\leq\frac{1}{n-i-N+1}\sup\limits_{v\in T_+}\sum\limits_{v'\in \relation_+(v)}\big|h(v)-h(v')\big|\\
&\hspace{1cm}+\frac{8\,\E_{T_0}|h-\lambda|}{n-i} \sum_{\ell=i+1}^n \delta_{i,\ell}\\
&\hspace{1cm}+\Big(\frac{4}{n-i} \sum_{\ell=i+1}^n \delta_{i,\ell}\Big)\, \max_{(v,v')\in \relation} \big\vert h(v)-h(v')\big\vert,
\end{align*}
where $N$ is defined by \eqref{eq: N definition}.
\end{lemma}

Before proving the lemma, let us comment on the idea behind the estimate.
Suppose that the function $h$ is a linear functional in $\R^n$ (actually this is the only
case interesting for us). Then, loosely speaking, we want to show that the difference
$\vert \E_{T_+} h-\E_{T_0} h\vert$ is essentially determined by the value $h(e_i)$.
This corresponds to the first term of the bound, whereas the second and third summands are
supposed to be negligible under appropriate conditions on $h$
(in fact, the second summand $\frac{8\,\E_{T_0}|h-\lambda|}{n-i} \sum_{\ell=i+1}^n \delta_{i,\ell}$
can be problematic and requires special handling).

\begin{proof}[Proof of Lemma~\ref{lem-h-T+}]
Fix any $\lambda\in\R$. 
Using the triangle inequality and the definition of $\E_{T_+} h$ and $\E_{T_0} h$, we obtain
\begin{align*}
\beta&:=\vert \E_{T_+} h-\E_{T_0} h\vert\\
&\leq 
\Big\vert \sum_{(v,v')\in \relation_+} \big(h(v)-h(v')\big)\, \rho(v,v')\Big\vert 
+ \Big\vert \sum_{(v,v')\in \relation_+} \big(\rho(v,v')-\rho'(v,v')\big)\, h(v')\Big\vert\\
&=\Big\vert \sum_{(v,v')\in \relation_+} \big(h(v)-h(v')\big)\, \rho(v,v')\Big\vert 
+ \Big\vert \sum_{(v,v')\in \relation_+} \big(\rho(v,v')-\rho'(v,v')\big)\, \big(h(v')-\lambda\big)\Big\vert\\
&\leq \Big\vert \sum_{(v,v')\in \relation_+} \big(h(v)-h(v')\big)\, \rho(v,v')\Big\vert 
+ \sum_{(v,v')\in \relation_+} \big\vert\rho(v,v')-\rho'(v,v')\big\vert\, \big\vert h(v')-\lambda\big\vert\\
&=\Big\vert \sum_{(v,v')\in \relation_+} \big(h(v)-h(v')\big)\, \rho(v,v')\Big\vert 
+\sum_{\ell=i+1}^n \sum_{\substack{v_\ell\neq v_{\ell}'\\(v,v')\in \relation_+}} \big\vert\rho(v,v')-\rho'(v,v')\big\vert\, \big\vert h(v')-\lambda\big\vert.
\end{align*}
For the first term, applying the definition of $\rho(v,v')$, we get
\begin{align*}
\Big\vert \sum_{(v,v')\in \relation_+} \big(h(v)-h(v')\big)\, \rho(v,v')\Big\vert
&\leq \sum_{(v,v')\in \relation_+}
\frac{\big|h(v)-h(v')\big|\Prob_Q(v)}{(n-i-N+1)\, \Prob_Q(T_+)}\\
&\leq \frac{1}{n-i-N+1}\sup\limits_{v\in T_+}\sum\limits_{v'\in \relation_+(v)}\big|h(v)-h(v')\big|.
\end{align*}
Next, in view of Lemma~\ref{lem-estimate-rho},
\begin{align*}
&\sum_{\ell=i+1}^n\sum_{\substack{v_\ell\neq v_{\ell}'\\(v,v')\in \relation_+}} \big\vert\rho(v,v')-\rho'(v,v')\big\vert\, \big\vert h(v')-\lambda\big\vert\\
&\hspace{0.5cm}\leq\sum_{\ell=i+1}^n \sum_{\substack{v_\ell\neq v_{\ell}'\\(v,v')\in \relation_+}} 
\big\vert h(v')-\lambda\big\vert\, \rho'(v,v') \Big[\delta_{i,\ell}+\frac{4}{n-i} \sum_{q=i+1}^n \delta_{i,q}\Big]\\
&\hspace{0.5cm}=
\frac{4\E_{T_0}|h-\lambda|}{n-i} \sum_{\ell=i+1}^n \delta_{i,\ell}+\sum_{\ell=i+1}^n \sum_{\substack{v_\ell\neq v_{\ell}'\\(v,v')\in \relation_+}} 
\delta_{i,\ell}\big\vert h(v')-\lambda\big\vert\,\rho'(v,v').
\end{align*}
Denote
$$
\alpha:=\max_{(v,v')\in \relation} \vert h(v)-h(v')\vert \quad \text{and}\quad a_\ell:= \sum_{\substack{v_\ell\neq v_{\ell}'\\(v,v')\in \relation_+}} 
\big(\alpha+\big\vert h(v')-\lambda\big\vert\big)\, \rho'(v,v')\quad\text{ for any $\ell>i$.}
$$
Then, obviously,
\begin{equation}\label{eq1-lem-h-T+}
\sum_{\ell=i+1}^n \sum_{\substack{v_\ell\neq v_{\ell}'\\(v,v')\in \relation_+}} 
\delta_{i,\ell}\big\vert h(v')-\lambda\big\vert\,\rho'(v,v')
\leq \sum_{\ell=i+1}^n \delta_{i,\ell} \, a_\ell.
\end{equation}

Similarly to the argument within the proof of Lemma~\ref{lem-prob-T+}, we shall compare 
$a_\ell$'s for any two distinct values of $\ell$. Fix $\ell\neq\ell'>i$ and define a bijection 
$f: \{v'\in T_0:\, v_\ell'=1\}\to \{v'\in T_0:\, v_{\ell'}'=1\}$ as follows: given $v'\in T_0$ with $v_\ell'=v_{\ell'}'=1$,
set $f(v'):=v'$; otherwise, if $v_\ell'=1$ and $v_{\ell'}'=0$ then
let $f(v')$ to be the vector obtained by swapping $\ell$-th and $\ell'$-th coordinate of $v'$. 
Note that in the latter case $(v',f(v'))\in \relation$.
Applying Lemma~\ref{lem-prob-v-v'}, we get
\begin{equation}\label{eq2-lem-h-T+}
\frac{a_\ell}{a_{\ell'}}\leq
 \max_{\substack{v'_\ell=1\\v'\in T_0}} 
 \frac{\big(\alpha+\big\vert h(v')-\lambda\big\vert\big)\, \Prob_Q(v')}{\big(\alpha+\big\vert h\big(f(v')\big)-\lambda\big\vert\big)\, \Prob_Q(f(v'))}
 \leq \max(1,\gamma_{\ell, \ell'}) \, \max_{\substack{v'_\ell=1\\v'\in T_0}} 
 \frac{\big(\alpha+\big\vert h(v')-\lambda\big\vert\big)}{\big(\alpha+\big\vert h\big(f(v')\big)-\lambda\big\vert\big)},
\end{equation}
where $\gamma_{\ell,\ell'}$ is defined by \eqref{eq: gamma def}.
On the other hand, since $(v',f(v'))\in \relation$ whenever $v'\neq f(v')$, we have 
$$
\big\vert h(v')-\lambda\big\vert\leq \big\vert h\big(f(v')\big)-\lambda\big\vert +
\big\vert h(v')-h\big(f(v')\big)\big\vert \leq \big\vert h\big(f(v')\big)-\lambda\big\vert +\alpha.
$$
Plugging the last relation into \eqref{eq2-lem-h-T+} and using the bound $\gamma_{\ell,\ell'}\leq 2$, we get 
$$
a_\ell\leq 4a_{\ell'},\;\;\;\ell,\ell'\in\{i+1,\ldots,n\}.
$$
This implies
$$
\max_{\ell>i} a_\ell\leq 4\min_{\ell>i} a_\ell \leq \frac{4}{n-i}\sum_{\ell>i} a_\ell
=\frac{4}{n-i}\big( \alpha+\E_{T_0} \vert h-\lambda\vert\big).
$$
Together with \eqref{eq1-lem-h-T+}, the last relation gives
$$\sum_{\ell=i+1}^n \sum_{\substack{v_\ell\neq v_{\ell}'\\(v,v')\in \relation_+}} 
\delta_{i,\ell}\big\vert h(v')-\lambda\big\vert\,\rho'(v,v')
\leq
\Big(\frac{4}{n-i} \sum_{\ell=i+1}^n \delta_{i,\ell}\Big)\, \big(\alpha+\E_{T_0} \vert h-\lambda\vert\big).
$$
It remains to combine the above estimates.
\end{proof}

\begin{Remark}
We do not know if a more careful analysis can give a bound for $\vert \E_{T_+} h-\E_{T_0} h\vert$
in the above lemma, not involving dependence on $\E_{T_0}|h-\lambda|$.
\end{Remark}

Let, as before, $h$ be a function on $Q$. We set
$$X_k:=\E[ h\mid \mathcal{F}_k\cap Q],\;\;\;k=i-1,\ldots,n,$$
where $\sigma$-algebras $\mathcal{F}_k$ are defined at the beginning of the section.
Clearly, $(X_k)_{i-1\leq k\leq n}$ is a martingale. 
Denote by $(d_k)_{i\leq k\leq n}$ the difference sequence, i.e.\
$$d_k:=X_k-X_{k-1},\;\;\;k=i,\ldots, n.$$ 
Further, let $M$ and $\sigma_i$ be smallest non-negative numbers such that $|d_k|\leq M$ a.s.\ for all $i\leq k\leq n$,
and $\sum_{k=i+1}^n\E(d_k^2\,|\,\mathcal F_{k-1}\cap Q)\leq{\sigma_i}^2$ a.s.\ (note that, since our probability space is finite,
such numbers always exist).

\begin{lemma}\label{lem-exp-on-T0}
Assume that $T_0$ is non-empty.
Then, with the above notations, we have
$$
\E_{T_0} ( h-\E_{T_0}  h)^2\leq {\sigma_i}^2.
$$
\end{lemma}
\begin{proof}
First, note that $\E_{T_0} ( h-\E_{T_0}  h)^2$, viewed as a (constant) function on $T_0$,
is just a restriction of the random variable $\E\big[( h-\E[h\mid \mathcal{F}_i\cap Q])^2\mid \mathcal{F}_i\cap Q\big]$
to the set $T_0$. Hence, it is sufficient to prove the inequality
$$\E\big[( h-\E[h\mid \mathcal{F}_i\cap Q])^2\mid \mathcal{F}_i\cap Q\big]\leq {\sigma_i}^2.$$
We have
$$h-\E[h\mid \mathcal{F}_i\cap Q]=\sum_{k=i+1}^n d_k,$$
whence
$$\E\big[( h-\E[h\mid \mathcal{F}_i\cap Q])^2\mid \mathcal{F}_i\cap Q\big]
= \sum_{k,\ell=i+1}^n \E\big[d_k d_\ell\mid \mathcal{F}_i\cap Q\big]\leq {\sigma_i}^2
+\sum_{\substack{k,\ell=i+1\\k\neq \ell}}^n\E\big[d_k d_\ell\mid \mathcal{F}_i\cap Q\big].$$
Finally, we note that $\E\big[d_k d_\ell\mid \mathcal{F}_i\cap Q\big]=0$ for all $k\neq \ell$.
\end{proof}

Now, we can state the main technical result of the sub-section:

\begin{lemma}\label{lem-estimate-diff}
Let, as before, the relation $\relation$, sets $T_0$ and $T_+$ and the number $N$
be defined by \eqref{eq: relation definition}, \eqref{eq: T+ T0 definition} and~\eqref{eq: N definition},
respectively, and let $\delta_{i,\ell}$ be given by \eqref{eq: delta def}.
Then, with the above notation for the martingale sequence,
\begin{align*}
\vert d_i\vert\leq
\frac{1}{n-i-N+1}\sup\limits_{v\in T_+}\sum\limits_{v'\in \relation_+(v)}\big|h(v)-h(v')\big|
&+\frac{8\sigma_i}{n-i} \sum_{\ell=i+1}^n \delta_{i,\ell}\\
&+\Big(\frac{4}{n-i} \sum_{\ell=i+1}^n \delta_{i,\ell}\Big)\, \max_{(v,v')\in \relation} \big\vert h(v)-h(v')\big\vert,
\end{align*}
and 
\begin{align*}
\E [d_i^2\mid\mathcal{F}_{i-1}\cap Q]
\leq \frac{4N}{n-i-N+1}
\Big[&\frac{1}{n-i-N+1}\sup\limits_{v\in T_+}\sum\limits_{v'\in \relation_+(v)}\big|h(v)-h(v')\big|\\
&+\frac{8\sigma_i}{n-i} \sum_{\ell=i+1}^n \delta_{i,\ell}\\
&+\Big(\frac{4}{n-i} \sum_{\ell=i+1}^n \delta_{i,\ell}\Big)\, \max_{(v,v')\in \relation} \big\vert h(v)-h(v')\big\vert\Big]^2.
\end{align*}
\end{lemma} 
\begin{proof}
When one of the sets $T_0$ or $T_+$ is empty, we have $d_i=0$, and the statement is obvious.
Otherwise, it is easy to see that 
$$
X_{i-1}= \E_Q h=\Prob_Q(T_+)\,\E_{T_+} h   + \Prob_Q(T_0)\,\E_{T_0}h
\quad \text{ and }\quad
X_i= \mathbf{1}_{T_+}\E_{T_+} h  + \mathbf{1}_{T_0}\E_{T_0} h,
$$
where $\mathbf{1}_{T_+},\mathbf{1}_{T_0}$ are indicators of the corresponding subsets of $Q$.
Thus, we have 
\begin{equation}\label{eq: di formula}
d_i= \mathbf{1}_{T_+}\Prob_Q(T_0)\, \big[\E_{T_+} h-\E_{T_0} h\big]
-\mathbf{1}_{T_0}\Prob_Q(T_+)\, \big[\E_{T_+} h-\E_{T_0} h\big],
\end{equation}
whence
$$
\vert d_i\vert
\leq\max\big(\Prob_Q(T_+), \Prob_Q(T_0)\big)\, \vert  \E_{T_+} h-\E_{T_0} h\vert \leq 
 \vert  \E_{T_+} h-\E_{T_0} h\vert.
$$
Applying Lemma~\ref{lem-h-T+} with $\lambda:=\E_{T_0}h$, we get
\begin{align*}
\vert d_i\vert
\leq\frac{1}{n-i-N+1}\sup\limits_{v\in T_+}\sum\limits_{v'\in \relation_+(v)}\big|h(v)-h(v')\big|
&+\frac{8\,\E_{T_0}|h-\E_{T_0}h|}{n-i} \sum_{\ell=i+1}^n \delta_{i,\ell}\\
&+\Big(\frac{4}{n-i} \sum_{\ell=i+1}^n \delta_{i,\ell}\Big)\, \max_{(v,v')\in \relation} \big\vert h(v)-h(v')\big\vert.
\end{align*}
The first part of the lemma follows by using Lemma~\ref{lem-exp-on-T0}.

Next, we calculate the conditional second moment of $d_i$. As an immediate consequence of \eqref{eq: di formula}, we get
$$
d_i^2=\Prob_Q(T_0)^2\big(\E_{T_+} h-\E_{T_0} h\big)^2 \,  \mathbf{1}_{T_+} 
+\Prob_Q(T_+)^2\big(\E_{T_+} h-\E_{T_0} h\big)^2 \,  \mathbf{1}_{T_0},
$$
whence
\begin{align*}
\E [d_i^2\mid\mathcal{F}_{i-1}\cap Q]&= \Prob_Q(T_0)^2\Prob_Q(T_+)\big(\E_{T_+} h-\E_{T_0} h\big)^2 
+\Prob_Q(T_+)^2\Prob_Q(T_0)\big(\E_{T_+} h-\E_{T_0} h\big)^2 \\
&= \Prob_Q(T_+)\Prob_Q(T_0) \, \big(\E_{T_+} h-\E_{T_0} h\big)^2 
\end{align*}
Applying Lemma~\ref{lem-h-T+} with $\lambda:=\E_{T_0}h$ and Lemma~\ref{lem-exp-on-T0}, we get
\begin{align*}
\frac{\E [d_i^2\mid\mathcal{F}_{i-1}\cap Q]}{\Prob_Q(T_+)\Prob_Q(T_0)}\leq 
\Big[&\frac{1}{n-i-N+1}\sup\limits_{v\in T_+}\sum\limits_{v'\in \relation_+(v)}\big|h(v)-h(v')\big|\\
&+\frac{8\sigma_i}{n-i} \sum_{\ell=i+1}^n \delta_{i,\ell}\\
&+\Big(\frac{4}{n-i} \sum_{\ell=i+1}^n \delta_{i,\ell}\Big)\, \max_{(v,v')\in \relation} \big\vert h(v)-h(v')\big\vert\Big]^2
\end{align*}
It remains to note that $\Prob_Q(T_+)\Prob_Q(T_0)\leq \Prob_Q(T_+)\leq \frac{2N}{n-i-N+1}$,
in view of Lemma~\ref{lem-prob-T+}.
\end{proof}

Both estimates of the absolute value of $d_i$ and of its conditional variance
contain the term $\frac{1}{n-i-N+1}\sup\limits_{v\in T_+}\sum\limits_{v'\in \relation_+(v)}\big|h(v)-h(v')\big|$.
In the next simple lemma, we bound the expression in the case when $h$ is a linear functional.

\begin{lemma}\label{l: average of hv}
Let a function $h:Q\to\R$ be given by $h(v):=\langle v,x\rangle$
for a fixed vector $x\in\R^n$. Further, assume that $i\leq n/4$. Then
$$\frac{1}{n-i-N+1}\sup\limits_{v\in T_+}\sum\limits_{v'\in \relation_+(v)}\big|h(v)-h(v')\big|
\leq |x_i|+\frac{8\|x\|_1}{n}.$$
\end{lemma}
\begin{proof}
Obviously, for any couple $(v,v')\in \relation_+$ with $v_\ell\neq v_\ell'$ for some $\ell>i$ we have
$$\big|h(v)-h(v')\big|\leq |x_i|+|x_\ell|.$$
Whence, for any $v\in T_+$,
$$\sum\limits_{v'\in \relation_+(v)}\big|h(v)-h(v')\big|
\leq |\relation_+(v)|\,|x_i|+\sum_{\ell>i: v_\ell=0}|x_\ell|\leq |\relation_+(v)|\,|x_i|+\|x\|_1.$$
It follows that
$$\frac{1}{n-i-N+1}\sup\limits_{v\in T_+}\sum\limits_{v'\in \relation_+(v)}\big|h(v)-h(v')\big|
< |x_i|+\frac{8\|x\|_1}{n}.$$ 
\end{proof}

\subsection{$(m+1)$-st row is conditionally concentrated}\label{sec-row-concentration}

In this sub-section we show that given a fixed vector $x\in\R^n$ and
a random vector $v$ distributed on $\Omega$ according to the measure $\Prob_\Omega$,
the scalar product $\langle v,x\rangle$ is concentrated around its expectation.
Naturally, this holds under some extra assumptions on the quantities $p_\ell$ introduced at the
beginning of the section, which measure how close to ``homogeneous'' the probability space $(\Omega,\Prob_\Omega)$
is. As everywhere in the sub-section, we assume that the degree sequences and parameters $p_\ell$ satisfy conditions \eqref{eq: degree condition}
and~\eqref{eq: pl strong}.
Additionally, throughout the sub-section we assume that
\begin{equation}\label{eq-condition-d-sec-row-concentration}
d\geq C_{\ref{sec-row-concentration}}\ln^2 n,
\end{equation}
where  $C_{\ref{sec-row-concentration}}$ is a sufficiently large universal constant
(let us note that in its full strength the assumption is only used in the proof of Lemma~\ref{l: small linfty norm} below).
Define a vector $\pvector=(\pvector_1,\pvector_2,\ldots,\pvector_n)$ as
$$\pvector_\ell:=\sum_{j=1}^n |p_\ell-p_j|,\;\;\ell\leq n.$$
Note that, in view of \eqref{eq: delta easy bound} and~\eqref{eq-condition-d-sec-row-concentration}, we have
\begin{equation}\label{eq: delta via pvector}
\sum_{\ell=1}^n \delta_{i,\ell}\leq \frac{40}{d}\pvector_i+1
\end{equation}
for any $i\leq n$.

\medskip

In the previous sub-section, we estimated parameters of the martingale difference sequence
generated by the variable $\langle \cdot,x\rangle$ and $\sigma$-algebras $\mathcal F_\ell$.
Recall that
the estimate of the upper bound for $|d_i|$ from Lemma~\ref{lem-estimate-diff} involves
the quantity $\frac{\sigma_i}{n-i} \sum_{\ell=1}^n \delta_{i,\ell}$.
In Section~\ref{s: tensorization}, applying \eqref{eq: delta via pvector},
we will show that for ``most'' indices $i$, the sum $\sum_{\ell=1}^n \delta_{i,\ell}$
is bounded by $O(n/\sqrt{d})$, whereas, as we shall see below, $\sigma_i=O(\sqrt{d/n})$ for any unit vector $x$.
Thus, the magnitude of $\frac{\sigma_i}{n-i} \sum_{\ell=1}^n \delta_{i,\ell}$
is of order $n^{-1/2}$, and it is necessarily dominated by a constant multiple of $\|x\|_\infty$.
However, for {\it some} indices $i$ the sum $\sum_{\ell=1}^n \delta_{i,\ell}$
can be as large as $n\ln n\,/\sqrt{d}$. Thus, a straighforward argument would give
$C(\|x\|_\infty +n^{-1/2}\ln n)$ as an upper bound for $d_i$, and the implied row concentration inequality
would bear the logarithmic error term. To overcome this problem, we have to consider separately two cases:
when the $\|\cdot\|_\infty$-norm of the vector $x$ is ``large'' and when it is ``small''. In the first case
(treated in Lemma~\ref{l: big linfty norm})
the logarithmic spikes of the vector $\pvector$ do not create problems.
In the second case, however, we have to apply a special ordering to coordinates of the row so that
large spikes of $\pvector$ are ``balanced'' by a small magnitude of $\sigma_i$ (which, for those coordinates $i$,
must be much smaller than $\sqrt{d/n}$). The second case is more technically involved and is given in Lemma~\ref{l: small linfty norm}.
Finally, when we have both statements in possession, we can complete the proof of the row concentration inequality.

\begin{lemma}\label{l: big linfty norm}
For any $L>0$ there exist $\alpha=\alpha(L)\in(0,1)$ and $\beta=\beta(L)\in(0,1)$ with the following property. 
Let $x\in S^{n-1}$ be an $\alpha n$-sparse vector and assume that
$1)$ $\|x\|_\infty \geq \ln(2n)\, n^{-1/2}$ and $2)$
$\|\pvector\|_{\psi,n}\leq L n\sqrt{d}$, where the norm $\|\cdot\|_{\psi,n}$ is defined by \eqref{eq: psi definition}.
Then, denoting by $\eta$ the random variable
$$\eta=\eta(v):=\langle v,x\rangle - \E_\Omega \langle \cdot,x\rangle,\quad v\in\Omega,$$
we have
$$\E_\Omega\, e^{\beta \lambda \eta}
\leq \exp\bigg(\frac{d}{n\|x\|_\infty^2}\,\gfunction(\lambda \|x\|_\infty)\bigg),\quad \lambda>0,$$
with $\gfunction(\cdot)$ defined by \eqref{eq: e definition}.
\end{lemma}
\begin{proof}
Let $L>0$ be fixed.
We define $\alpha=\alpha(L)$ as the largest number in $(0,1/4]$ such that
\begin{equation}\label{eq: alpha condition}
32\cdot 640C_{\ref{l: elementary psi estimate}} L\sqrt{\alpha}\ln\frac{2}{\alpha}\leq 1,
\end{equation}
where the constant $C_{\ref{l: elementary psi estimate}}$ is given in Lemma~\ref{l: elementary psi estimate}.

Pick an $\alpha n$-sparse vector $x\in S^{n-1}$
and let $\pi$ be a permutation on $[n]$ such that $|x_{\pi(1)}|\geq |x_{\pi(2)}|\geq\ldots\geq |x_{\pi(n)}|$.
For any $i\leq n$, we denote by $\pi(\mathcal F_i)$ the $\sigma$-algebra generated by
coordinates $\pi(1),\ldots,\pi(i)$ of a vector distributed on $\Omega$ according to the measure $\Prob_\Omega$, i.e.\
$$\pi(\mathcal F_i):=\sigma\big(\bigl\{v\in\Omega:\,v_{\pi(j)}=b_j\mbox{ for all }j\leq i\bigr\},\;b_1,b_2,\ldots,b_i\in\{0,1\}\big).$$

Define a function $h$ on $\Omega$ by
$$h(v):= \langle v,x\rangle,\quad v\in\Omega,$$
and let
$$X_\ell:=\E_\Omega[h\mid \pi(\mathcal F_\ell)],\;\;\ell\leq n,$$
and $d_\ell:=X_{\ell}-X_{\ell-1}$.
Further, let $M$ and $\sigma$ be the smallest non-negative numbers such that $|d_\ell|\leq M$ everywhere on $\Omega$ for all $\ell\leq n$,
and $\sum_{\ell=1}^n\E(d_\ell^2\,|\,\pi(\mathcal F_{\ell-1}))\leq\sigma^2$ everywhere on $\Omega$.
Clearly, for any $i>\alpha n$ we have $d_i=0$.
Now, fix $i\leq \alpha n$ and follow the notations of the previous sub-section (with $\pi(\ell)$ replacing $\ell$
where appropriate). More precisely, we take an atom of the algebra $\pi(\mathcal F_{i-1})$
i.e.\ the set $Q$ of vectors in $\Omega$ with some prescribed values of their coordinates with indices $\pi(1),\dots,\pi(i-1)$.
Then $\relation$ is a collection of all pairs of vectors from $Q$ which differ by two coordinates
and $N$ is the number of non-zero coordinates in every $v\in Q$, excluding coordinates with indices $\pi(1),\dots,\pi(i-1)$.
In view of the choice of $\pi$ and the definition of $h$, we have
$$
\max\limits_{\substack{(v,v')\in \relation\\ }}\vert h(v)-h(v')\vert \leq 2\vert x_{\pi(i)}\vert.
$$
Further, using the condition $\delta_{\pi(i),\ell}\leq 1/4$, we get
$$\Big(\frac{4}{n-i} \sum_{\ell=1}^n \delta_{\pi(i),\ell}\Big)\, \max_{(v,v')\in \relation} \big\vert h(v)-h(v')\big\vert
\leq 4|x_{\pi(i)}|.$$
Together with Lemma~\ref{lem-estimate-diff}, Lemma~\ref{l: average of hv} and (\ref{eq: delta via pvector}), this gives 
\begin{align*}
\vert d_i\vert &\leq 5|x_{\pi(i)}|+\frac{8\|x\|_1}{n}
+\frac{640\sigma}{dn}\pvector_{\pi(i)}+\frac{16\sigma}{n}
\end{align*}
everywhere on $Q$ and, in fact, everywhere on $\Omega$ as the right-hand side of the last relation does not depend on the choice
of atom $Q$.
Further, applying the second part of Lemma~\ref{lem-estimate-diff} with Lemma~\ref{l: average of hv}
and relations $i\leq n/4$, $N\leq d\leq n/2+c_0n$ and \eqref{eq: delta via pvector}, we get
\begin{align*}
\E [d_i^2\mid\mathcal{F}_{i-1}]
&\leq \frac{32d}{n}
\Big[5|x_{\pi(i)}|+\frac{8\|x\|_1}{n}
+\frac{16\sigma}{n} \sum_{\ell=1}^n \delta_{\pi(i),\ell}\Big]^2\\
&\leq \frac{32d}{n}
\Big[75{x_{\pi(i)}}^2+\frac{48}{n}+3\Big(\frac{640\sigma}{dn}\pvector_{\pi(i)}+\frac{16\sigma}{n}\Big)^2\Big],
\end{align*}
where in the last inequality we used the convexity of the square and $\Vert x\Vert_1^2\leq \alpha n\, \Vert x\Vert_2^2\leq n/4$. 
Again, the bound for $\E [d_i^2\mid\mathcal{F}_{i-1}]$ holds everywhere on $\Omega$.
Summing over all $i\leq \alpha n$, we get from the last relation
\begin{align*}
\sigma^2&\leq
\frac{32d}{n}
\Big[87+3\sum_{i=1}^{\lfloor \alpha n\rfloor}\Big(\frac{640\sigma}{dn}\pvector_{\pi(i)}+\frac{16\sigma}{n}\Big)^2\Big]\\
&\leq \frac{32d}{n}
\Big[87+\frac{6\cdot 640^2\sigma^2}{d^2n^2}\sum_{i=1}^{\lfloor \alpha n\rfloor}
{\pvector_{\pi(i)}}^2+\frac{6\cdot 16^2\sigma^2}{n}\Big].
\end{align*}
In view of the condition on $\|\pvector\|_{\psi,n}$, relation~\eqref{eq: alpha condition}
and Lemma~\ref{l: elementary psi estimate}, we have
$$\frac{32d}{n}\frac{6\cdot 640^2\sigma^2}{d^2n^2}\sum_{i=1}^{\lfloor \alpha n\rfloor}
{\pvector_{\pi(i)}}^2\leq \frac{\sigma^2}{4}.$$
Thus, the self-bounding estimate for $\sigma$ implies
$$\sigma^2< \frac{Cd}{n},$$
for an appropriate constant $C>0$. 
Whence, from the above estimate of $|d_i|$'s we obtain
\begin{align*}
M&\leq 5\|x\|_\infty+\frac{8\|x\|_1}{n}
+\frac{640\sigma}{dn}\|\pvector\|_\infty+\frac{16\sigma}{n}\\
&\leq  5\|x\|_\infty+\frac{8}{\sqrt{n}}+ \frac{640\, \sqrt{C}\, L\ln(en)}{\sqrt{n}}+ \frac{16\sqrt{C}}{n}
\end{align*}
where we employed the relations $\|\pvector\|_\infty\leq \ln(en) \Vert \pvector\Vert_{\psi,n}\leq L n\sqrt{d}\ln (en)$
(see formula~\eqref{eq: psi infinity}) and the estimate for $\sigma$ established above.
This, together with the assumption $\|x\|_\infty\geq \ln(2n)\, n^{-1/2}$,
implies that $M\leq C'(1+L)\, \Vert x\Vert_\infty$ for an appropriate constant $C'$. 
It remains to apply Theorem~\ref{th-freedman-moment} in order to finish the proof.
\end{proof}
The next lemma is a counterpart of the above statement, covering the case when the $\|\cdot\|_\infty$-norm
of the vector $x$ is small.
\begin{lemma}\label{l: small linfty norm}
For any $L>0$ there exist $\alpha=\alpha(L)\in(0,1)$ and $\beta=\beta(L)\in(0,1)$ with the following property.
Let $x\in S^{n-1}$ be an $\alpha n$-sparse vector and assume that
$1)$ $\|x\|_\infty < \ln(2n)\, n^{-1/2}$ and $2)$
$\|\pvector\|_{\psi,n}\leq L n\sqrt{d}$.
Then, denoting by $\eta$ the random variable
$$\eta=\eta(v):=\langle v,x\rangle - \E_\Omega \langle \cdot,x\rangle,\quad v\in\Omega,$$
we have
$$\E_\Omega\, e^{\beta\lambda \eta}
\leq \exp\bigg(\frac{d}{n\|x\|_\infty^2}\,\gfunction(\lambda \|x\|_\infty)\bigg),\quad \lambda>0.$$
\end{lemma}
\begin{proof}
Again, we fix $L>0$. Let $C_1>0$ be a large enough universal constant (whose exact value can be determined from the proof below). 
We define $\alpha=\alpha(L)$ as the largest number in $(0,1/4]$ such that
\begin{equation}\label{eq2: alpha condition}
C_1 L^2\alpha\ln^2\frac{e}{\alpha}\leq \frac{1}{2}.
\end{equation}
Let $\pi$ be a permutation on $[n]$ such that $|x_{\pi(i)}|=0$ for all $i>|\supp x|$ and
the sequence $(\pvector_{\pi(i)})_{i\leq |\supp x|}$ is non-decreasing.
We define the function $h$, $\sigma$-algebras $\pi(\mathcal F_\ell)$ and the difference sequence
$(d_\ell)_{\ell\leq n}$ the same way as in the proof of Lemma~\ref{l: big linfty norm}.
We have $d_i=0$ for all $i>|\supp x|$.
Let $M$ and $\sigma_i$ ($i\leq |\supp x|$) be the smallest numbers such that everywhere on $\Omega$ we have
$|d_i|\leq M$ for all $i\leq |\supp x|$ and
$$\sum_{\ell=i}^{|\supp x|}\E(d_\ell^2\,|\,\pi(\mathcal F_{\ell-1}))\leq {\sigma_i}^2,\quad i\leq |\supp x|.$$
We fix any $i\leq |\supp x|$ and follow the notations of the previous sub-section
(the way it was described in Lemma~\ref{l: big linfty norm}). 
Recall that $\max_{(v,v')\in \relation} \big\vert h(v)-h(v')\big\vert\leq 2\Vert x\Vert_\infty$. Now, using 
Lemmas~\ref{lem-estimate-diff} and~\ref{l: average of hv}, inequality \eqref{eq: delta via pvector}, as well as relations $i\leq n/4$
and $N\leq d\leq n/2+c_0 n$, we obtain
\begin{align*}
\E [d_i^2\mid \pi(\mathcal{F}_{i-1})]
&\leq \frac{32d}{n}
\Big[|x_{\pi(i)}|+\frac{8\Vert x\Vert_1}{n}
+(\Vert x\Vert_\infty+\sigma_i)\frac{8}{n-i} \sum_{\ell=1}^n \delta_{\pi(i),\ell}\Big]^2\\
&\leq \frac{32d}{n}
\Big[|x_{\pi(i)}|+\frac{8}{\sqrt{n}}
+(\Vert x\Vert_\infty+\sigma_i)\frac{16}{n}\Big(\frac{40\pvector_{\pi(i)}}{d}+1\Big)\Big]^2\\
&\leq \frac{32d}{n}\Big[4 {x_{\pi(i)}}^2+\frac{256}{n}+
(4\Vert x\Vert_\infty^2+4\sigma_i^2)\frac{256}{n^2}\Big(\frac{40\pvector_{\pi(i)}}{d}+1\Big)^2\Big],
\end{align*}
where in the last inequality we used the convexity of the square. Since $\sigma_\ell \leq \sigma_i$ for any 
$i\leq \ell \leq \vert \supp x\vert$, we have 
$$
\E [d_\ell^2\mid \pi(\mathcal{F}_{\ell-1})]\leq 
\frac{32d}{n}\Big[4 {x_{\pi(\ell)}}^2+\frac{256}{n}+
(4\Vert x\Vert_\infty^2+4\sigma_i^2)\frac{256}{n^2}\Big(\frac{40\pvector_{\pi(\ell)}}{d}+1\Big)^2\Big]
$$
for any $i\leq \ell\leq \vert \supp x\vert$.  
Summing over all such $\ell$'s, we get
\begin{equation}\label{eq1-sigma: small linfty norm}
{\sigma_i}^2
\leq \frac{128d}{n}\Big[\sum_{\ell=i}^{|\supp x|}{x_{\pi(\ell)}}^2+\frac{64}{n}\big(|\supp x|-i+1\big)+
(\Vert x\Vert_\infty^2+\sigma_i^2)\frac{256}{n^2}\sum_{\ell=i}^{\vert\supp x\vert}\Big(\frac{40\pvector_{\pi(\ell)}}{d}+1\Big)^2\Big].
\end{equation}
Note that, by the definition of $\|\cdot\|_{\psi,n}$-norm and in view of the fact that the sequence $(\pvector_{\pi(\ell)})_{\ell\leq |\supp x|}$
is non-decreasing, we get
\begin{equation}\label{eq-order-estimate-P_i}
\pvector_{\pi(\ell)}\leq Ln\sqrt{d}\ln\bigg(\frac{en}{|\supp x|-\ell+1}\bigg),\quad i\leq \ell\leq |\supp x|.
\end{equation}
Moreover, Lemma~\ref{l: elementary psi estimate} implies
$$
\sum_{\ell=i}^{|\supp x|}{\pvector_{\pi(\ell)}}^2
\leq CL^2n^2d\big(|\supp x|-i+1\big)\ln^2\bigg(\frac{en}{|\supp x|-i+1}\bigg),
$$
for a sufficiently large universal constant $C$.
Plugging in the estimate into \eqref{eq1-sigma: small linfty norm}, we get
$$
{\sigma_i}^2\leq \frac{C' d}{n}\sum_{\ell=i}^{|\supp x|}{x_{\pi(\ell)}}^2 +\frac{C'dm}{n^2}
+C' L^2(\Vert x\Vert_\infty^2+{\sigma_i}^2)\frac{m}{n}\ln^2\bigg(\frac{en}{m}\bigg),
$$
for an appropriate constant $C'$, where $m:=|\supp x|-i+1$.
Now, if $C_1$ in~\eqref{eq2: alpha condition} is sufficiently large, the above self-bounding estimate for $\sigma_i$ implies
$$
{\sigma_i}^2\leq 
\frac{2C_1 d}{n}\sum_{\ell=i}^{|\supp x|}{x_{\pi(\ell)}}^2 +\frac{2C_1dm}{n^2}
+2C_1L^2\Vert x\Vert_\infty^2\frac{m}{n}\ln^2\bigg(\frac{en}{m}\bigg).
$$
Using the condition $\Vert x\Vert_\infty\leq \ln(2n)/\sqrt{n}$, the assumption on $d$ given by
\eqref{eq-condition-d-sec-row-concentration} and relation~\eqref{eq2: alpha condition}, 
we obtain
$$
\sigma^2:={\sigma_1}^2\leq C_2\frac{d}{n},
$$
for an appropriate constant $C_2$ and 
$$
{\sigma_i}^2\leq (1+L^2)C_3d\Vert x\Vert_\infty^2\frac{|\supp x|-i+1}{n},\quad i\leq |\supp x|,
$$
for a sufficiently large constant $C_3$.

\medskip

Now, let us turn to estimating the absolute value of $d_i$'s. Again, we fix any $i\leq |\supp x|$
and follow notations of the previous sub-section,
replacing $\ell$ with $\pi(\ell)$ where appropriate.
By Lemmas~\ref{lem-estimate-diff} and~\ref{l: average of hv},
inequality \eqref{eq: delta via pvector} and the above estimate of $\sigma_i$, we have
\begin{align*}
\vert d_i\vert &\leq |x_{\pi(i)}|+\frac{8\Vert x\Vert_1}{n}
+(\Vert x\Vert_\infty+\sigma_i)\frac{8}{n-i} \sum_{\ell=1}^n \delta_{\pi(i),\ell}\\ 
& \leq C_4\Vert x\Vert_\infty \Big[1+ \frac{L}{n\sqrt{d}} \sqrt{\frac{|\supp x|-i+1}{n}} \pvector_{\pi(i)}\Big],
\end{align*}
for some constant $C_4>0$. 
Using first \eqref{eq-order-estimate-P_i} then the relation~\eqref{eq2: alpha condition}, we deduce that 
$$
\vert d_i\vert\leq C_4(1+L)\Vert x\Vert_\infty.
$$
Thus, we get that $M\leq C_4 (1+ L) \|x\|_\infty$.
Finally, we apply Theorem~\ref{th-freedman-moment} with parameters
$M$ and $\sigma$ estimated above.
\end{proof}

Now, we can state the main result of the section.
\begin{theorem}\label{th-concentration-row}
For any $L>0$ there is $\gamma(L)\in(0,1]$ with the following property:
Assume that $\|\pvector\|_{\psi,n}\leq Ln\sqrt{d}$, let $x\in S^{n-1}$, and denote by $\eta$ the random variable
$$\eta=\eta(v):=\langle v,x\rangle-\E_\Omega\langle x,\cdot\rangle,\quad v\in \Omega.$$
Then
$$\E_\Omega\, e^{\gamma\lambda \eta}
\leq \exp\bigg(\frac{d}{n\|x\|_\infty^2}\,\gfunction(\lambda \|x\|_\infty)\bigg),\quad \lambda>0.$$
\end{theorem}
\begin{proof}
Let $\alpha=\alpha(L)\in(0,1)$ be the largest number in $(0,1/4]$ satisfying both
\eqref{eq: alpha condition} and~\eqref{eq2: alpha condition}.
We represent the vector $x$ as a sum
$$x=x^1+x^2+\ldots+x^m,$$
where $x^1,x^2,\ldots,x^m$ are vectors with pairwise disjoint supports such that $|\supp x^j|\leq \alpha n$ ($j\leq m$)
and $m:=\big\lceil n/\lfloor \alpha n\rfloor\big\rceil$. For every $j\leq m$, applying either Lemma~\ref{l: big linfty norm}
or Lemma~\ref{l: small linfty norm} (depending on the $\|\cdot\|_\infty$-norm of $x^j/\|x^j\|_2$), we obtain
$$
\max\big(\E e^{\beta \lambda \eta_j},\E e^{-\beta \lambda \eta_j}\big)
\leq\exp\bigg(\frac{d\Vert x^j\Vert_2^2}{n\|x^j\|_\infty^2} \,\gfunction(\lambda \|x^j\|_\infty) \bigg)\leq 
\exp\bigg(\frac{d}{n\|x\|_\infty^2}\,\gfunction(\lambda \|x\|_\infty)\bigg),\quad \lambda>0,$$
for some $\beta=\beta(L)>0$, where
$$\eta_j:=\langle x^j,v\rangle-\E_\Omega\langle x^j,\cdot\rangle,\quad v\in \Omega.$$
Since $\eta= \eta_1+\eta_2+\ldots+\eta_m$ everywhere on $\Omega$, we get from H\"older's inequality 
$$
\E e^{\beta \lambda \eta}= \E \prod_{j=1}^m e^{\beta \lambda \eta_j}\leq \left(\prod_{j=1}^m \E e^{\beta m\lambda \eta_j}\right)^{\frac{1}{m}} 
\leq \exp\bigg(\frac{d}{n\|x\|_\infty^2}\,\gfunction(\lambda m \|x\|_\infty)\bigg).
$$
The statement follows with $\gamma:=\beta/m$.
\end{proof}

\medskip

The above theorem leaves open the question of estimating the expectation $\E_\Omega\langle \cdot,x\rangle$.
This problem is addressed in the last statement of the section.

\begin{prop}\label{p: expectation}
For any non-zero vector $x\in\R^n$ we have
$$\Bigl|\E_{\Omega}\langle \cdot,x\rangle-\frac{\OutDeg_{m+1}}{n}\sum_{i=1}^n x_i\Bigr|\leq
\frac{C_{\ref{p: expectation}}d\|x\|_1}{n^2}+\frac{C_{\ref{p: expectation}}}{n}
\|x\|_{\log,n}\|\pvector\|_{\psi,n}$$
where $C_{\ref{p: expectation}}>0$ is a sufficiently large universal constant
and $\|\cdot\|_{\log,n}$ is defined by~\eqref{eq: log norm definition}.
\end{prop}
\begin{proof}
Let ${\bf{}V}$ be a random vector distributed on $\Omega$ according to the measure $\Prob_\Omega$.
First, we compare expectations of individual coordinates of ${\bf{}V}$, using Lemma~\ref{lem-prob-v-v'}.
We let $\gamma_{i,j}$ be defined by \eqref{eq: gamma def}.
Recall that according to Remark~\ref{rem: gamma}, we have $1-8c_0\leq \gamma_{i,j}\leq 1+50c_0$.
Take any $i\neq j\leq n$ and define a bijective map $f:\Omega\to\Omega$ as
$$f\bigl((v_1,v_2,\ldots,v_n)\bigr):=(v_{\sigma(1)},v_{\sigma(2)},\ldots,v_{\sigma(n)}),\;\;(v_1,v_2,\ldots,v_n)\in\Omega,$$
where $\sigma$ is the transposition of $i$ and $j$. Then
for any $v\in\Omega$, in view of Lemma~\ref{lem-prob-v-v'}, we have
$$\Prob_\Omega(v)
\leq \max\bigl(\gamma_{i,j},\gamma_{j,i}\bigr)\Prob_\Omega(f(v)).$$
Hence,
\begin{align*}
\E_\Omega {\bf{}V}_i&=\Prob_\Omega\{v\in\Omega:\,v_i=1\}\\
&\leq \max\bigl(\gamma_{i,j},\gamma_{j,i}\bigr)\sum_{v\in\Omega:v_i=1}\Prob_\Omega\bigl(f(v)\bigr)\\
&=\max\bigl(\gamma_{i,j},\gamma_{j,i}\bigr)\Prob_\Omega\{v\in\Omega:\,v_j=1\}\\
&=\max\bigl(\gamma_{i,j},\gamma_{j,i}\bigr)\E_\Omega {\bf{}V}_j.
\end{align*}
Together with an obvious relation $\sum_{i=1}^n\E_\Omega {\bf{}V}_i=\OutDeg_{m+1}$, this implies for any fixed $i\leq n$:
$$\sum_{j=1}^n\max\bigl(\gamma_{i,j},\gamma_{j,i}\bigr)^{-1}\E_\Omega {\bf{}V}_i\leq \OutDeg_{m+1}
\leq \sum_{j=1}^n\max\bigl(\gamma_{i,j},\gamma_{j,i}\bigr)\E_\Omega {\bf{}V}_i,$$
whence
$$\Bigl|\E_{\Omega}{\bf{}V}_i-\frac{\OutDeg_{m+1}}{n}\Bigr|\leq \frac{C\OutDeg_{m+1}}{n}\Bigl(\frac{1}{n}\sum_{j=1}^n\delta_{i,j}\Bigr),$$
where $\delta_{i,j}$ are defined by \eqref{eq: delta def} and $C>0$ is a universal constant.

Thus, for any non-zero vector $x=(x_1,x_2,\ldots,x_n)$ we get, in view of \eqref{eq: delta via pvector},
\begin{align*}
\Bigl|\E_{\Omega}\langle {\bf{}V},x\rangle-\frac{\OutDeg_{m+1}}{n}\sum_{i=1}^n x_i\Bigr|
&\leq \frac{C\OutDeg_{m+1}}{n}\sum_{i=1}^n |x_i|\Bigl(\frac{1}{n}\sum_{j=1}^n
\delta_{i,j}\Bigr)\\
&\leq  \frac{C'\OutDeg_{m+1}}{n}
\sum_{i=1}^n |x_i|\Bigl(\frac{1}{nd}\pvector_i+\frac{1}{n}\Bigr)\\
&=\frac{C'\OutDeg_{m+1}\|x\|_1}{n^2}+\frac{C'\OutDeg_{m+1}}{n^2d}\sum_{i=1}^n |x_i|\pvector_i,
\end{align*}
where $C'$ is a universal constant.
Finally, applying Fenchel's inequality to the sum on the right hand side and using the definition of the Orlicz norms
$\|\cdot\|_{\psi,n}$ and $\|\cdot\|_{\log,n}$, we obtain
\begin{align*}
\sum_{i=1}^n |x_i|\pvector_i
&=\|x\|_{\log,n}\|\pvector\|_{\psi,n}\sum_{i=1}^n \frac{|x_i|}{\|x\|_{\log,n}}\,\frac{\pvector_i}{\|\pvector\|_{\psi,n}}\\
&\leq \|x\|_{\log,n}\|\pvector\|_{\psi,n}\sum_{i=1}^n \bigg(\frac{|x_i|}{\|x\|_{\log,n}}\ln_+\Big(\frac{|x_i|}{\|x\|_{\log,n}}\Big)
+\exp\big(\pvector_i/\|\pvector\|_{\psi,n}\big)\bigg)\\
&\leq (e+1)n\,\|x\|_{\log,n}\|\pvector\|_{\psi,n}.
\end{align*}
The result follows.
\end{proof}

\section{Tensorization}\label{s: tensorization}

The goal of this section is to transfer the concentration inequality for a single row obtained
in the previous section (Theorem~\ref{th-concentration-row}) to the whole matrix.
Throughout the section, we assume that the degree sequences $\InDeg$, $\OutDeg$
satisfy \eqref{eq: degree condition} for some $d$, and that $d$ itself satisfies
\eqref{eq-condition-d-sec-row-concentration}.
Moreover, we always assume that the set of matrices $\MatrixSet(\InDeg,\OutDeg)$ is non-empty.
It will be convenient to introduce in this section a ``global'' random object --- a matrix $\ranmtx$
uniformly distributed on $\MatrixSet(\InDeg,\OutDeg)$.

Let $G$ be a directed graph on $n$ vertices
with degree sequences $\InDeg$, $\OutDeg$, and let $M=(M_{ij})$ be the adjacency matrix of $G$.
Next, let $I$ be a subset of $[n]$ (possibly, empty).
We define quantities $p_j^{col}(I,M)$, $p_j^{row}(I,M)$ ($j\leq n$) as in the Introduction
(let us repeat the definition here for convenience):
\begin{align*}
p_j^{col}(I,M)&:=\InDeg_j-|\{q\in I:\,M_{qj}=1\}|=|\{q\in I^c:\,M_{qj}=1\}|;\\
p_j^{row}(I,M)&:=\OutDeg_j-|\{q\in I:\,M_{jq}=1\}|=|\{q\in I^c:\,M_{jq}=1\}|.
\end{align*}
Again, we define vectors $\pvector^{col}(I,M)$, $\pvector^{row}(I,M)\in\R^n$ coordinate-wise as
\begin{align*}
\pvector^{col}_j(I,M)&:=\sum_{\ell=1}^n |p_j^{col}(I,M)-p_\ell^{col}(I,M)|;\\
\pvector^{row}_j(I,M)&:=\sum_{\ell=1}^n |p_j^{row}(I,M)-p_\ell^{row}(I,M)|.
\end{align*}
Clearly, these objects are close relatives of the quantities $p_j$ and the vector $\pvector$
defined in the previous section. In fact, if $\widetilde \MatrixSet$ is the subset of all matrices from $\MatrixSet(\InDeg, \OutDeg)$
with a fixed realization of rows from $I$ then $p_j^{col}(I,\cdot)$ ($j\leq n$) and
$\pvector^{col}(I,\cdot)$ are constants on $\widetilde \MatrixSet$, which, up to relabelling the graph vertices,
correspond to $p_j$'s and $\pvector$ from Section~\ref{s: row concentr}.

Note that Theorem~\ref{th-concentration-row} operates under assumption that the vector $\pvector$,
or, in context of this section, random vectors $\pvector^{col}(I,\ranmtx)$ for appropriate subsets $I$, have small magnitude
in $\|\cdot\|_{\psi,n}$-norm --- the fact which still needs to be established.
For any $L>0$, let $\Event_\pvector(L)$ be given by \eqref{eq: event pvector definition}, i.e.\
\begin{equation*}
\begin{split}
\Event_\pvector(L)=\Big\{
&\|\pvector^{row}(I,\ranmtx)\|_{\psi,n},\|\pvector^{col}(I,\ranmtx)\|_{\psi,n}\leq  L  n\sqrt{d}\\
&\mbox{for any interval subset $I\subset[n]$ of cardinality at most $c_0 n$}\Big\}.\end{split}
\end{equation*}
To make Theorem~\ref{th-concentration-row} useful, we need to show that for some appropriately chosen
parameter $L$ the event $\Event_\pvector(L)$ has probability close to one.
Obviously, this will require much stronger assumptions on the degree sequences
than ones we employed up to this point. But, even under the stronger assumptions on $\InDeg,\OutDeg$,
proving an upper estimate for $\|\pvector^{col}(I,\ranmtx)\|_{\psi,n}$,
$\|\pvector^{row}(I,\ranmtx)\|_{\psi,n}$ will require us to use the concentration results
from Section~\ref{s: row concentr}. In order not to create a vicious cycle, we will argue in the following manner:
First, we apply Theorem~\ref{th-concentration-row} in the situation when the set $I$ has very small cardinality.
It can be shown that in this case we get the required assumptions on $\|\pvector^{col}(I,\ranmtx)\|_{\psi,n}$
for free, as long as the degree sequences satisfy certain additional conditions. This, in turn,
will allow us to establish the required bounds for $\|\pvector^{col}(I,\ranmtx)\|_{\psi,n}$
for ``large'' subsets $I$. Finally, having this result in possession, we will be able to use the full strength of Theorem~\ref{th-concentration-row}
and complete the tensorization.

Let us note that condition $\|\pvector^{col}(I,M)\|_{\psi,n}=O(n\sqrt{d})$ for a matrix $M\in\MatrixSet(\InDeg,\OutDeg)$
and a subset $I$ of cardinality at most $c_0 n$ automatically implies an analog of condition \eqref{eq: pl strong},
as long as $n$ is sufficiently large. To be more precise, we have the following
\begin{lemma}\label{l: analog of pl strong}
There is a universal constant $c_{\ref{l: analog of pl strong}}>0$ with the following property:
Assume that for some matrix $M\in\MatrixSet(\InDeg,\OutDeg)$ and $I\subset[n]$ with $|I|\leq c_0 n$ we have
$$\|\pvector^{col}(I,M)\|_{\psi,n}\leq c_{\ref{l: analog of pl strong}}nd/\ln n.$$
Then necessarily
$$\InDeg_j\geq p_j^{col}(I,M)\geq (1-2c_0)\InDeg_j$$
for all $j\leq n$.
\end{lemma}
\begin{proof}
Assume that $p_i^{col}(I,M)<(1-2c_0)\InDeg_i$ for some $i\leq n$.
Define
$$J:=\big\{j\leq n: p_j^{col}(I,M)<(1-1.5c_0)\InDeg_j\big\}.$$
Then, obviously,
$$\big|\big\{(k,\ell)\in I\times [n]:\,M_{k\ell}=1\big\}\big|\geq 1.5c_0\sum_{j\in J}\InDeg_j\geq 1.4c_0d|J|.$$
On the other hand,
$$\big|\big\{(k,\ell)\in I\times [n]:\,M_{k\ell}=1\big\}\big|=\sum_{k\in I}\OutDeg_k\leq c_0 nd.$$
Thus, $|J|\leq \frac{5}{7}n$. This implies that
$$\pvector_i^{col}(I,M)\geq \sum_{k\in J^c} \big|p_i^{col}(I,M)-p_k^{col}(I,M)\big|> c_0|J^c|d/4\geq c_0nd/14.$$
Hence, by \eqref{eq: psi infinity}, we get
$$\|\pvector^{col}(I,M)\|_{\psi,n}> \frac{c_0 nd}{14\ln(en)}.$$
The result follows.
\end{proof}

The above lemma allows us not to worry about condition \eqref{eq: pl strong}
and focus our attention on the $\|\cdot\|_{\psi,n}$-norm of vectors $\pvector^{col}(I,\ranmtx)$.
The bounds for $\|\pvector^{col}(I,\ranmtx)\|_{\psi,n}$ are obtained in Proposition~\ref{p: bounds for pvector}.
But first we need to consider two auxiliary statements.

\begin{lemma}\label{l: concentration sum p-l}
For any $L>0$ there are $\gamma(L)\in(0,1]$ and $K=K(L)>0$ such that the following holds. 
Let the degree sequences $\InDeg$ and $\OutDeg$
be such that $\bigl\|(\InDeg_i-d)_{i=1}^n\bigr\|_{\psi,n},\bigl\|(\OutDeg_i-d)_{i=1}^n\bigr\|_{\psi,n}\leq L\sqrt{d}$,
where $\|\cdot\|_{\psi,n}$ is defined by \eqref{eq: psi definition}.
Further, let $J\subset[n]$ be a subset of cardinality $\sqrt{d}/2\leq |J|\leq \sqrt{d}$, and let $I\subset [n]$
be any non-empty subset. Define a $|J|$-dimensional random vector in $\R^J$ as
$$v(I):=(v_k)_{k\in J},\;\;\;
v_k:=\big|p^{row}_k(I,\ranmtx)-\frac{\OutDeg_k|I^c|}{n}\big|,\;k\in J.$$
Then for any subset $T\subset J$ and any $t \geq K \sqrt{d}\, \vert T\vert$, we have 
$$
\Prob\Big\{\sum_{k\in T} v_k \geq t \Big\} 
\leq \exp\left(-t \gamma\ln \Big(1+\frac{t \gamma n}{d\vert I\vert\, \vert T\vert}\Big)\right).
$$
\end{lemma}
\begin{proof}
Denote
$$x^I:=|I|^{-1/2}\sum_{i\in I}e_i.$$
To simplify the notation, let us assume that $J=\{1,\ldots, \vert J\vert\}$
(we can permute the degree sequence $\OutDeg$ accordingly).  
Take any matrix $M\in \MatrixSet(\InDeg,\OutDeg)$.
Note that, by the assumption on the cardinality of $J$, we have
$$\big|p_\ell^{col}([k],M)-p_{\ell'}^{col}([k],M)\big|
-\big|\InDeg_{\ell}-\InDeg_{\ell'}\big|\leq \sqrt{d},\;\;\;k< |J|,\;\;\ell,\ell'\leq n.$$
Hence, for any $j\leq n$, $k\leq|J|$ we have
$$\pvector^{col}_j([k],M)
\leq n\sqrt{d}+\sum_{\ell=1}^n |\InDeg_\ell-d|+n|\InDeg_j-d|.
$$
Note that $\Vert \cdot\Vert_1\leq n\, \Vert\cdot \Vert_{\psi,n}$
by convexity of $\exp(\cdot)$.
Then, in view of the assumptions on $\bigl\|(\InDeg_i-d)_{i=1}^n\bigr\|_{\psi,n}$, we get 
$$
\pvector^{col}_j([k],M)
\leq (1+L)n\sqrt{d}+n|\InDeg_j-d|.$$
Thus, by the triangle inequality,
$$\|\pvector^{col}([k],M)\|_{\psi,n}\leq
(1+L)n\sqrt{d}\|(1,1,\ldots,1)\|_{\psi,n}+n\bigl\|(\InDeg_i-d)_{i=1}^n\bigr\|_{\psi,n}\leq
 (L+2)n\sqrt{d}$$
for any $k\leq|J|$ and $M\in\MatrixSet(\InDeg,\OutDeg)$. 
For every $k\leq\vert J\vert$, we denote by $\eta_k$ the random variable 
$$
\eta_k:= \big\vert\langle \Row_k(\ranmtx),x^I\rangle -\E\big[\langle \Row_k(\cdot),x^I\rangle\,|\,
\Row_j(\ranmtx),\,j\leq k-1\big]\big\vert.
$$
In view of the above estimate of $\|\pvector^{col}([k],M)\|_{\psi,n}$ and
Theorem~\ref{th-concentration-row}, there is $\gamma'(L)\in (0,1)$ such that for any $\lambda>0$
we have
$$
\E \left[e^{\gamma' \lambda\, \sqrt{\vert I\vert} \eta_k}\mid \Row_j(\ranmtx),\,j\leq {k-1}\right]
\leq 2\exp\bigg(\frac{d\vert I\vert}{n}\,\gfunction(\lambda)\bigg),
$$
for every $k\leq\vert J\vert$ (recall that $\|x^I\|_\infty=|I|^{-1/2}$). 
Further, for any $k\leq\vert J\vert$ we have
$$\frac{1}{\sqrt{\vert I\vert}} v_k=\frac{1}{\sqrt{\vert I\vert}}\big\vert p_k^{row}(I,\ranmtx)-\frac{\OutDeg_k|I^c|}{n}\big\vert
=\big\vert\langle \Row_k(\ranmtx),x^I\rangle -\frac{\OutDeg_k\sqrt{|I|}}{n}\big\vert.
$$
Thus, using Proposition~\ref{p: expectation} and Lemma~\ref{l: elementary log estimate}, we get
\begin{align*}
v_k&\leq \sqrt{\vert I\vert} \eta_k+\sqrt{|I|}\;\big|\E\big[\langle \Row_k(\cdot),x^I\rangle\,|\,
\Row_j(\ranmtx),\,j\leq k-1\big]-\frac{\OutDeg_k\sqrt{|I|}}{n}\big|\\
&\leq \sqrt{\vert I\vert} \eta_k+ \frac{\mu\sqrt{d}\vert I\vert}{n}\, \ln\left(\frac{2n}{\vert I\vert}\right)
\end{align*}
for some $\mu=\mu(L)\geq 1$.
Hence, for any $k\leq|J|$ and any $\lambda>0$ we have 
$$
\E \left[e^{\gamma' \lambda\,v_k}\mid \Row_j(\ranmtx),\,j\leq {k-1}\right]
\leq 2\exp\bigg(\frac{d\vert I\vert}{n}\gfunction(\lambda )+ \gamma'\lambda\frac{\mu\sqrt{d}\vert I\vert}{n}
\, \ln\Big(\frac{2n}{\vert I\vert}\Big)\bigg).
$$
By Lemma~\ref{lem-freedman-tensorization}, this
implies that for any subset $T\subset J$ and any $\lambda>0$ we have
$$
\E \, e^{\gamma' \lambda\,  \sum_{k\in T} v_k}
\leq 2^{|T|}\exp\left[\vert T\vert \bigg(\frac{d\vert I\vert}{n}\gfunction(\lambda )
+ \gamma'\lambda  \frac{\mu\sqrt{d}\vert I\vert}{n}\, \ln\Big(\frac{2n}{\vert I\vert}\Big)\bigg)\right].
$$
Now, fix any $t\geq 4\mu\sqrt{d}|T|$. By the above estimate for the moment generation function and Markov's inequality, we get
\begin{align*}
\Prob\Big\{\sum_{k\in T} v_k \geq t \Big\} 
&\leq \exp\left[-\gamma' \lambda t+ |T|
+\frac{d\vert I\vert|T|}{n}\gfunction(\lambda )
+\gamma'\lambda|T|  \frac{\mu\sqrt{d}\vert I\vert}{n}\, \ln\Big(\frac{2n}{\vert I\vert}\Big)\right]\\
&\leq \exp\left[-\frac{1}{2}\gamma' \lambda t+ |T|
+\frac{d\vert I\vert|T|}{n}\gfunction(\lambda )\right]
\end{align*}
for any $\lambda>0$. It is easy to see that the last espression is minimized for
$\lambda:=\ln\big(1+\frac{\gamma' nt}{2d|I| |T|}\big)$.
Plugging in the value of $\lambda$ into the exponent, we get
\begin{align*}
\Prob\Big\{\sum_{k\in T} v_k \geq t \Big\} 
&\leq\exp\left[|T|+\frac{1}{2}\gamma' t-\frac{1}{2}\gamma'\lambda t-\frac{d|I||T|}{n}\lambda\right]\\
&=\exp\left[|T|+\frac{d|I||T|}{n}\Big(\frac{\gamma' n t}{2d|I||T|}-\frac{\gamma' n t}{2d|I||T|}\lambda-\lambda\Big)\right]\\
&=\exp\left[|T|-\frac{d|I||T|}{n} H\Big(\frac{\gamma' n t}{2d|I||T|}\Big)\right],
\end{align*}
where the function $H$ is defined by \eqref{eq: H definition}.
Finally, applying the relation \eqref{eq-min-h}, we get
that for a large enough $K=K(L)$ and all $t\geq K\sqrt{d}|T|$ we have
$$|T|-\frac{d|I||T|}{n} H\Big(\frac{\gamma' n t}{2d|I||T|}\Big)\leq -\frac{d|I||T|}{2n} H\Big(\frac{\gamma' n t}{2d|I||T|}\Big).$$
The result follows.
\end{proof}

\begin{lemma}\label{l: pl on sqrt d scale}
Let $a\in (0,1)$ and suppose that $n^a\leq d$.
Further, let the degree sequences $\InDeg$ and $\OutDeg$,
the subset $J$, the random vectors $v(I)\in\R^J$ and the parameters $L$ and $\gamma(L),K(L)$
be the same as in Lemma~\ref{l: concentration sum p-l}.
Then for a sufficiently large universal constant $C_{\ref{l: pl on sqrt d scale}}$ we have
\begin{align*}
\Prob\Big\{\Vert v(I)\Vert_{\psi, \vert J\vert} \leq \frac{C_{\ref{l: pl on sqrt d scale}}K\sqrt{d}}{\gamma a}\;\;\mbox{for any interval}&\\
\mbox{subset $I\subset[n]$ of cardinality at most $c_0 n$}&\Big\}\geq 1-\frac{1}{n}.
\end{align*}
\end{lemma}

\begin{proof}
Let $C_{\ref{l: pl on sqrt d scale}}$ be a sufficiently large constant (its value can be recovered from the proof below).
Further, let $I\subset [n]$ be a fixed interval subset of $[n]$ of size at most $c_0n$.
In view of Lemma~\ref{l: psi lower bound}, for any vector $x\in\R^J$ with
$\|x\|_{\psi,|J|}\geq C_{\ref{l: pl on sqrt d scale}}\gamma^{-1}a^{-1}\sqrt{d}$
there is a natural $t\leq 2\ln(e|J|)$ such that
$$\Big|\Big\{i\in J:\,|x_i|\geq \frac{C_{\ref{l: pl on sqrt d scale}}t\sqrt{d}}{2\gamma a}\Big\}\Big|\geq |J|(2e)^{-t}.$$
In particular, we can write
\begin{align*}
p:=\Prob \Big\{\Vert v(I)\Vert_{\psi, \vert J\vert}\geq \frac{C_{\ref{l: pl on sqrt d scale}}\sqrt{d}}{\gamma a}\Big\}
\leq \sum_{t=1}^{\lfloor 2\ln(e|J|)\rfloor}
\sum_{\substack{T\subset J,\\ \vert T\vert = \lceil \vert J\vert\, (2e)^{-t}\rceil}}  
\Prob\Big\{\forall k\in T,\, v_k(I) \geq \frac{C_{\ref{l: pl on sqrt d scale}}t\sqrt{d}}{2\gamma a}\Big\}.
\end{align*}
Then, applying Lemma~\ref{l: concentration sum p-l}, we get 
\begin{align*}
p&\leq \sum_{t=1}^{\lfloor 2\ln(e|J|)\rfloor} \sum_{\substack{T\subset J,\\ \vert T\vert = \lceil \vert J\vert\, (2e)^{-t}\rceil}} 
\exp\left[-\frac{C_{\ref{l: pl on sqrt d scale}}t\sqrt{d}\vert T\vert}{2a} \, 
\ln \left(1+\frac{C_{\ref{l: pl on sqrt d scale}} \, n\, t}{2a\sqrt{d}\vert I\vert}\right)\right]\\
&\leq  \sum_{t=1}^{\lfloor 2\ln(e|J|)\rfloor}
\exp\left[4\lceil |J|(2e)^{-t}\rceil \,t -\frac{C_{\ref{l: pl on sqrt d scale}}t\sqrt{d}\lceil |J|(2e)^{-t}\rceil}{2a} \, 
\ln \left(1+\frac{C_{\ref{l: pl on sqrt d scale}} \, n\, t}{2a\sqrt{d}\vert I\vert}\right)\right].
\end{align*}
Now using that $\ln\left(1+\frac{C_{\ref{l: pl on sqrt d scale}} \, n\, t}{2a\sqrt{d}\vert I\vert}\right)
\gg \frac{t}{\sqrt{d}}\geq \frac{1}{\sqrt{d}}$ for any $t$ in the above sum, we get 
\begin{align*}
p&\leq  \sum_{t=1}^{\lfloor 2\ln(e|J|)\rfloor} \exp\left(-\frac{C_{\ref{l: pl on sqrt d scale}}t\lceil |J|(2e)^{-t}\rceil}{4a}\right)\\
&\leq \lfloor 2\ln(e|J|)\rfloor
\max_{t=1,\ldots, \lfloor 2\ln(e|J|)\rfloor} \exp\left(-\frac{C_{\ref{l: pl on sqrt d scale}}t |J|(2e)^{-t}}{4a}\right)
\ll \frac{1}{n^3},
\end{align*}
where the last inequality follows from the lower bound on $d$
and the choice of $C_{\ref{l: pl on sqrt d scale}}$. 
It remains to apply the union bound over all interval subsets (of which there are $O(n^2)$) to finish the proof.
\end{proof}

\begin{Remark}\label{r: pvector probability power}
It is easy to see from the proof that the probability estimate $1-n^{-1}$ in the lemma
can be replaced with $1-n^{-m}$ for any $m>0$ at the expense of replacing
$C_{\ref{l: pl on sqrt d scale}}$ by a larger constant. 
\end{Remark}

As a consequence of the above, we obtain

\begin{prop}\label{p: bounds for pvector}
For any parameters $a\in(0,1)$, $L\geq 1$ and $m\in\N$ there is $n_0=n_0(a,m,L)$ and
$\widetilde L=\widetilde L(L,m)$ (i.e.\ $\widetilde L$ depends only on $L$ and $m$) with the following property:
Let $n\geq n_0$, $n^a\leq d$ and let the degree sequences $\InDeg$ and $\OutDeg$
be such that
$\bigl\|(\InDeg_i-d)_{i=1}^n\bigr\|_{\psi,n},\bigl\|(\OutDeg_i-d)_{i=1}^n\bigr\|_{\psi,n}\leq L\sqrt{d}$. 
Then the event $\Event_\pvector(a^{-1}\widetilde L)$
(defined by formula \eqref{eq: event pvector definition})
has probability at least $1-n^{-m}$.
\end{prop}
\begin{proof}
Let us partition $[n]$ into at most $2\sqrt{d}$ subsets $J_1,J_2,\dots,J_r$ ($r\leq 2\sqrt{d}$),
where each $J_j$ satisfies $\sqrt{d}/2\leq |J_j|\leq \sqrt{d}$.
For any $j\leq r$, in view of Lemma~\ref{l: pl on sqrt d scale}, with
probability at least $1-n^{-m-2}$ the $|J_j|$-dimensional vector
$$v^j(I)=(v_k^j)_{k\in J_j},\;\;\;
v_k^j:=\big|p^{row}_k(I,\ranmtx)-\frac{\OutDeg_k|I^c|}{n}\big|,\;k\in J_j,$$
satisfies $\|v^j(I)\|_{\psi, \vert J^j\vert} \leq K'a^{-1}\sqrt{d}$ for some $K'=K'(m, L)\geq 1$ for any 
interval subset $I\subset[n]$ of cardinality at most $c_0 n$. 
Hence, with probability at least $1-n^{-m-1}$, the concatenated $n$-dimensional vector
$$v(I)=(v_1,v_2,\dots,v_n),\;\;v_k=v_k^j\;\;\mbox{ for any $j\leq r$ and $k\in J_j$}$$
satisfies $\|v(I)\|_{\psi,n}\leq K' a^{-1}\sqrt{d}$ for any 
interval subset $I\subset[n]$ of cardinality at most $c_0 n$. 
Next, note that for any $k\leq n$ and any $I\subset [n]$ we have
\begin{align*}
nv_k&=n\big|p^{row}_k(I,\ranmtx)-\frac{\OutDeg_k|I^c|}{n}\big|\\
&\geq \sum_{i=1}^n \big|p^{row}_k(I,\ranmtx)-p^{row}_i(I,\ranmtx)\big|
-\sum_{i=1}^n \big|p^{row}_i(I,\ranmtx)-\frac{\OutDeg_i|I^c|}{n}\big|\\
&\hspace{0.5cm}-\sum_{i=1}^n \big|\frac{\OutDeg_i|I^c|}{n}-\frac{d|I^c|}{n}\big|
-|I^c|\big|\OutDeg_k-d\big|\\
&\geq \pvector^{row}_k(I,\ranmtx) -\sum_{i=1}^n v_i-\sum_{i=1}^n\big|\OutDeg_i-d\big|
-n\big|\OutDeg_k-d\big|.
\end{align*}
Hence, in view of the convexity of $\exp(\cdot)$, we get
$$\pvector^{row}_k(I,\ranmtx)\leq n v_k+n\big|\OutDeg_k-d\big|
+n\big\|(\OutDeg_i-d)_{i=1}^n\big\|_{\psi,n}
+n \|v(I)\|_{\psi,n},
$$
which implies that
$$\|\pvector^{row}(I,\ranmtx)\|_{\psi,n}\leq 2n\big\|(\OutDeg_i-d)_{i=1}^n\big\|_{\psi,n}
+2n\|v(I)\|_{\psi,n}.$$
Therefore, with probability at least $1-n^{-m-1}$, we have $\|\pvector^{row}(I,\ranmtx)\|_{\psi,n}\leq \widetilde La^{-1}n\sqrt{d}$ for any 
interval subset $I\subset[n]$ of cardinality at most $c_0 n$ and
$\widetilde L:=2K'+2L$.
Clearly, the same estimate holds for $\pvector^{col}(I,\ranmtx)$ and the proof is complete.
\end{proof}

\bigskip

Let us introduce a family of random variables on the probability space $(\MatrixSet(\InDeg,\OutDeg),\Prob)$
as follows.
Take any index $i\leq n$ and any subset $I\subset[n]$ not containing $i$.
Further, let $x\in\R^n$ be any vector.
Then we define $\theta(i,I,x):\MatrixSet(\InDeg,\OutDeg)\to\R$ as
$$\theta(i,I,x):=\E\big[\langle \Row_i(\ranmtx),x\rangle\mid\Row_j(\ranmtx),\;j\in I\big].$$
In other words, $\theta(i,I,x)$ is the conditional expectation of $\langle \Row_i(\ranmtx),x\rangle$,
conditioned on realizations of rows $\Row_j(\ranmtx)$ ($j\in I$).

\begin{lemma}\label{l: main concentration light}
Let $L>0$ be some parameter and let the event $\Event_{\pvector}(L)$
be defined by \eqref{eq: event pvector definition}.
Let $I$ be any non-empty interval subset of $[n]$ of length at most $c_0 n$ and let $Q=(Q_{ij})$ be a fixed $n\times n$
matrix with all entries with indices outside $I\times[n]$ equal to zero.
Then for any $t>0$ we have
\begin{align*}
\Prob&\Big\{\Big|\sum_{i\in I}\Big(\sum_{j=1}^n{\ranmtx}_{ij}Q_{ij}-\theta\big(i,\{\inf I,\dots, i-1\},\Row_i(Q)\big)\Big)\Big|
>t\,\mid\,{\ranmtx}\in\Event_{\pvector}(L)
\Big\}\\
&\leq \frac{2}{\Prob(\Event_\pvector(L))}\exp\left(-\frac{d\,\|Q\|_{HS}^2}{n\,\|Q\|_\infty^2}
\, H\left(\frac{\gamma tn\|Q\|_\infty}{d\|Q\|_{HS}^2}\right)\right),
\end{align*}
where $\gamma=\gamma(L)$ is taken from Theorem~\ref{th-concentration-row}.
\end{lemma}
\begin{proof}
Fix for a moment any $i\in I$ and let
$$\Event_i:=\big\{\|\pvector^{col}(\{\inf I,\dots, i-1\},\ranmtx)\|_{\psi,n}\leq   L  n\sqrt{d}\big\}.$$
Further, denote by
$\eta_i$ the random variable
$$\eta_i:=\bigg[\sum_{j=1}^n{\ranmtx}_{ij}Q_{ij}-\theta\big(i,\{\inf I,\dots, i-1\},\Row_i(Q)\big)\bigg]\,
\chi_{i},$$
where $\chi_{i}$ is the indicator function of the event $\ranmtx \in\Event_i$.
Note that $\|\pvector^{col}(\{\inf I,\dots, i-1\},\ranmtx)\|_{\psi,n}$ is uniquely determined by
realizations of $\Row_{\inf I}(\ranmtx),\dots,\Row_{i-1}(\ranmtx)$.
Now, assume that $Y_j$ ($j=\inf I,\dots,i-1$) is any realization of rows $\Row_j(\ranmtx)$ 
($j=\inf I,\dots,i-1$) such that, conditioned on this realization, $\ranmtx$ belongs to $\Event_i$. That is,
$$\big\{\Row_j(\ranmtx)=Y_j,\;\;j=\inf I,\dots,i-1\big\}\subset \Event_i.$$
Then, applying Theorem~\ref{th-concentration-row}, we obtain
\begin{align*}
\E\big[&e^{\gamma\lambda\sum_{j=1}^n{\ranmtx}_{ij}Q_{ij}-\gamma\lambda\theta(i,\{\inf I,\dots, i-1\},\Row_i(Q))}\mid
\Row_j(\ranmtx)=Y_j,\,j=\inf I,\dots,i-1\big]\\
&\leq \exp\bigg(\frac{d\,\|\Row_i(Q)\|^2}{n\max_{j\leq n}{Q_{ij}}^2}\,\gfunction(\lambda \max_{j\leq n}{Q_{ij}})\bigg),
\;\;\lambda>0,
\end{align*}
for some $\gamma=\gamma(L)$.
Note that the value of $\eta_j$ is uniquely determined by realizations of rows $\Row_k(\ranmtx)$ ($k\leq j$).
Hence, in view of the definition of $\eta_i$, we get from the last relation
$$\E\big[e^{\lambda\eta_i}\mid
\eta_j,\;j=\inf I,\dots,i-1\big]
\leq \exp\bigg(\frac{d\,\|\Row_i(Q)\|^2}{n\max_{j\leq n}{Q_{ij}}^2}\,\gfunction(\lambda \gamma^{-1}\max_{j\leq n}{Q_{ij}})\bigg),
\;\;\lambda>0.$$
Now, let
$$\eta:=\sum_{i\in I}\eta_i.$$
By the above inequality and by Corollary~\ref{lem-freedman-tensorization}, we get
$$\Prob\big\{\eta\geq t\big\}\leq
\exp\left(-\frac{d\,\|Q\|_{HS}^2}{n\,\|Q\|_\infty^2}\, H\left(\frac{\gamma tn\|Q\|_\infty}{d\|Q\|_{HS}^2}\right)\right),\quad t>0.$$
Finally, note that
$$\Event_\pvector(L)\subset\bigcap_{i\in I}\Event_i,$$
whence, restricted to $\Event_\pvector(L)$, the variable $\eta$ is equal to
$$\sum_{i\in I}\Big(\sum_{j=1}^n{\ranmtx}_{ij}Q_{ij}-\theta\big(i,\{\inf I,\dots, i-1\},\Row_i(Q)\big)\Big).$$
It follows that
\begin{align*}
\Prob&\Big\{\sum_{i\in I}\Big(\sum_{j=1}^n{\ranmtx}_{ij}Q_{ij}-\theta\big(i,\{\inf I,\dots, i-1\},\Row_i(Q)\big)\Big)\geq t
\mid\, \ranmtx\in\Event_\pvector(L)\Big\}\\
&\leq\frac{1}{\Prob(\Event_\pvector(L))}
\exp\left(-\frac{d\,\|Q\|_{HS}^2}{n\,\|Q\|_\infty^2}\, H\left(\frac{\gamma tn\|Q\|_\infty}{d\|Q\|_{HS}^2}\right)\right),\quad t>0.
\end{align*}
Applying a similar argument to the variable $-\eta$, we get the result.
\end{proof}

The next lemma allow us to replace the variables $\theta\big(i,\{\inf I,\dots, i-1\},\Row_i(Q)\big)$
with constants.
\begin{lemma}\label{l: theta to constants}
For any $L\geq 1$ there is $n_0=n_0(L)$ with the following property.
Let $n\geq n_0$,
let $I$ be any non-empty interval subset of $[n]$ of length at most $c_0 n$ and let $Q=(Q_{ij})$ be a fixed $n\times n$
matrix with all entries with indices outside $I\times[n]$ equal to zero.
Then
\begin{align*}
\Big|\sum_{i\in I}\theta\big(i,\{\inf I,\dots, i-1\},\Row_i(Q)\big)
-\sum_{i\in I}\frac{\OutDeg_{i}}{n}\sum_{j=1}^n Q_{ij}\Big|
\leq C_{\ref{l: theta to constants}}L\sqrt{d}\,
\sum_{i\in I}\|\Row_i(Q)\|_{\log,n}
\end{align*}
everywhere on $\Event_\pvector(L)$. Here, $C_{\ref{l: theta to constants}}>0$ is a universal constant.
\end{lemma}
\begin{proof}
In view of the relation $\|\cdot\|_1\leq en\|\cdot\|_{\log,n}$ which follows from convexity of the function $t\ln_+(t)$,
it is enough to show that for any $i\in I$ we have
\begin{align*}
\Big|&\theta\big(i,\{\inf I,\dots, i-1\},\Row_i(Q)\big)
-\frac{\OutDeg_{i}}{n}\sum_{j=1}^n Q_{ij}\Big|\\
&\leq \frac{C \sqrt{d}}{n}\|\Row_i(Q)\|_1+\frac{C}{n}
\|\Row_i(Q)\|_{\log,n}\|\pvector^{col}(\{\inf I,\dots,i-1\},\ranmtx)\|_{\psi,n}
\end{align*}
everywhere on $\Event_\pvector(L)$ for a sufficiently large constant $C>0$.
But this follows immediately from Proposition~\ref{p: expectation}.
\end{proof}

Finally, we can prove the main technical result of the paper.
To make the statement self-contained, we explicitly mention all the assumptions on parameters.
Given an $n\times n$ matrix $Q$, we define {\it the shift} $\Delta(Q)$ as
$$\Delta(Q):=\sqrt{d}\,\sum_{i=1}^n\|\Row_i(Q)\|_{\log,n}.$$
\begin{theorem}\label{th: main concentration}
For any $L\geq 1$ there are $\gamma=\gamma(L)>0$ and $n_0=n_0(L)$ with the following properties.
Assume that $n\geq n_0$ and that the degree sequences $\InDeg,\OutDeg$ satisfy
$$(1-c_0)d\leq \InDeg_i,\OutDeg_i\leq d,\quad i\leq n$$
for some natural $d$ with $C_1\ln^2 n\leq d\leq (1/2+c_0)n$.
Further, assume that the set $\MatrixSet(\InDeg,\OutDeg)$ is non-empty.
Then, with $\Event_\pvector(L)$ defined by \eqref{eq: event pvector definition},
we have for any $n\times n$ matrix $Q$:
\begin{align*}
\Prob&\Big\{\Big|\sum_{i=1}^{n}\sum_{j=1}^n
{\ranmtx}_{ij}Q_{ij}-\sum_{i=1}^n\frac{\OutDeg_{i}}{n}\sum_{j=1}^n Q_{ij}\Big|
>t+C_2L\Delta(Q)\,\mid\,{\ranmtx}\in\Event_{\pvector}(L)
\Big\}\\
&\leq \frac{C_3}{\Prob(\Event_\pvector(L))}\exp\left(-\frac{d\,\|Q\|_{HS}^2}{n\,\|Q\|_\infty^2}
\, H\left(\frac{\gamma tn\|Q\|_\infty}{d\|Q\|_{HS}^2}\right)\right),\quad t>0.
\end{align*}
Here, $C_1,C_2,C_3>0$ are sufficiently large universal constants.
\end{theorem}
\begin{proof}
Let us partition $[n]$ into $\lceil 2/c_0\rceil$ interval subsets $I_j$ ($j\leq \lceil 2/c_0\rceil$),
with each $I_j$ of cardinality at most $c_0n$. Further, define $n\times n$ matrices $Q^j$ ($j\leq \lceil 2/c_0\rceil$) as
$$Q^j_{k,\ell}:=\begin{cases}Q_{k\ell},&\mbox{if }k\in I^j;\\ 0,&\mbox{otherwise.}\end{cases}$$
Note that each $Q^j$ satisfies assumptions of both Lemma~\ref{l: main concentration light} and
Lemma~\ref{l: theta to constants}. Combining the lemmas, we get
\begin{align*}
\Prob&\Big\{\Big|\sum_{k\in I^j}\sum_{\ell=1}^n
{\ranmtx}_{k\ell}Q_{k\ell}-\sum_{k\in I^j}\frac{\OutDeg_{k}}{n}\sum_{\ell=1}^n Q_{k\ell}\Big|
>t+C L\Delta_j\,\mid\,{\ranmtx}\in\Event_{\pvector}(L)
\Big\}\\
&\leq \frac{2}{\Prob(\Event_\pvector(L))}\exp\left(-\frac{d\,\|Q^j\|_{HS}^2}{n\,\|Q^j\|_\infty^2}
\, H\left(\frac{\gamma tn\|Q^j\|_\infty}{d\|Q^j\|_{HS}^2}\right)\right),\quad t>0,
\end{align*}
where $\Delta_j:=\sqrt{d}\,\sum_{k\in I^j}\|\Row_k(Q)\|_{\log,n}$.
It is not difficult to check that the function
$f(s,w):=\frac{s^2}{w^2}H(\frac{bw}{s^2})$ is decreasing in both arguments $s$ and $w$ for any value of parameter $b>0$.
Hence, the above quantity is majorized by
$$
\frac{2}{\Prob(\Event_\pvector(L))}\exp\left(-\frac{d\,\|Q\|_{HS}^2}{n\,\|Q\|_\infty^2}
\, H\left(\frac{\gamma tn\|Q\|_\infty}{d\|Q\|_{HS}^2}\right)\right).
$$
Finally, note that if for some matrix $M\in\MatrixSet(\InDeg,\OutDeg)$ and $t>0$
we have
$$\Big|\sum_{k=1}^n\sum_{\ell=1}^n
{M}_{k\ell}Q_{k\ell}-\sum_{k=1}^n\frac{\OutDeg_{k}}{n}\sum_{\ell=1}^n Q_{k\ell}\Big|
>t+C L\Delta(Q)$$
then necessarily
$$\Big|\sum_{k\in I^j}\sum_{\ell=1}^n
{M}_{k\ell}Q_{k\ell}-\sum_{k\in I^j}\frac{\OutDeg_{k}}{n}\sum_{\ell=1}^n Q_{k\ell}\Big|
>\frac{t}{\lceil 2/c_0\rceil}+C L\Delta_j$$
for some $j\leq \lceil 2/c_0\rceil$.
The result follows.
\end{proof}

\begin{Remark}\label{rem: concentration constant}
It is easy to see that constant $C_3$ in the above theorem can be replaced by any number strictly
greater than one, at the expense of decreasing $\gamma$.
\end{Remark}

\begin{Remark}\label{rem: concentration-shift}
Note that, in view of Lemma~\ref{l: elementary log estimate}, we have
$$\Delta(Q)\leq C_{\ref{l: elementary log estimate}}\sqrt{\frac{d}{n}}\sum_{i=1}^n \|\Row_i(Q)\|
\leq C_{\ref{l: elementary log estimate}}\sqrt{d}\|Q\|_{HS}.$$
In particular, if $x$ and $y$ are unit vectors in $\R^n$ then
$\Delta(x y^T)\leq C_{\ref{l: elementary log estimate}}\sqrt{d}$.
Further, if all non-zero entries of the matrix $Q$ are located in a submatrix of size $k\times \ell$
(for some $k,\ell\leq n$) then, again applying Lemma~\ref{l: elementary log estimate},
we get
$$\Delta(Q)\leq C_{\ref{l: elementary log estimate}}\frac{\sqrt{d\ell}}{n}\ln\frac{2n}{\ell}\,\sum_{i=1}^n\|\Row_i(Q)\|
\leq C_{\ref{l: elementary log estimate}}\frac{\sqrt{d k\ell}}{n}\ln\frac{2n}{\ell}\,\|Q\|_{HS}.$$
In particular, given a $k$-sparse unit vector $x$ and an $\ell$-sparse unit vector $y$, we have
$$\Delta(x y^T)
\leq C_{\ref{l: elementary log estimate}}\frac{\sqrt{d k\ell}}{n}\ln\frac{2n}{\ell}.$$
\end{Remark}

\begin{Remark}\label{rem: d instead of dout}
Assume that $\bigl\|(\OutDeg_i-d)_{i=1}^n\bigr\|_{\psi,n}\leq K\sqrt{d}$ for some parameter $K>0$.
Then we have, in view of Lemma~\ref{l: elementary psi estimate}:
\begin{align*}
\Big|\sum_{i=1}^n\frac{\OutDeg_{i}}{n}\sum_{j=1}^n Q_{ij}-\frac{d}{n}\sum_{i,j=1}^n Q_{ij}\Big|
&\leq \frac{1}{n}\sum_{j=1}^n\Big|\sum_{i=1}^n (\OutDeg_{i}-d)Q_{ij}\Big|\\
&\leq \frac{1}{n}\bigl\|(\OutDeg_i-d)_{i=1}^n\bigr\|\sum_{j=1}^n\|\Col_j(Q)\|\\
&\leq C_{\ref{l: elementary psi estimate}}K\sqrt{d}\|Q\|_{HS}.
\end{align*}
Together with Remark~\ref{rem: concentration-shift}, this implies that the quantity ``$\sum_{i=1}^n\frac{\OutDeg_{i}}{n}\sum_{j=1}^n Q_{ij}$''
in the estimate of Theorem~\ref{th: main concentration} can be replaced with ``$\frac{d}{n}\sum_{i,j=1}^n Q_{ij}$''
at expense of substituting $\sqrt{d}\|Q\|_{HS}$ for the shift $\Delta(Q)$.
\end{Remark}

\begin{Remark}
The Bennett--type concentation inequality for linear forms obtained in \cite{CGJ} (see formula (6) there)
contains a parameter playing the same role as shift $\Delta(Q)$ in our theorem.
However, the dependence of the ``shift'' in \cite{CGJ} on matrix $Q$ is fundamentally different from ours.
Given a random matrix $\widetilde\ranmtx$ uniformly distributed on the set $\SymMatrixSet(d)$,
for every matrix $Q$ with non-negative entries and zero diagonal, Theorem~5.1 of \cite{CGJ} gives:
$$\Prob\Big\{\Big|\sum_{i,j=1}^n \widetilde\ranmtx_{ij}Q_{ij}-\frac{d}{n}\sum_{i,j=1}^n Q_{ij}\Big|\geq t+\frac{Cd^2}{n^2}
\sum_{i,j=1}^n Q_{ij}\Big\}\leq 2\exp\bigg(-\frac{d\,\|Q\|_{HS}^2}{n\,\|Q\|_\infty^2}
\, H\left(\frac{ctn\|Q\|_\infty}{d\|Q\|_{HS}^2}\right)\bigg).$$
In view of \eqref{eq: log 1},
the ``shift'' $\frac{d^2}{n^2}\sum_{i,j=1}^n Q_{ij}$ is majorized by
$\frac{ed^2}{n}\sum_{i=1}^n \|\Row_i(Q)\|_{\log,n}=\frac{ed^{3/2}}{n}\Delta(Q)$.
Thus, the concentration inequality from \cite{CGJ} gives sharper estimates than ours provided that $d=O(n^{2/3})$.
On the other hand, for $d\gg n^{2/3}$ the estimate in \cite{CGJ} becomes insufficient to produce
the optimal upper bound on the matrix norm, whereas our shift $\Delta(Q)$ gives satisfactory estimates
for all large enough $d$. Let us emphasize that this comparison is somewhat artificial
since \cite{CGJ} deals only with undirected graphs and symmetric matrices, while our Theorem~\ref{th: main concentration}
applies to the directed setting.
\end{Remark}

The proof of Theorem~D from the Introduction is obtained by combining Theorem~\ref{th: main concentration}
with Remarks~\ref{rem: concentration constant}--\ref{rem: d instead of dout} and Proposition~\ref{p: bounds for pvector}.

\bigskip

Let us finish this section by discussing the necessity of the tensorization procedure.
As we mentioned in the Introduction,
Freedman's inequality for martingales was employed in paper \cite{DJPP} dealing with {\it the permutation
model} of regular graphs (when the adjacency matrix of corresponding random multigraph is constructed
using independent random permutation matrices and their transposes).
It was proved in \cite{DJPP} that the second largest eigenvalue of such a graph is of order $O(\sqrt{d})$
with high probability.
Importantly, in \cite{DJPP} the martingale sequence was constructed for {\it the entire} matrix,
thereby yielding a concentration inequality directly after applying Freedman's theorem and
without any need for a tensorization procedure.
The fact that in our paper we construct martingales row by row is essentially responsible for the 
presence of the ``shift'' $\Delta(Q)$ in our concentration inequality, and
forced us to develop the lengthy and technical tensorization.
However, when constructing a single martingale sequence over the entire matrix, revealing the matrix entries
one by one in some appropriate order,
it is not clear to us how to control
martingale's parameters (absolute values of the differences and their variances).
Nevertheless, it seems natural to expect that some kind of an ``all-matrix'' martingale
can be constructed and analysed, yielding a much stronger concentration inequality for linear forms.

\section{The Kahn--Szem\'eredi argument}\label{s: kahn-szemeredi}

In this section, we use the concentration result established above and the
well known argument of Kahn and Szem\'eredi \cite{FKS}
to bound $s_2(\ranmtx)$, for $\ranmtx$ uniformly distributed on $\MatrixSet(\InDeg,\OutDeg)$.
The agrument was originally devised to handle $d$-regular undirected graphs, and we refer to \cite{CGJ} for a detailed exposition in that setting. 
In our situation, the Kahn--Szem\'eredi argument must be adapted to take into account absence of symmetry.
Still, let us emphasize that the structure of proofs given in this section bears a lot of similarities with those presented in \cite{CGJ}.
Set
$$S_0^{n-1}:=\Big\{y\in S^{n-1}:\, \sum_{i=1}^ny_i=0\Big\}.$$
The Courant--Fischer formula implies
$$
s_2(\ranmtx)\leq \sup_{y\in S_0^{n-1}} \Vert \ranmtx y\Vert=  \sup_{(x,y)\in S^{n-1}\times S_0^{n-1}} \langle \ranmtx y,x\rangle
$$
(of course, the above relation is true for {\it any} $n\times n$ matrix $M$).
To estimate the expression on the right hand side, we shall apply
our concentration inequality to $\langle \ranmtx y,x\rangle$ 
for any fixed couple $(x,y)\in S^{n-1}\times S_0^{n-1}$, and then invoke a covering argument.
Let us take a closer look at the procedure.
We have for any admissible $x,y$:
$$
\langle \ranmtx y,x\rangle= \sum_{i,j=1}^n x_i\ranmtx_{ij}  y_j=\sum_{i,j=1}^n \ranmtx_{ij}Q_{ij},
$$
where $Q:= xy^t$ satisfies $\Vert Q\Vert_{\rm HS}=1$ and 
$\Vert Q\Vert_{\infty}= \max\limits_{i,j\in [n]} \vert x_iy_j\vert= \Vert x\Vert_{\infty}\, \Vert y\Vert_\infty$. 
Therefore, in view of the concentration statement obtained in Section~\ref{s: tensorization}, the (conditional) probability that
$\langle \ranmtx y,x\rangle\gg\sqrt{d}$ is bounded by 
$$
\exp\bigg(-\frac{d}{n}\frac{H\big(\frac{n}{\sqrt{d}}\, \Vert x\Vert_{\infty}\, \Vert y\Vert_\infty\big)}{\Vert x\Vert_{\infty}^2\, \Vert y\Vert_\infty^2}\bigg).
$$
(we disregard any constant factors in the above expression).
However, when $\Vert x\Vert_{\infty}\, \Vert y\Vert_\infty\gg\sqrt{d}/n$, the estimate
becomes too weak (larger than $C^{-n}$)
to apply the union bound over a net of size exponential in $n$. The idea of Kahn and Szem\'eredi is to split the entries of $Q$ 
into two groups
according to their magnitude.
Then the standard approach discussed above
would work for the collection of entries smaller than $\sqrt{d}/n$.
Corresponding pairs of indices are called \textit{light couples}.
For the second group,
the key idea is to exploit discrepancy properties of the associated graph;
again, our concentration inequality will play a crucial role in their verification.

Given $(x,y)\in S^{n-1}\times S_0^{n-1}$, let us define
$$
\mathcal{L}(x,y):=\big\{(i,j)\in [n]^2:\, \vert x_iy_j\vert\leq \sqrt{d}/n\big\}
\quad \text{and}\quad 
\mathcal{H}(x,y):=\big\{(i,j)\in [n]^2:\, \vert x_iy_j\vert> \sqrt{d}/n\big\}.
$$
The notation $\mathcal{L}(x,y)$ stands for \textit{light couples} while $\mathcal{H}(x,y)$ refers to \textit{heavy couples}. 
Moreover, we will represent the corresponding partition of $Q$ as
$Q=Q_\mathcal{L}+Q_\mathcal{H}$,
where $Q_\mathcal{L},Q_\mathcal{H}$ are both $n\times n$ matrices in which the entries
from ``the alien'' collection are replaced with zeros.

Throughout the section, we always assume that the degree sequences $\InDeg$, $\OutDeg$
satisfy \eqref{eq: degree condition} for some $d$, and that $d$ itself satisfies
\eqref{eq-condition-d-sec-row-concentration}.
Moreover, we always assume that the set of matrices $\MatrixSet(\InDeg,\OutDeg)$ is non-empty.
As before, $\ranmtx$ is the random matrix 
uniformly distributed on $\MatrixSet(\InDeg,\OutDeg)$ and $\rangr$ is the associated random graph.

\begin{lemma}\label{lem-light-couples}
For any $L\geq 1$ there is $\gamma=\gamma(L)>0$ with the following property:
Let $n\geq C_{\ref{lem-light-couples}}$ and let $(x,y)\in S^{n-1}\times S_0^{n-1}$. Then for any $t>0$ we have  
$$
\Prob\Big\{\Big\vert \sum_{(i,j)\in \mathcal{L}(x,y)} x_i\ranmtx_{ij}y_j\Big\vert
\geq (C_{\ref{lem-light-couples}}\, L+t)\sqrt{d}\mid\, \Event_{\pvector}(L)\Big\}
\leq \frac{C_{\ref{lem-light-couples}}}{\Prob(\Event_\pvector(L))}\,  \exp\big(-n\, H(\gamma\, t)\big).
$$
Here, $C_{\ref{lem-light-couples}}>0$ is a sufficiently large universal constant
and $\Event_\pvector(L)$ is defined by \eqref{eq: event pvector definition}.
\end{lemma}
\begin{proof}
Let $(x,y)\in S^{n-1}\times S_0^{n-1}$ and denote $Q:= xy^t$. Let $Q_{\mathcal L}$ and $Q_{\mathcal H}$
be defined as above.
By the definition of $\mathcal{L}(x,y)$, we have $\Vert Q_{\mathcal{L}}\Vert_\infty\leq \sqrt{d}/n$, and,
since $\Vert x\Vert=\Vert y\Vert=1$, we have
$\Vert Q_{\mathcal{L}}\Vert_{\rm HS}\leq 1$. 
Further, note that 
$$
\sum_{i=1}^n \Vert \Row_i(Q_{\mathcal{L}})\Vert_1 
\leq \Vert x\Vert_1\, \Vert y\Vert_1\leq n,
$$
whence, in view of Lemma~\ref{l: elementary log estimate},
$$
\sum_{i=1}^n\|\Row_i(Q_{\mathcal{L}})\|_{\log,n}\leq C_{\ref{l: elementary log estimate}}.
$$
Applying Theorem~\ref{th: main concentration}
to matrix $Q_{\mathcal{L}}$ with $t:=r \sqrt{d}$ ($r>0$), we get that
there exists $\gamma:=\gamma(L)>0$ depending on $L$ such that 
\begin{align}\nonumber
\Prob&\Big\{\Big\vert \sum_{(i,j)\in \mathcal{L}(x,y)} x_i\ranmtx_{ij}y_j-\sum_{(i,j)\in \mathcal{L}(x,y)} \frac{\OutDeg_i\, x_iy_j}{n}\Big\vert 
\geq (C\, L+r)\sqrt{d}\mid\, \ranmtx\in\Event_{\pvector}(L)\Big\}\\\label{eq1-lem-light-couples}
&\leq \frac{C_3}{\Prob(\Event_\pvector(L))} \exp\big(-n\, H(\gamma\, r)\big),
\end{align}
where $C$ is a universal constant and $C_3$ is the constant from Theorem~\ref{th: main concentration}.
Since the coordinates of $y$ sum up to zero, we have for any $i\leq n$: 
$$
\Big\vert \sum_{j:\, (i,j)\in \mathcal{L}(x,y)} \OutDeg_i\, x_{i}y_j\Big\vert
=\Big\vert \sum_{j:\, (i,j)\in \mathcal{H}(x,y)} \OutDeg_i\, x_{i}y_j\Big\vert
\leq d\sum_{j:\, (i,j)\in \mathcal{H}(x,y)} \frac{(x_iy_j)^2}{\sqrt{d}/n},
$$
where in the last inequality we used that $\OutDeg_i\leq d$ and $\vert x_{i}y_j\vert\geq \sqrt{d}/n$ for $(i,j)\in  \mathcal{H}(x,y)$. 
Summing over all rows and using the condition $\Vert x\Vert=\Vert y\Vert=1$, we get 
$$
\Big|\sum_{(i,j)\in \mathcal{L}(x,y)} \frac{\OutDeg_i\, x_iy_j}{n}\Big|\leq \sqrt{d}.
$$
This, together with  \eqref{eq1-lem-light-couples}, finishes the proof after choosing $C_{\ref{lem-light-couples}}\geq C+1$.
\end{proof}

Next, we prove a discrepancy property for our model.
In what follows, for any subsets $S,T\subset[n]$, $\edg(S,T)$ denotes the set of edges of $\rangr$ emanating from $S$ and landing in $T$.
For any $K_1,K_2\geq 1$, we denote by $\Event_{\ref{prop-edge-count}}(K_1,K_2)$ the event that
for {\it all} subsets $S,T\subset [n]$ at least one of the following is true: 
\begin{equation}\label{eq: discrep1}
\vert \edg(S,T)\vert\leq K_1\,\frac{d}{n}\, \vert S\vert \, \vert T\vert,
\end{equation}
or 
\begin{equation}\label{eq: discrep2}
\vert \edg(S,T)\vert\, \ln\bigg(\frac{\vert \edg(S,T)\vert}{ \frac{d}{n}\, 
\vert S\vert \, \vert T\vert}\bigg)\leq K_2\max(\vert S\vert, \vert T\vert)\, 
\ln\left(\frac{e\,n}{\max(\vert S\vert,\vert T\vert)}\right).
\end{equation}
Let us note that both conditions above can be equivalently restated using a single formula; however,
the presentation in form \eqref{eq: discrep1}--\eqref{eq: discrep2} nicely captures the underlying dichotomy
within a ``typical'' realization of $\rangr$:
either both $S$ and $T$ are ``large'', in which case the edge count does not deviate too much from its expectation,
or at least one of the sets is ``small'', and the edge count, up to a logarithmic multiple, is bounded by the cardinality
of the larger vertex set.

\begin{prop}\label{prop-edge-count}
For any $L\geq 1$ and $m\in\N$ there are $n_0=n_0(L,m)$, $K_1=K_1(L,m)$ and $K_2=K_2(L,m)$
such that for $n\geq n_0$ and $d$ satisfying \eqref{eq-condition-d-sec-row-concentration}
we have
$$
\Prob\big(\Event_{\ref{prop-edge-count}}(K_1,K_2)\mid \Event_{\pvector}(L)\big)
\geq 1-\frac{1}{\Prob(\Event_\pvector(L))\, n^{m}}.
$$
\end{prop}

\begin{proof}
Fix for a moment any $S, T\subset [n]$ and let
$Q$ be the $n\times n$ matrix whose entries are equal to $1$ on $S\times T$ and $0$ elsewhere. Set
$k:=\vert S\vert$ and $\ell:=\vert T\vert$. 
From Remark~\ref{rem: concentration-shift}, we have 
$$
\Delta(Q)\leq C_{\ref{l: elementary log estimate}}\frac{\sqrt{d k\ell}}{n}\ln\frac{2n}{\ell}\,\|Q\|_{HS}
\leq C_{\ref{l: elementary log estimate}} \frac{\sqrt{d} k\ell}{n}\ln(2n)\leq C_{\ref{l: elementary log estimate}} \frac{d k\ell}{n},
$$
where the last inequality follows from the assumption \eqref{eq-condition-d-sec-row-concentration} on $d$. 
Using the estimate together with the inequality $\OutDeg_i\leq d$ ($i\leq n$) and applying
Theorem~\ref{th: main concentration}, for any $r>0$ we obtain
\begin{equation}\label{eq-proof-dicrep}
\Prob\Big\{\vert\edg(S,T)\vert 
> (CL+ r) \frac{d}{n} k\, \ell\,\mid\,{\ranmtx}\in\Event_{\pvector}(L)
\Big\}\leq \frac{C}{\Prob(\Event_\pvector(L))}\, \exp\left(-\frac{d\, k\,\ell}{n}\, H(\gamma\, r)\right),
\end{equation}
for a universal constant $C\geq 1$ and some $\gamma=\gamma(L)>0$.
Now, we set $K_1:=2CL$ and let $K_2=K_2(L,m)$ be the minimum number such that $K_2 H(\gamma t)\geq 2(3+m)t\ln(2t)$
for all $t\geq CL$ (note that the definition of $K_1$, $K_2$ does not depend on $S$ and $T$).
Since the function $H$ is strictly increasing on $(0,\infty)$, there is a unique number $r_1>0$
such that
$$
H(\gamma r_1)= \frac{(3+m)\max(k,\ell)}{\frac{d}{n}\, k\, \ell}\, \ln\left(\frac{e\, n}{\max(k,\ell)}\right).
$$
Next, note that if for a fixed realization
of the graph $\rangr$ we have $\vert\edg(S,T)\vert \leq (CL+r_1)\, \frac{d}{n} k\, \ell$
then either \eqref{eq: discrep1} or \eqref{eq: discrep2} holds.
Indeed, 
if $r_1\leq CL$ then the assertion is obvious.
Otherwise,
if $r_1>CL$ then, by the definition of $K_2$, we have $(CL+r_1)\ln (CL+r_1)\leq \frac{K_2}{3+m}\,H(\gamma r_1)$.
Together with the trivial estimate
$$
\vert \edg(S,T)\vert\, \ln\left(\frac{\vert \edg(S,T)\vert}{ \frac{d}{n}\, 
k\, \ell}\right)\leq \frac{d}{n}k \, \ell\, (CL+r_1)\ln(CL+r_1),
$$
this gives
$$
\vert \edg(S,T)\vert\, \ln\left(\frac{\vert \edg(S,T)\vert}{ \frac{d}{n}\, 
k \, \ell}\right)\leq \frac{K_2}{3+m}\, \frac{d}{n}k \, \ell \, H(\gamma r_1)= K_2 \max(k,\ell)\,   \ln\left(\frac{e\, n}{\max(k,\ell)}\right).
$$
Thus, all realizations of $\ranmtx$ (or, equivalently, $\rangr$) with $\vert\edg(S,T)\vert \leq (CL+r_1)\, \frac{d}{n} |S|\, |T|$
for all $S,T\subset[n]$,
necessarily fall into event $\Event_{\ref{prop-edge-count}}(K_1,K_2)$.
It follows that
$$
\Prob\big(\Event_{\ref{prop-edge-count}}^c(K_1,K_2)\mid \Event_{\pvector}(L)\big)
\leq \Prob\Big\{\exists S, T\subset [n]:\, \vert\edg(S,T)\vert 
> (CL+ r_1) \frac{d}{n} |S|\, |T|\,\mid\,{\ranmtx}\in\Event_{\pvector}(L)
\Big\}.
$$
Applying \eqref{eq-proof-dicrep}, we get
\begin{align*}
\Prob\big(\Event_{\ref{prop-edge-count}}^c(K_1,K_2)\mid \Event_{\pvector}(L)\big)
&\leq \frac{C}{\Prob(\Event_\pvector(L))}\,\sum_{k,\ell=1}^n {n\choose k}\, {n\choose \ell} \exp\left(-\frac{d\, k\ell}{n}\, H(\gamma r_1)\right)\\
&\leq \frac{C}{\Prob(\Event_\pvector(L))}\,\sum_{k,\ell=1}^n
\exp\left(k\ln\big(\frac{en}{k}\big)+\ell\ln\big(\frac{en}{\ell}\big)-\frac{d\, k\ell}{n}\, H(\gamma r_1)\right)\\
&\leq \frac{C}{\Prob(\Event_\pvector(L))}\,\sum_{k,\ell=1}^n \exp\left[-(m+1)\max(k,\ell)\, \ln\left(\frac{e\, n}{\max(k,\ell)}\right)\right]\\
&\leq \frac{C}{\Prob(\Event_\pvector(L))\, n^{m+1}}\\
&\leq \frac{1}{\Prob(\Event_\pvector(L))\, n^{m}},
\end{align*}
where we used the estimate $\max(k,\ell)\, \ln\left(\frac{e\, n}{\max(k,\ell)}\right)\geq \ln n$.
\end{proof}

The conditions on the edge count of a graph expressed via \eqref{eq: discrep1} or \eqref{eq: discrep2},
are a basic element in the argument of Kahn and Szem\'eredi. The following 
lemma shows that the contribution of heavy couples to the matrix norm is {\it deterministically} controlled
once we suppose that either 
\eqref{eq: discrep1} or \eqref{eq: discrep2} holds for all vertex subsets of corresponding graph.

\begin{lemma}\label{lem-kahn-szemeredi}
For any $K_1,K_2>0$ there exists $\beta>0$ depending only on $K_1,K_2$ such that the following holds.  
Let, as usual, the degree sequences $\InDeg,\OutDeg$ be bounded from above by $d$ (coordinate-wise)
and let $M\in\Event_{\ref{prop-edge-count}}(K_1,K_2)$.
Then for any $x, y\in S^{n-1}$, we have
$$
\Big\vert\sum_{(i,j)\in \mathcal{H}(x,y)} x_i M_{ij} y_j \Big\vert\leq \beta \sqrt{d}.
$$
\end{lemma}
A proof of this statement in the undirected $d$-regular setting is well known \cite{FKS, CGJ}.
In the appendix to this paper, we include the proof adapted to our situation.

\bigskip

In order to simultaneously estimate contribution of all pairs of vectors from $S^{n-1}\times S_0^{n-1}$
to the second largest singular value of our random matrix, we shall discretize this set. 
The following lemma is quite standard.

\begin{lemma}\label{lem-net-approx}
Let $\varepsilon\in (0,1/2)$, $\epsnet$ be a Euclidean $\varepsilon$-net in $S^{n-1}$,
and $\epsnet^0$ be a Euclidean $\varepsilon$-net in $S_0^{n-1}$. Further, let $A$ be any $n\times n$ non-random matrix
and $R$ be any positive number such that 
$\vert\langle Ax,y\rangle\vert\leq R$ for all $(x,y)\in \epsnet\times \epsnet^0$.
Then $\vert\langle Ax,y\rangle\vert \leq R/(1-2\varepsilon)$ for all $(x,y)\in S^{n-1}\times S_0^{n-1}$. 
\end{lemma}
\begin{proof}
Let $(x_0, y_0)\in S^{n-1}\times S_0^{n-1}$ be such that $a:= \sup_{(x,y)\in S^{n-1}\times S_0^{n-1}} \langle Ax,y\rangle = \langle Ax_0,y_0\rangle$. 
By the definition of $\epsnet$ and $\epsnet^0$, there exists a pair $(x_0',y_0')\in \epsnet\times \epsnet^0$ such that
$\Vert x_0-x_0'\Vert\leq \varepsilon$ and $\Vert y_0-y_0'\Vert\leq \varepsilon$. 
Together with the fact that the normalized difference of two elements in $S_0^{n-1}$ remains in $S_0^{n-1}$, this yields
\begin{align*}
\langle Ax_0,y_0\rangle&= \langle A(x_0-x_0'),y_0\rangle+\langle Ax_0',y_0-y_0'\rangle+\langle Ax_0',y_0'\rangle \\
&\leq a\Vert x_0-x_0'\Vert+ a\Vert y_0-y_0'\Vert+ \sup_{(x,y)\in \epsnet\times \epsnet^0} \vert\langle Ax,y\rangle\vert.
\end{align*}
Hence,
$$
a\leq 2\varepsilon\, a+ R,
$$
which gives that $a\leq R/(1-2\varepsilon)$.

\end{proof}

Now, we can prove the main statement of this section. It is easy to check that the theorem below, together with
Proposition~\ref{p: bounds for pvector}, 
gives Theorem~C from the Introduction.
To make the statement self-contained, we explicitly mention all the assumptions on parameters.

\begin{theorem}\label{th: main spectral gap}
For any $L, m\geq 1$ there exist $\kappa=\kappa(L,m)>0$ and $n_0=n_0(L,m)$ with the following properties.
Assume that $n\geq n_0$ and that the degree sequences $\InDeg,\OutDeg$ satisfy
$$(1-c_0)d\leq \InDeg_i,\OutDeg_i\leq d,\quad i\leq n$$
for some natural $d$ with $C_{\ref{sec-row-concentration}}\ln^2 n\leq d\leq (1/2+c_0)n$. 
Then, with $\Event_\pvector(L)$ defined by \eqref{eq: event pvector definition}, we have 
$$
\Prob\big\{M\in \MatrixSet(\InDeg, \OutDeg): s_2(M)\geq \kappa\, \sqrt{d}\big\}\leq \frac{1}{n^m}+\Prob(\Event_\pvector(L)^c).
$$
\end{theorem}

\begin{proof}
Let $K_1=K_1(L,m+1)$ and $K_2=K_2(L,m+1)$ be defined as in Proposition~\ref{prop-edge-count}, and
let $\gamma=\gamma(L)$ and $\beta=\beta(K_1,K_2)$ be functions
from Lemmas~\ref{lem-light-couples} and~\ref{lem-kahn-szemeredi}. 
We will use the shorter notation $\Event_{\pvector}$ and $\Event_{\ref{prop-edge-count}}$ instead of $\Event_{\pvector}(L)$ and 
$\Event_{\ref{prop-edge-count}}(K_1,K_2)$, respectively.
Set
$$r:=\gamma^{-1}\, H^{-1}(1+\ln 81)$$
and denote
$$
\Event:=\big\{M\in \MatrixSet(\InDeg, \OutDeg):\, s_2(M)\geq 2(C_{\ref{lem-light-couples}}\, L+ \beta+r)\sqrt{d}\big\}.
$$ 
Using the Courant--Fischer formula, we obtain
\begin{align*}
\Prob(\Event\mid \Event_{\pvector}\cap \Event_{\ref{prop-edge-count}})
\leq \Prob\Big\{& \,\exists (x,y)\in S^{n-1}\times S_0^{n-1}\text{ such that } \\ 
&\vert\langle \ranmtx y,x\rangle\vert\geq 2(C_{\ref{lem-light-couples}}\, L+ \beta+r)\sqrt{d}\mid 
\ranmtx\in \Event_{\pvector}\cap \Event_{\ref{prop-edge-count}}\Big\}.
\end{align*}
Let $\mathcal{N}$ be a $1/4$-net in $S^{n-1}$ and $\mathcal{N}_0$ be a $1/4$-net in $S_0^{n-1}$. 
Standard volumetric estimates show that we may take $\mathcal{N}$ and $\mathcal{N}_0$
such that $\max(\vert \mathcal{N}\vert, \vert \mathcal{N}_0\vert) \leq 9^n$. 
Applying Lemma~\ref{lem-net-approx}, we get
\begin{align}\label{eq1-proof-main-theorem}
\Prob\{\Event\mid \Event_{\pvector}\cap \Event_{\ref{prop-edge-count}}\}
\leq \Prob\Big\{&\exists (x,y)\in \mathcal{N}\times \mathcal{N}_0 \text{ such that }\nonumber\\ 
&
\vert\langle \ranmtx y,x\rangle\vert\geq (C_{\ref{lem-light-couples}}\, L+ \beta+r)
\sqrt{d}\mid \ranmtx\in\Event_{\pvector}\cap \Event_{\ref{prop-edge-count}}\Big\}\nonumber\\
&\nonumber\\
&\hspace{-3cm}\leq (81)^n\max_{(x,y)\in S^{n-1}\times S_0^{n-1}}
\Prob\Big\{ \vert\langle \ranmtx y,x\rangle\vert\geq (C_{\ref{lem-light-couples}}\, L+ \beta+r)\sqrt{d}\mid 
\ranmtx\in\Event_{\pvector}\cap \Event_{\ref{prop-edge-count}}\Big\}.
\end{align}
Given $(x,y)\in S^{n-1}\times S_0^{n-1}$, we obviously have
$$
\vert\langle \ranmtx y,x\rangle\vert 
\leq \Big\vert\sum_{(i,j)\in \mathcal{L}(x,y)} x_i \ranmtx_{ij} y_j\Big\vert +\Big\vert\sum_{(i,j)\in \mathcal{H}(x,y)} x_i \ranmtx_{ij} y_j\Big\vert.
$$
From Lemma~\ref{lem-kahn-szemeredi}, we get $\Big\vert\sum_{(i,j)\in \mathcal{H}(x,y)} x_i \ranmtx_{ij}y_j\Big\vert\leq \beta\sqrt{d}$ 
whenever $\ranmtx\in  \Event_{\ref{prop-edge-count}}$. 
Hence, in view of \eqref{eq1-proof-main-theorem}, 
$$
\Prob(\Event\mid \Event_{\pvector}\cap \Event_{\ref{prop-edge-count}})
\leq (81)^n\max_{(x,y)\in S^{n-1}\times S_0^{n-1}}
\Prob\Big\{  \Big\vert\sum_{(i,j)\in \mathcal{L}(x,y)} x_i \ranmtx_{ij} y_j \Big\vert\geq (C_{\ref{lem-light-couples}}\, L+r)\sqrt{d}\mid 
\Event_{\pvector}\cap \Event_{\ref{prop-edge-count}}\Big\}.
$$
Applying Lemma~\ref{lem-light-couples}, we further obtain, by the choice of $r$,
$$
\Prob(\Event\mid \Event_{\pvector}\cap \Event_{\ref{prop-edge-count}})
\leq \frac{C_{\ref{lem-light-couples}}\, (81)^n}{\Prob(\Event_\pvector)}\,  \exp\left(-n\, H(\gamma\, r)\right)\leq 
 \frac{C_{\ref{lem-light-couples}}\, e^{-n}}{\Prob(\Event_\pvector)}.
$$
To finish the proof, note that 
$$
\Prob(\Event)\leq \Prob(\Event\mid \Event_{\pvector}\cap \Event_{\ref{prop-edge-count}})\, \Prob(\Event_{\pvector})+ 
\Prob( \Event_{\ref{prop-edge-count}}^c\mid \Event_{\pvector})\, \Prob(\Event_{\pvector})+ \Prob(\Event_{\pvector}^c)
$$
and use the above estimate together with Proposition~\ref{prop-edge-count}. 
\end{proof}

The concentration inequality obtained in Theorem~\ref{th: main concentration}, was used in its full strength
in Proposition~\ref{prop-edge-count} to control the input of heavy couples.
For the light couples though, it would be sufficient to apply a weaker Berstein--type
bound where the function $H(\tau)$ in the exponent
is replaced with $\frac{\tau^2}{2+2\tau/3}$.

\section{The undirected setting}\label{s: undirected-case}

In this section, we show how to deduce Theorem~A from Theorem~C.
In \cite{TY_short}, we showed that in a rather general setting
the norm of a random matrix, whose distribution is invariant under joint permutations of rows and columns, 
can be bounded in terms of the norm of its $n/2\times n/2$ submatrix located in the top right corner.
Moreover, for matrices with constant row and column sums, an analogous phenomenon
holds for the second largest singular values. 
Since the distribution of edges in the undirected uniform model is invariant under permutation of the set of vertices,
the results of \cite{TY_short} are applicable in our context. 

We will need the following definition. For any $\ell,d>0$ and any parameter $\delta>0$ we set
\begin{align*}
\degreeset_\ell(d,\delta):=\Big\{&(u,v)\in \N^\ell\times\N^\ell:\, \Vert u\Vert_1=\Vert v\Vert_1\;\;\mbox{ AND }\\
&\big|\big\{i\leq \ell:\,\big|u_i-d\big|> k\delta\big\}\big|
\leq \ell e^{-k^2}\mbox{ for all }k\in\N\;\;\mbox{ AND}\\
&\big|\big\{i\leq \ell:\,\big|v_i-d\big|> k\delta\big\}\big|
\leq \ell e^{-k^2}\mbox{ for all }k\in\N\Big\}.
\end{align*}
Note that any pair of vectors $(u,v)$ from $\degreeset_\ell(d,\delta)$
necessarily satisfy $\|u-d\vofones\|_{\psi,n},\|v-d\vofones\|_{\psi,n}\leq C\delta$
for some universal constant $C>0$.

Below we state a special case of the main result of \cite{TY_short}, where we replace a general random matrix 
with constant row/column sums by the adjacency matrix of a random regular graph.

\begin{theorem}[\cite{TY_short}]\label{th: main2}
There exist positive universal constants $c,C$ such that the following holds. 
Let $n\geq C$ and let $d\in\N$ satisfy $d\geq C\ln n$.
Further, let $\rangr$ be a random undirected graph uniformly distributed on $\UGraphSet(d)$ and
let $T$ be the $\lfloor n/2\rfloor\times\lfloor n/2\rfloor$ top right corner of the adjacency matrix of $\rangr$. 
Then, viewing $T$ as the adjacency matrix of a random directed graph on $\lfloor n/2\rfloor$ vertices, for any $t\geq C$ we have 
\begin{align*}
\Prob\big\{s_2(\rangr)\geq Ct\sqrt{d}\big\}\leq \frac{1}{c}\Prob\Big\{ s_2(T)\geq ct\sqrt{d} \mbox{ \ AND }
\big(\InDeg(T),\OutDeg(T)\big)\in \degreeset_{\lfloor n/2\rfloor}\big(d/2,C\sqrt{d}\big)\Big\}.
\end{align*}
\end{theorem}

\medskip

Equipped with the above statement and with Theorem~C, we can proceed with the proof of Theorem~A. 
\begin{proof}[Proof of Theorem~A]
Let $m\in \N$, $\alpha>0$ and let $c, C$ be the constants from Theorem~\ref{th: main2}.
We assume that $n^\alpha\leq d\leq n/2$. 
Denote by $A=(a_{ij})$ the adjacency matrix of the random graph $\rangr$ uniformly distributed on 
$\UGraphSet(d)$.
Let $T$ be the $\lfloor n/2\rfloor\times \lfloor n/2\rfloor$ 
top right corner of $A$. 

Fix for a moment {\it any} degree sequences $(\InDeg,\OutDeg)$ of length $\lfloor n/2\rfloor$ bounded above by $d$
such that the event $\{(\InDeg(T),\OutDeg(T))=(\InDeg,\OutDeg)\}$ is non-empty. 
Then, conditioned on the event, the directed random graph on $\lfloor n/2\rfloor$ vertices with adjacency matrix $T$
is {\it uniformly} distributed on ${\mathcal D}_{\lfloor n/2\rfloor}(\InDeg,\OutDeg)$. In other words,
the distribution of $T$, conditioned on the event $\{(\InDeg(T),\OutDeg(T))=(\InDeg,\OutDeg)\}$,
is uniform on the set ${\mathcal M}_{\lfloor n/2\rfloor}(\InDeg,\OutDeg)$.

Now if $(\InDeg,\OutDeg)\in \degreeset_{\lfloor n/2\rfloor}(d/2, C\sqrt{d})$,
then, applying Theorem~C, we get
\begin{equation}\label{eq: last}
\Prob\Big\{s_2(T)\geq \tilde t\sqrt{d}\mid \big(\InDeg(T),\OutDeg(T)\big)=(\InDeg,\OutDeg)\Big\}\leq \frac{1}{n^m},
\end{equation}
for some $\tilde t$ depending on $\alpha, C$ and $m$. 
Set $t:=C\max(1, \tilde t/c)$. In view of Theorem~\ref{th: main2}, we get
\begin{align*}
\Prob\big\{s_2(\rangr)\geq t \sqrt{d} \big\}\leq \frac{1}{c}\Prob\Big\{ &s_2(T)\geq \tilde t \sqrt{d} \mbox{ \ AND }\\
& \big(\InDeg(T),\OutDeg(T)\big)\in \degreeset_{\lfloor n/2\rfloor}\big(d/2, C\sqrt{d}\big)\Big\}=:\eta.
\end{align*}
Obviously,
$$
\eta=\frac{1}{c} \sum_{(\InDeg,\OutDeg)\in \degreeset_{\lfloor n/2\rfloor}\big(d/2, C\sqrt{d}\big)}
\Prob\Big\{ s_2(T)\geq \tilde t \sqrt{d} \mbox{ \ AND }\big(\InDeg(T),\OutDeg(T)\big)=(\InDeg,\OutDeg)\Big\}.
$$
Hence, applying \eqref{eq: last}, we get 
\begin{align*}
\eta&\leq \frac{1}{c\, n^m}\sum_{(\InDeg,\OutDeg)\in \degreeset_{\lfloor n/2\rfloor}\big(d/2, C\sqrt{d}\big)} 
\Prob\Big\{\big(\InDeg(T),\OutDeg(T)\big)=(\InDeg,\OutDeg)\Big\}\leq \frac{1}{c\, n^m},\\
\end{align*}
and complete the proof.
\end{proof}

\bigskip

{\bf Acknowledgments.}
A significant part of this work was done when the second named author visited
the University of Alberta in May--June 2016, and when both authors visited
the Texas A\&M University in July 2016. Both authors are grateful to the University of Alberta
and the Texas A\&M University for excellent working conditions, and would especially like to thank
Nicole Tomczak--Jaegermann, Bill Johnson, Alexander Litvak and Grigoris Paouris. 
We would also like to thank Djalil Chafa\"{\i} for helpful comments. 
The first named author is partially supported by the Simons Foundation ({\it{}Collaboration on Algorithms and Geometry}).

\section{Appendix}

Here, we provide a detailed proof of Lemma~\ref{lem-kahn-szemeredi}. Let us emphasize
that corresponding result for undirected graphs is well known (see a detailed proof in \cite{CGJ}); 
the sole purpose of this part of the paper is
to convince the reader that the argument carries easily to the directed setting.

\begin{proof}[Proof of Lemma~\ref{lem-kahn-szemeredi}] 
Let $\InDeg,\OutDeg$ be the two given degree sequences, and $K_1$ and $K_2$ be the two 
parameters in the definition of $\Event_{\ref{prop-edge-count}}(K_1,K_2)\subset\MatrixSet(\InDeg,\OutDeg)$.
Let $M$ be any fixed matrix in $\Event_{\ref{prop-edge-count}}(K_1,K_2)$ and
$\nrangr$ be the corresponding graph.

Let $x,y\in S^{n-1}$, and for any $i\geq 1$ define 
$$
S_i:=\Big\{k\in[n]:\, \vert x_k\vert\in \frac{1}{\sqrt{n}}[2^{i-1},2^i)\Big\}\quad \text{and}\quad 
T_i:=\Big\{k\in[n]:\, \vert y_k\vert\in \frac{1}{\sqrt{n}}[2^{i-1},2^i)\Big\}.
$$
Note that any couple $(i,j)$ with $\min(|x_i|,|y_j|)<n^{-1/2}$ is {\it light}.
Further, whenever $(k,\ell)\in \mathcal{H}(x,y)\cap (S_i\times T_j)$ for some $i,j\geq 1$, we have 
$$
\frac{\sqrt{d}}{n}\leq \vert x_k y_\ell\vert \leq \frac{2^{i+j}}{n}.
$$
Hence, 
$$
\Big\vert\sum_{(i,j)\in \mathcal{H}(x,y)} x_i M_{ij}y_j\Big\vert\leq
\sum_{(i,j)\in \cal{I}} \frac{2^{i+j}}{n}\, \vert\edg(S_i,T_j)\vert,
$$
where ${\cal{I}}:=\{(i,j):\,  2^{i+j}\geq \sqrt{d}\}$. 
Set
$$
\InI:=\{(i,j)\in {\cal{I}}:\, \vert S_i\vert\geq \vert T_j\vert\}\quad \text{ and } \quad\OutI:=\{(i,j)\in {\cal{I}}:\, \vert S_i\vert\leq \vert T_j\vert\}.
$$
We have 
$$
\Big\vert\sum_{(i,j)\in \mathcal{H}(x,y)} x_i M_{ij}y_j\Big\vert\leq 
\sum_{(i,j)\in \InI} \frac{2^{i+j}}{n}\, \vert\edg(S_i,T_j)\vert 
+
\sum_{(i,j)\in \OutI} \frac{2^{i+j}}{n}\, \vert\edg(S_i,T_j)\vert.
$$
In what follows, we will bound the first term in the above inequality; the other summand is estimated in exactly the same way.
Given $(i,j)\in \InI$, 
denote 
$$
r_{ij}:=\frac{\vert \edg(S_i,T_j)\vert}{\frac{d}{n}\, \vert S_i\vert\, \vert T_j\vert},\quad 
\alpha_i:=\frac{2^{2i}}{n}\,\vert S_i\vert, \quad \beta_j:=\frac{2^{2j}}{n}\, \vert T_j\vert 
\quad  \text{and}\quad s_{ij}:=\frac{\sqrt{d}}{2^{i+j}}\, r_{ij}.
$$
Note that $s_{ij}\leq r_{ij}$. Further,
\begin{equation}\label{eq1-lem-kahn-szemeredi}
\sum_{i\geq 1} \alpha_i=4\sum_{i\geq 1}\vert S_i\vert \frac{2^{2i-2}}{n}\leq 4\sum_{i\geq 1} 
\sum_{k\in S_i} x_k^2\leq 4.
\end{equation}
Similarly, we have $\sum_{j\geq 1} \beta_j\leq 4$. 
Since the in- and out-degrees are bounded by $d$, we have 
$$
\vert \edg(S_i,T_j)\vert\leq \min( \sum_{k\in S_i} \OutDeg_k, \sum_{k\in T_j} \InDeg_k)\leq d\min (\vert S_i\vert , \vert T_j\vert)= d\vert T_j\vert,
$$ 
implying
\begin{equation}\label{eq2-lem-kahn-szemeredi}
r_{ij}\leq \frac{n}{\vert S_i\vert}
=  \frac{2^{2i}}{\alpha_i}.
\end{equation}
Next, as $M\in \Event_{\ref{prop-edge-count}}(K_1,K_2)$, we have
either $r_{ij}\leq K_1$ or 
\begin{equation}\label{eq-discrep-kahn-szemeredi}
r_{ij}\, \ln(r_{ij})\leq \frac{ K_2\, 2^{2j}}{d\, \beta_j} \, 
\ln\left( \frac{e\, 2^{2i}}{\alpha_i}\right).
\end{equation}
With the above notation,
$$
\sum_{(i,j)\in \InI} \frac{2^{i+j}}{n}\, \vert\edg(S_i,T_j)\vert 
=\sqrt{d} \sum_{(i,j)\in\InI}\alpha_i\beta_js_{ij}.
$$
Our aim is to show that
$$\widetilde g(M):=\sum_{(i,j)\in\InI}\alpha_i\beta_js_{ij}=O(1).$$
Let us divide $\InI$ into five subsets: 
\begin{align*}
\quad \quad 
&\InI_1:=\{(i,j)\in \InI:\, s_{ij}\leq K_1\}\\ 
&\\
& \InI_2:=\{(i,j)\in  \InI:\, 2^i\leq 2^{j}/\sqrt{d}\}\\
&\\
& \InI_3:=\Big\{(i,j)\in \InI:\, r_{ij}>  \left(\frac{e\, 2^{2i}}{\alpha_i}\right)^{\frac14}\Big\}\setminus (\InI_1\cup\InI_2)\\
&\\
& \InI_4:=\Big\{(i,j)\in  \InI:\, \frac{1}{\alpha_i}\leq e\, 2^{2i}\Big\}\setminus ( \InI_1\cup  \InI_2\cup\InI_3) \\
&\\
& \InI_5:= \InI\setminus (\InI_1\cup \InI_2 \cup\InI_3\cup\InI_4)
\end{align*}
For every $s=1,2,3,4,5$, we write 
$$
g_s(M):=\sum_{(i,j)\in\InI_s}\alpha_i\beta_js_{ij}.
$$
Obviously, $\widetilde g(M)\leq \sum_{s=1}^5 g_s(M)$. 

\medskip

\noindent\textbf{Claim 1.} $g_1(M)\leq 16K_1$.
\begin{proof}
Since $s_{ij}\leq K_1$ for $(i,j)\in \InI_1$, then 
in view of \eqref{eq1-lem-kahn-szemeredi}, we get
$$
g_1(M)\leq K_1\sum_{(i,j)\in \InI_1}\alpha_i\beta_j
\leq K_1\sum_{i\geq 1} \alpha_i\sum_{j\geq 1}\beta_j
\leq 16K_1.
$$
\end{proof}

\medskip

\noindent\textbf{Claim 2.} $g_2(M)\leq 8$.
\begin{proof}
In view of \eqref{eq2-lem-kahn-szemeredi}, we have
$$
g_2(M)=\sqrt{d}\, \sum_{(i,j)\in\InI_2}\alpha_i\beta_j\frac{r_{ij}}{2^{i+j}}
\leq \sqrt{d}\, \sum_{(i,j)\in\InI_2}\beta_j\frac{2^i}{2^{j}}
=\sqrt{d}\, \sum_{j\geq 1}\beta_j\, 2^{-j} \sum_{i: (i,j)\in\InI_2} 2^i.
$$
Since $2^{i}\leq 2^j/\sqrt{d}$ for $(i,j)\in \InI_2$,
the second sum is bounded by $2\cdot 2^j/\sqrt{d}$. Thus, we have 
$$
g_2(M)\leq 2 \sum_{j\geq 1}\beta_j\leq 8,
$$
where the last inequality follows from \eqref{eq1-lem-kahn-szemeredi} (with $\beta_j$ replacing $\alpha_i$).
\end{proof}

\medskip

\noindent\textbf{Claim 3.} $g_3(M)\leq 32K_2$.
\begin{proof}
First note that when $(i,j)\not\in \InI_1$, we have $s_{ij}> K_1$.  Combined with \eqref{eq-discrep-kahn-szemeredi}, this implies 
$$
r_{ij}\ln r_{ij}\leq \frac{K_2\, 2^{2j}}{d\, \beta_j}\, \ln\left(\frac{e\, 2^{2i}}{\alpha_i}\right)
$$
for any  $(i,j)\not\in \InI_1$. After an appropriate transformation, we get
\begin{equation}\label{eq-not-I1-kahn-szemeredi}
\beta_j\, s_{ij}\ln r_{ij}\leq \frac{K_2\, 2^{j}}{\sqrt{d}\, 2^{i}}\, \ln\left(\frac{e\, 2^{2i}}{\alpha_i}\right)
\end{equation} 
for any  $(i,j)\not\in \InI_1$. When $(i,j)\in \InI_3$, we have
$$
\ln r_{ij}\geq \frac{1}{4}\ln\left(\frac{e\, 2^{2i}}{\alpha_i}\right).
$$
This, together with \eqref{eq-not-I1-kahn-szemeredi}, yields
$$
\beta_j s_{ij}\leq \frac{4K_2\, 2^{j}}{\sqrt{d}\, 2^i},
$$
for any  $(i,j)\in \InI_3$. 
Thus,
$$
g_3(M)\leq \frac{4K_2}{\sqrt{d}}\sum_{i\geq 1}\alpha_i\, 2^{-i} \sum_{j:(i,j)\in\InI_3} 2^j
$$
Since $2^{j}\leq 2^i\, \sqrt{d}$ for $(i,j)\not\in \InI_2$,
the second sum is bounded by $2\cdot 2^i\, \sqrt{d}$. Hence, we have
$$
g_3(M)\leq 8K_2\sum_{i\geq 1}\alpha_i\leq 32K_2,
$$
where in the last inequality we used \eqref{eq1-lem-kahn-szemeredi}.
\end{proof}

\medskip

\noindent\textbf{Claim 4.} $g_4(M)\leq \frac{8K_2\sqrt{6e}}{K_1 \ln K_1}$.
\begin{proof}
In view of \eqref{eq-not-I1-kahn-szemeredi}, we have for any $(i,j)\in \InI_4$: 
$$
\beta_j s_{ij}\ln r_{ij}\leq 
\frac{K_2\, 2^j}{\sqrt{d}\, 2^i}\ln( e^2\, 2^{4i})
\leq \frac{K_2\sqrt{6}\, 2^j}{\sqrt{d}},
$$
where in the last inequality we used $\ln( e^2\, 2^{4i})\leq \sqrt{6}\, 2^i$.
Since $r_{ij}\geq s_{ij}> K_1$ for $(i,j)\not\in \InI_1$, the above inequality implies that 
$$
\beta_j s_{ij}
\leq \frac{K_2\sqrt{6}\, 2^j}{\sqrt{d}\, \ln K_1}
$$
for any $(i,j)\in \InI_4$. Therefore,
\begin{equation}\label{eq-g4-kahn-szemeredi}
g_4(M)\leq  \frac{K_2\sqrt{6}}{\sqrt{d}\, \ln K_1}\sum_{i\geq 1}\alpha_i \sum_{j:(i,j)\in \InI_4} 2^j.
\end{equation}
Now note that whenever $(i,j)\in \InI_4$, we have
$$
K_1< s_{ij}= \frac{\sqrt{d} r_{ij}}{2^{i+j}}\leq \frac{\sqrt{d}}{2^{i+j}} \left(\frac{e\, 2^{2i}}{\alpha_i}\right)^{\frac14}\leq \sqrt{e\, d}\, 2^{-j},
$$
which implies that $2^{j}\leq \sqrt{e\, d}/K_1$. Thus, the second sum in \eqref{eq-g4-kahn-szemeredi} 
is bounded by $2\cdot \sqrt{e\, d}/K_1$, whence
$$
g_4(M)\leq \frac{2K_2\sqrt{6e}}{K_1 \ln K_1}\sum_{i\geq 1}\alpha_i \leq 
\frac{8K_2\sqrt{6e}}{K_1 \ln K_1},
$$
where the last inequality follows from \eqref{eq1-lem-kahn-szemeredi}.
\end{proof}

\medskip

\noindent\textbf{Claim 5.} $g_5(M)\leq 16$.
\begin{proof}
First note that if $(i,j)\in \InI_5$, we have 
$$
\alpha_i< \frac{1}{e\, 2^{2i}}\quad \text{and} \quad r_{ij}\leq \left(\frac{e\, 2^{2i}}{\alpha_i}\right)^{\frac14}.
$$
Hence, for any $(i,j)\in\InI_5$ we obtain
$$
\alpha_i s_{ij}=\sqrt{d} \frac{\alpha_i}{2^{i+j}}r_{ij}\leq \sqrt{d} \frac{\alpha_i}{2^{i+j}}\left(\frac{e\, 2^{2i}}{\alpha_i}\right)^{\frac14} 
=\frac{\sqrt{\alpha_i d}}{2^{i+j}} \left(\alpha_i e\, 2^{2i}\right)^{\frac14}\leq 2\frac{\sqrt{d}}{2^{i+j}},
$$
where in the last inequality we used a crude bound $\alpha_i\leq 4$. 
Thus,
$$
g_5(M)\leq 2 \sum_{j\geq 1}\beta_j\sum_{i:(i,j)\in \InI_5} \sqrt{d}2^{-i-j}.
$$
Since the second sum is bounded by $2$, we deduce that 
$$
g_5(M)\leq 4\sum_{j\geq 1}\beta_j \leq 16,
$$
where the in last inequality we used that $\sum_{j\geq 1}\beta_j\leq 4$. 
\end{proof}

\medskip

Putting all the claims together, we get 
$$
g(M)=\sum_{(i,j)\in\InI}\alpha_i\beta_js_{ij}\leq
16K_1+24+32K_2+\frac{8K_2\sqrt{6e}}{K_1 \ln K_1}:=U(K_1,K_2).
$$
Working with the transposed matrix (and corresponding graph), we get
$$
\sum_{(i,j)\in\OutI}\alpha_i\beta_js_{ij}\leq U.
$$
Putting together the two estimates above, we complete the proof.
\end{proof}

\bigskip

\noindent {\small Konstantin Tikhomirov,}\\
{\small Department of Mathematics, Princeton University,}\\
{\small E-mail: kt12@math.princeton.edu}

\bigskip

\noindent {\small Pierre Youssef,}\\
{\small Laboratoire de Probabilit\'es et de Mod\`eles al\'eatoires,
Universit\'e Paris Diderot,}\\
{\small E-mail: youssef@math.univ-paris-diderot.fr}


\begin{thebibliography}{99}

\bibitem{Alon}
{
N. Alon, Eigenvalues and expanders, Combinatorica {\bf 6} (1986), no.~2, 83--96. MR0875835
}

\bibitem{AM}
{
N. Alon\ and\ V. D. Milman, $\lambda\sb 1,$ isoperimetric inequalities for graphs, and superconcentrators,
J. Combin. Theory Ser. B {\bf 38} (1985), no.~1, 73--88. MR0782626
}

\bibitem{BHKY}
{
R. Bauerschmidt, J. Huang, A.Knowles, and H.-T. Yau. Bulk eigenvalue statistics for random regular graphs, arXiv:1505.06700.
}

\bibitem{BKY}
{
R. Bauerschmidt, A. Knowles, H.-T. Yau,
Local semicircle law for random regular graphs,
arXiv:1503.08702.
}

\bibitem{Bennett}
{
G. Bennett, Probability Inequalities for the Sum of Independent Random Variables,
Journal of the American Statistical Association {\bf 297} (1962), 33--45, doi:10.2307/2282438.
}

\bibitem{Bernstein}
{
S. N. Bernstein, Theory of Probability (in Russian), Moscow, 1927.
}

\bibitem{B}
{
C. Bordenave,
A new proof of Friedman's second eigenvalue Theorem and its extension to random lifts,
arXiv:1502.04482.
}

\bibitem{BFSU}
{
A. Z. Broder, A. M. Frieze, S. Suen, E. Upfal, Optimal construction of edge-disjoint paths in random graphs,
SIAM J. Comput. {\bf 28} (1999), no.~2, 541--573 (electronic). MR1634360
}

\bibitem{BS}
{
A. Broder, E. Shamir, On the second eigenvalue of random regular graphs,
Proceedings of the 28th Annual Symposium on Foundations of Computer Science (1987), 286--294.
}

\bibitem{Cook RSA}
{
N. Cook, Discrepancy properties for random regular digraphs, Random Structures Algorithms,
DOI: 10.1002/rsa.20643.
}

\bibitem{Cook PTRF}
{
N. Cook, On the singularity of adjacency matrices for random regular digraphs, Prob. Theory and Related Fields,
to appear. arXiv:1411.0243.
}

\bibitem{CGJ}
{
N. Cook, L. Goldstein, T. Johnson,
Size biased couplings and the spectral gap for random regular graphs,
arXiv:1510.06013.
}

\bibitem{DJPP}
{
I. Dumitriu, T. Johnson, S. Pal, E. Paquette, Functional limit theorems for random regular graphs,
Probab. Theory Related Fields {\bf 156} (2013), no.~3-4, 921--975. MR3078290
}

\bibitem{DP}
{
I. Dumitriu\ and\ S. Pal, Sparse regular random graphs: spectral density and eigenvectors,
Ann. Probab. {\bf 40} (2012), no.~5, 2197--2235. MR3025715
}

\bibitem{Freedman}
{
D. A. Freedman, On tail probabilities for martingales, Ann. Probability {\bf 3} (1975), 100--118. MR0380971
}

\bibitem{Friedman}
{
J. Friedman, A proof of Alon's second eigenvalue conjecture and related problems,
Mem. Amer. Math. Soc. {\bf 195} (2008), no.~910, viii+100 pp. MR2437174
}

\bibitem{F91}
{
J. Friedman, On the second eigenvalue and random walks in random $d$-regular graphs,
Combinatorica {\bf 11} (1991), no.~4, 331--362. MR1137767
}

\bibitem{FKS}
{
J. Friedman, J. Kahn, E. Szemer\'edi, 
On the second eigenvalue of random regular graphs,
Proceedings of the twenty-first annual ACM symposium on Theory of computing (1989), 587--598.
}



\bibitem{LLTTY}
{
A.E. Litvak, A. Lytova, K. Tikhomirov, N. Tomczak-Jaegermann, P. Youssef,
Adjacency matrices of random digraphs: singularity and anti-concentration,  J. of Math. Analysis and Appl., 445 (2017), 1447-1491. arXiv:1511.00113.
}



\bibitem{nicole-sasha}
{
A.~E.~Litvak, A.~Pajor, M.~Rudelson, N.~Tomczak-Jaegermann, Smallest singular value of random matrices and geometry of random polytopes. Adv. Math. 195 (2005), no. 2, 491-523.
}



\bibitem{McKay}
{
B. D. McKay, The expected eigenvalue distribution of a large regular graph,
Linear Algebra Appl. {\bf 40} (1981), 203--216. MR0629617
}

\bibitem{puder}
{
D. Puder. Expansion of random graphs: New proofs, new results, Inventiones Mathematicae, 201 (3), 845-908, 2015.
}

\bibitem{RR}
{
M. M. Rao\ and\ Z. D. Ren, {\it Theory of Orlicz spaces}, Monographs and Textbooks in Pure and Applied Mathematics,
146, Dekker, New York, 1991. MR1113700
}
\bibitem{RV}
{
M.~Rudelson and  R.~Vershynin, The Littlewood-Offord problem and invertibility of random matrices. Adv. Math. 218 (2008), no. 2, 600-633.
}
\bibitem{Senior}
{
J. K. Senior, Partitions and their representative graphs, Amer. J. Math. {\bf 73} (1951), 663--689. MR0042678
}

\bibitem{TY_short}
{
K. Tikhomirov and P. Youssef, On the norm of a random jointly exchangeable matrix, arXiv:1610.01751.
}

\bibitem{TVW}
{
L. V. Tran, V. H. Vu\ and\ K. Wang, Sparse random graphs: eigenvalues and eigenvectors,
Random Structures Algorithms {\bf 42} (2013), no.~1, 110--134. MR2999215
}

\bibitem{Vu}
{
V. Vu, Random discrete matrices, in {\it Horizons of combinatorics}, 257--280, Bolyai Soc. Math. Stud., 17, Springer, Berlin. MR2432537
}

\bibitem{Vu ICM}
{
V. Vu, Combinatorial problems in random matrix theory,
Proceedings ICM, Vol. 4, 2014, 489--508.
}

\bibitem{Wormald}
{
N. C. Wormald, Models of random regular graphs, in {\it Surveys in combinatorics, 1999 (Canterbury)},
239--298, London Math. Soc. Lecture Note Ser., 267, Cambridge Univ. Press, Cambridge. MR1725006
}

\end{thebibliography}
\end{document}